\documentclass[11pt,reqno]{amsart}

\usepackage{amssymb,amsmath,amsthm,amsfonts}
\usepackage{hyperref, multirow, bm, color, marginnote, graphicx, wrapfig, subcaption}
\usepackage{diagbox}

\makeatletter
\renewcommand{\tocsection}[3]{%
  \indentlabel{\@ifnotempty{#2}{\bfseries\ignorespaces#1 #2\quad}}\bfseries#3}
\renewcommand{\tocsubsection}[3]{%
  \indentlabel{\@ifnotempty{#2}{\ignorespaces#1 #2\quad}}#3}

\newcommand\@dotsep{4.5}
\def\@tocline#1#2#3#4#5#6#7{\relax
  \ifnum #1>\c@tocdepth 
  \else
    \par \addpenalty\@secpenalty\addvspace{#2}%
    \begingroup \hyphenpenalty\@M
    \@ifempty{#4}{%
      \@tempdima\csname r@tocindent\number#1\endcsname\relax
    }{%
      \@tempdima#4\relax
    }%
    \parindent\z@ \leftskip#3\relax \advance\leftskip\@tempdima\relax
    \rightskip\@pnumwidth plus1em \parfillskip-\@pnumwidth
    #5\leavevmode\hskip-\@tempdima{#6}\nobreak
    \leaders\hbox{$\m@th\mkern \@dotsep mu\hbox{.}\mkern \@dotsep mu$}\hfill
    \nobreak
    \hbox to\@pnumwidth{\@tocpagenum{\ifnum#1=1\bfseries\fi#7}}\par
    \nobreak
    \endgroup
  \fi}
\AtBeginDocument{%
\expandafter\renewcommand\csname r@tocindent0\endcsname{0pt}
}
\def\l@subsection{\@tocline{2}{0pt}{2.5pc}{5pc}{}}
\makeatother

\usepackage[margin=1in]{geometry}

\newcommand{\paren}[1]{\left(#1\right)}

\newcommand{\R}{\mathbb{R}}

\newcommand{\bars}{\overline s}

\newcommand{\X}{\bm{X}}
\newcommand{\be}{\bm{e}}

\newcommand{\p}{\partial}

\newcommand{\abs}[1]{\left\lvert #1 \right\rvert}
\newcommand{\norm}[1]{\left\lVert #1 \right\rVert}
\newcommand{\wh}[1]{\widehat{#1}}
\newcommand{\wt}[1]{\widetilde{#1}}
\newcommand{\mc}[1]{\mathcal{#1}}

\newtheorem{theorem}{Theorem}[section]
\newtheorem{lemma}[theorem]{Lemma}

\newtheorem{corollary}[theorem]{Corollary}

\theoremstyle{definition}
\newtheorem{remark}[theorem]{Remark}
\newtheorem{obs}[theorem]{Observations}

\allowdisplaybreaks

\graphicspath{{figures/}}

\begin{document}
\title[Classical elastohydrodynamics for a swimming filament]{Well-posedness and applications of classical elastohydrodynamics for a swimming filament}

\author{Yoichiro Mori}
\address[Y. Mori]{Department of Mathematics, University of Pennsylvania, Philadelphia, PA 19104}
\email{y1mori@sas.upenn.edu}

\author{Laurel Ohm}
\address[L. Ohm]{Department of Mathematics, Princeton University, Princeton, NJ 08540}
\email[Corresponding author]{laurel.ohm@princeton.edu}

\begin{abstract} 
We consider a classical elastohydrodynamic model of an inextensible filament undergoing planar motion in $\R^3$. The hydrodynamics are described by resistive force theory, and the fiber elasticity is governed by Euler-Bernoulli beam theory.   
Our aim is twofold: (1) Serve as a starting point for developing the mathematical analysis of filament elastohydrodynamics, particularly the analytical treatment of an inextensibility constraint, and (2) As an application, prove conditions on internal fiber forcing that allow a free-ended filament to swim. Our analysis of fiber swimming speed is supplemented with a numerical optimization of the internal fiber forcing, as well as a novel numerical method for simulating an inextensible swimmer.
\end{abstract}

\maketitle

\tableofcontents

\setlength{\parskip}{6pt}


\section{Introduction}

We consider the planar motion of an inextensible, free-ended filament immersed in a Stokes fluid in $\R^3$. 
The hydrodynamic effects on the filament are described using resistive force theory (or local slender body theory), one of the most fundamental tools for modeling flagellar locomotion in viscous fluids \cite{gray1955propulsion, johnson1979flagellar, keller1976swimming, pironneau1974optimal}. The structure of the filament itself is described by the classical Euler-Bernoulli beam theory, coupled with internal forcing corresponding to a (possibly nonzero) preferred curvature \cite{hines1978bend, wiggins1998flexive, wiggins1998trapping, camalet1999self, camalet2000generic, riedel2007molecular, hilfinger2009nonlinear, sartori2016dynamic, tornberg2004simulating}.
Resistive force theory coupled with the elastic response of the fiber has been shown to be remarkably effective at capturing actual filament dynamics, as demonstrated by \cite{friedrich2010high, yu2006experimental}.

\emph{Aim 1: Analysis of inextensibility.}
A key feature of this model is the inextensibility constraint which prohibits the fiber from growing or shrinking in length over time. Inextensibility features prominently in related models of, for example, vesicle dynamics \cite{veerapaneni2011fast, veerapaneni2009boundary, rahimian2010dynamic}, but its treatment is relatively underdeveloped from the perspective of mathematical analysis, particularly in the dynamical setting. 

We begin by considering the relaxation of the filament with zero preferred curvature and show global-in-time existence and uniqueness of solutions for small initial data, as well as local well-posedness for large initial data. Global well-posedness for large data is hindered by the behavior of the filament tension. However, we show that if the filament's initial bending energy is small, then a solution exists globally. As a byproduct, we obtain nonlinear stability of the straight filament. The filament evolution in this setting may be regarded as a simpler version of the Peskin problem (immersed boundary method), which has been the subject of many recent PDE works \cite{gancedo2020global, garcia2020peskin, lin2019solvability, mori2019well, tong2021regularized}. This work also follows a recent program by the authors to place nonlocal slender body theory on firm theoretical footing \cite{closed_loop, free_ends, rigid, inverse, regularized}. These previous papers have all considered the static boundary value problem for Stokes flow about a curve at a single instant in time. The form of the integral operator in nonlocal slender body renders the dynamic problem for the curve evolution much more difficult. Using resistive force theory for the force-to-velocity map along the fiber simplifies the curve evolution problem while still remaining relevant from a current modeling perspective (see, for example, \cite{allende2018stretching, chakrabarti2021signatures, coy2017counterbend, du2019dynamics, liu2021theoretical, nguyen2014hydrodynamics, pozveh2021resistive, stein2021swirling, waszkiewicz2021stability, young2012dynamics}).

\emph{Aim 2: Analysis of swimming and applications.}
Resistive force theory coupled with fiber elastodynamics has been especially useful for understanding undulatory swimming at low Reynolds number \cite{el2020optimal, gadelha2010nonlinear, gadelha2019flagellar, hu2022enhanced, montenegro2015spermatozoa, lauga2013shape, lauga2007floppy,  spagnolie2010optimal, lauga2009hydrodynamics}. This brings us to our second aim: understanding conditions on a filament's internal forcing which allow it to swim. On the analysis side, we consider the filament dynamics under small-amplitude periodic forcing, corresponding to a nonzero preferred curvature along the fiber. We show existence of a unique periodic solution and derive an expression for the swimming speed of the filament. From the swimming expression, we obtain conditions on the periodic forcing that give rise to propulsion. Our swimming analysis may also be compared to other analysis works on swimming, including \cite{galdi2020self, martin2008initial, galdi2009motion, nevcasova2011weak}.  

The swimming expression we obtain allows us to perform a small numerical optimization to determine the internal fiber forcing which gives rise to the fastest swimming for a fixed amount of work and bending energy. The optimal forcing that we find is a type of traveling wave; notably, we do not \emph{a priori} assume any conditions on the forcing other than time periodicity. Our optimization can be compared with related optimizations of the actual filament shape \cite{lighthill1975mathematical,spagnolie2010optimal,lauga2013shape,neal2020doing,pironneau1974optimal}, although we emphasize that our optimization is for the active forcing along the filament, and the resulting fiber deformation is an emergent property.

Finally, we develop a novel numerical method based on two recent approaches to inextensibility, due to Moreau et al. \cite{moreau2018asymptotic} and Maxian et al. \cite{maxian2021integral}, which rely on different but related methods of avoiding the need to solve for the filament tension. 
These reformulations of filament elastohydrodynamics are interesting from an analysis perspective because, as seen in the well-posedness results of this paper, estimates for the filament tension are a limiting factor in the solution theory which hinder us in showing global existence for large initial data. A weaker notion of inextensibility --  perhaps related to these reformulations -- may overcome this difficulty. Moreover, a broader aim of this paper is to serve as a starting point for a full numerical analysis of the various formulations of fiber inextensibility used in numerical simulations, ranging from penalization methods \cite{tornberg2004simulating} to variants of the more recent approaches which avoid solving for the fiber tension altogether \cite{moreau2018asymptotic, hall2019efficient, walker2020efficient, walker2021regularised, maxian2021integral, maxian2022hydrodynamics}. Our own numerical method is validated against a direct implementation of the classical formulation of filament elastohydrodynamics. We use the method to verify numerically the observations from our swimming analysis.

\subsection{Setup and statement of results}
We consider an inextensible elastic filament in $\R^3$ with centerline $\X:[0,L]\times[0,T]\to \R^3$ parameterized by arclength $s$. Throughout, we will use the subscript $\cdot_s$ to denote differentiation $\frac{\p}{\p_s}\cdot$ with respect to arclength. The filament is assumed to undergo planar deformations only; in particular, we can define a unique in-plane unit normal vector $\be_{\rm n}(s,t)$ along the filament centerline. In the presence of active forcing in the form of a preferred curvature $\kappa_0(s,t)$, the classical resistive force theory formulation of the evolution of $\X(s,t)$ is given by
\begin{equation}\label{classical0}
\begin{aligned}
\frac{\p\X}{\p t}(s,t) &= \big(c_{\perp}{\bf I}+(c_{\parallel}-c_{\perp})\X_s\X_s^{\rm T}\big)\big(E(\X_{sss}-(\kappa_0)_s\be_{\rm n})-\tau(s,t)\X_s\big)_s \\
\abs{\X_s}^2&=1 
\end{aligned}
\end{equation}
along with the force-free and torque-free boundary conditions $\big(E(\X_{sss}-(\kappa_0)_s\be_{\rm n})-\tau\X_s\big)\big|_{s=0,L}=0$, $E(\X_{ss}-\kappa_0\be_{\rm n})\big|_{s=0,L}=0$. 
For a filament of radius $\varepsilon$, the constants $c_{\perp}\approx \frac{\log(\varepsilon/L)}{4\pi}$ and $c_{\parallel}\approx \frac{\log(\varepsilon/L)}{2\pi}$ for $\varepsilon/L\ll 1$; in particular, $c_{\parallel}\approx 2c_{\perp} <0$. (Note that some sources define these constants as $1/c_{\perp}$ and $1/c_{\parallel}$; see, for example, \cite[Chapter 6.3]{lauga2020fluid}).
The hydrodynamic force on the right hand side consists of an elastic term with constant bending stiffness $E>0$, the unknown tension $\tau(s,t)$ which serves as a Lagrange multiplier to enforce $\abs{\X_s}^2=1$, and active forcing due to a time-varying preferred curvature $\kappa_0(s,t)$ along the filament (see \cite{camalet2000generic,thomases2017role} or the derivation in Section \ref{sec:periodic}). We will consider both the evolution of a passive filament ($\kappa_0(s,t)\equiv0$) and an active filament with time-periodic forcing. 

Rescaling arclength as $s/L$, $\X$ as $\X/L$, tension as $\frac{L^2}{E}\tau$, forcing as $L\kappa_0$, time as $\frac{E\abs{c_{\perp}}}{L^4}t$, and defining $\gamma=\frac{c_{\parallel}}{c_{\perp}}-1$, we may rewrite \eqref{classical0} as 
\begin{equation}\label{classical}
\begin{aligned}
\frac{\p\X}{\p t}(s,t) &= -\big({\bf I}+\gamma\X_s\X_s^{\rm T}\big)\big(\X_{sss}-\tau(s,t)\X_s-(\kappa_0)_s\be_{\rm n}\big)_s \\
\abs{\X_s}^2&=1 \\
(\X_{ss}-\kappa_0\be_{\rm n})\big|_{s=0,1}&=0\,, \quad (\X_{sss}-\tau\X_s-(\kappa_0)_s\be_{\rm n})\big|_{s=0,1}=0\,.
\end{aligned}
\end{equation}
Throughout we will use the notation $I=[0,1]$ to denote the unit interval. We also note that the constant $\gamma\approx 1 >0$.  
Since the filament is planar and inextensible, the tangent vector $\be_{\rm t}(s,t)=\X_s(s,t)$ at each point may be determined with respect to the angle $\theta(s,t)$ between $\be_{\rm t}(s,t)$ and a fixed axis, which we take to be $\be_{\rm t}(0,0)$. 
In particular, we may write 
\begin{align*}
\X_s &= \begin{pmatrix}
\cos\theta\\
\sin\theta
\end{pmatrix} = \be_{\rm t}\,, \quad \X_{ss}= \begin{pmatrix}
-\sin\theta\\
\cos\theta
\end{pmatrix} \theta_s = \theta_s \be_{\rm n}\, .
\end{align*}

Differentiating the formulation \eqref{classical} with respect to $s$ gives the evolution of the tangent vector $\dot\be_{\rm t}=\dot\X_s$. The normal and tangential components of this evolution yield two equations, one for the evolution of the tangent angle $\theta$ and one for the tension $\tau$ at each time: 
\begin{align}
\dot\theta &=-\theta_{ssss}+(2+\gamma)(\theta_s^3)_s +(2+\gamma)\tau_s\theta_s+\tau\theta_{ss} + (\kappa_0)_{sss}-(1+\gamma)\theta_s^2(\kappa_0)_s  \label{thetaEQ0} \\
(1+\gamma)\tau_{ss}-(\theta_s)^2\tau &= (\theta_s)^4+\theta_{ss}^2-(4+3\gamma)(\theta_{ss}\theta_s)_s + (2+\gamma)(\kappa_0)_{ss}\theta_s +(1+\gamma)\theta_{ss}(\kappa_0)_s  \label{tauEQ0} \\
 (\theta_s-\kappa_0)\big|_{s=0,1}&=0\,,\quad (\theta_{ss}-(\kappa_0)_s)\big|_{s=0,1}=0\,, \quad (\tau+\kappa_0^2)\big|_{s=0,1}=0\,. \label{BCs0}
\end{align}

Note that the boundary conditions for $\tau$ are due to the torque-free condition $\X_{ss}\big|_{s=0,1}=\theta_s\be_{\rm n}\big|_{s=0,1}=\kappa_0\be_{\rm n}\big|_{s=0,1}$ along with the form of $\X_{sss}=-\theta_s^2\be_{\rm t}+\theta_{ss}\be_{\rm n}$ in the force-free condition in \eqref{classical}.
The tangent angle formulation \eqref{thetaEQ0}-\eqref{tauEQ0}, with or without a forcing $\kappa_0$, is a popular choice for numerical simulations \cite{camalet1999self, camalet2000generic, hilfinger2009nonlinear, hines1978bend, hu2022enhanced, moreau2018asymptotic, sartori2016dynamic, wiggins1998flexive, yu2006experimental}. Given the boundary conditions \eqref{BCs0} and the fact that only derivatives of $\theta$ appear on the right hand side of \eqref{thetaEQ0} and in the tension equation \eqref{tauEQ0}, we find it more convenient to work with the evolution of the filament curvature $\kappa=\theta_s$ rather than $\theta$, as is done in \cite{goldstein1995nonlinear,thomases2017role}. In particular, given an initial tangent angle $\theta\big|_{t=0}=\theta_{\rm in}(s)$, once the curvature evolution $\kappa(s,t)$ is known we may uniquely recover $\theta(s,t)$ by
\begin{equation}\label{thetadot}
\dot\theta = -\kappa_{sss}+(2+\gamma)(\kappa^3)_{ss} +(2+\gamma)\tau_s\kappa +\tau\kappa_s + (\kappa_0)_{sss}-(1+\gamma)\kappa^2(\kappa_0)_s \, .
\end{equation}

Furthermore, it will then be convenient to consider the evolution of the difference $\overline\kappa=\kappa-\kappa_0$ as well as $\overline\tau=\tau+\kappa_0^2$. Using equations \eqref{thetaEQ0}--\eqref{BCs0}, we have that $\overline\kappa$ and $\overline\tau$ satisfy the equations  
\begin{align}
\dot{\overline\kappa} &=-\overline\kappa_{ssss}-\dot\kappa_0 +\big[3(2+\gamma)\overline\kappa(\overline\kappa+2\kappa_0)\overline\kappa_s+(5+3\gamma)\kappa_0^2\overline\kappa_s+ (5+2\gamma)\overline\kappa^2(\kappa_0)_s  \nonumber \\
&\hspace{3cm}
+2(3+\gamma)\overline\kappa\kappa_0(\kappa_0)_s +(2+\gamma)\overline\tau_s(\overline\kappa+\kappa_0)+\overline\tau(\overline\kappa+\kappa_0)_s \big]_s  \label{theta_eqn}\\
\overline\tau_{ss}-\frac{(\overline\kappa+\kappa_0)^2}{1+\gamma}\overline\tau &=\frac{1}{1+\gamma}\big[\overline\kappa(\overline\kappa+\kappa_0)^2(\overline\kappa+2\kappa_0) + (\overline\kappa+\kappa_0)_s\overline\kappa_s \nonumber\\
&\hspace{2cm}-(1+\gamma)\big(\overline\kappa(\overline\kappa+2\kappa_0)\big)_{ss} -(2+\gamma)\big(\overline\kappa_s(\overline\kappa+\kappa_0)\big)_s  \big] \label{T_eqn} \\
%
%
 \overline\kappa\big|_{s=0,1}&=0\,,\quad \overline\kappa_s\big|_{s=0,1}=0\,, \quad \overline\tau\big|_{s=0,1}=0\,. \label{BCs}
\end{align}


The formulation \eqref{theta_eqn}--\eqref{BCs} serves as the basis of our analysis. We will begin by considering the passive filament $\kappa_0(s,t)\equiv 0$, for which the dynamics reduce to the simpler system 
\begin{align}
\dot\kappa &=-\kappa_{ssss}+(2+\gamma)(\kappa^3)_{ss} +(2+\gamma)(\tau_s\kappa)_s+(\tau\kappa_s)_s  \label{noforce_eqn1}\\
\tau_{ss}-\frac{\kappa^2}{1+\gamma}\tau &=\frac{1}{1+\gamma}\bigg(\kappa^4 + (\kappa_s)^2-(4+3\gamma)(\kappa_s\kappa)_s  \bigg) \label{noforce_eqn2} \\
 \kappa\big|_{s=0,1}&=0\,,\quad \kappa_s\big|_{s=0,1}=0\,, \quad \tau\big|_{s=0,1}=0\,. \label{noforce_BCs}
\end{align}

We define the operator $\mc{L}$ to be the linear evolution of \eqref{noforce_eqn1} near $\kappa\equiv 0$: 
\begin{equation}\label{L_op}
 \mc{L}[\psi]:= \p_{ssss}\psi\,, \qquad \psi(0)=\psi(1)=0\,, \; \psi_s(0)=\psi_s(1)=0\,.
\end{equation}

Since the remaining nonlinear terms of \eqref{noforce_eqn1} are a perfect derivative in $s$, we write the remainder as $(\mc{R}[\kappa])_s$, where
\begin{equation}\label{remainder}
\mc{R}[\kappa]= 3(2+\gamma)\kappa^2\kappa_s+ (2+\gamma)\tau_s\kappa + \tau\kappa_s \,.
\end{equation}

Letting $\kappa_{\rm in}(s)=\kappa(s,0)$ denote the initial curvature of the fiber, we may define a mild solution $\kappa(s,t)$ to \eqref{noforce_eqn1} by the Duhamel formula 
\begin{equation}\label{theta_mild}
\kappa(s,t) = e^{-t\mc{L}}\kappa_{\rm in}(s) + \int_0^t e^{-(t-t')\mc{L}}(\mc{R}[\kappa(s,t')])_s \, dt'\,.
\end{equation}

The fiber curvature in \eqref{noforce_eqn1}--\eqref{noforce_BCs} obeys the energy identity  
\begin{equation}\label{Theta_energy}
\frac{1}{2}\frac{d}{dt}\int_0^1\kappa^2 \, ds + \int_0^1\left(\kappa_{ss}-\kappa^3-\tau\kappa \right)^2\, ds + \int_0^1(1+\gamma)\left(3\kappa\kappa_s + \tau_s \right)^2\,ds =0\,.
\end{equation}
In particular, $\norm{\kappa}_{L^2}$ is nonincreasing in time.
The identity \eqref{Theta_energy} is more easily seen at the level of \eqref{classical}, where sufficiently regular curves $\X(s,t)$ satisfy
\begin{equation}\label{energy}
\frac{1}{2}\frac{d}{dt}\int_0^1\abs{\X_{ss}}^2\,ds = -\int_0^1\bigg((\X_{sss}-\tau\X_s)_s^2+\gamma\big(\X_s\cdot(\X_{sss}-\tau\X_s)_s\big)^2\bigg)\,ds \le 0\,. 
\end{equation}

This may be shown by multiplying both sides of \eqref{classical} by $(\X_{sss}-\tau\X_s)_s$ and integrating in $s$. Integration by parts on the left hand side yields 
\begin{align*}
\int_0^1\frac{\p\X}{\p t}\cdot(\X_{sss}-\tau\X_s)_s \, ds &= -\int_0^1\dot\X_s(s,t)\cdot(\X_{sss}-\tau\X_s) \, ds \\
&= -\int_0^1\dot\X_s\cdot\X_{sss} \, ds = \int_0^1\dot\X_{ss}\cdot\X_{ss}\,ds \\
&= \frac{1}{2}\frac{d}{dt}\int_0^1\abs{\X_{ss}}^2\,,
\end{align*}
where we have used the inextensibility constraint $\dot\X_s\cdot\X_s=0$.  \\

The system \eqref{noforce_eqn1}--\eqref{noforce_BCs} is invariant under the scaling 
\begin{align*}
(\kappa(s,t),\tau(s,t))\mapsto (\lambda\kappa(\lambda s,\lambda^4 t)\,, \lambda^2\tau(\lambda s,\lambda^4 t))\,.
\end{align*}
Norms which are invariant under the above scaling symmetry are known as critical norms and play an important role in the well-posedness theory. Norms which decay upon `zooming in' via the above scaling are known as subcritical. The natural expectation is that the PDE is (locally) well-posed in subcritical spaces and solutions belonging to subcritical cases are in fact smooth. Typically, if we have \emph{a priori} control on a subcritical norm for all time, then the solution should be globally well-posed. This is the case for the 1D viscous Burgers' equation, for example. In our case, any $H^s$ norm with $s>-1/2$ is subcritical. Furthermore, by \eqref{Theta_energy} we have that $\norm{\kappa}_{L^2}$ is monotone decreasing. However, we have found the regularity theory to be much more subtle due to the behavior of the tension $\tau$ (see Remark \ref{rmk} and Section \ref{subsec:tension}).\footnote{In particular, the local solution to \eqref{noforce_eqn1}--\eqref{noforce_BCs} described in Theorem \ref{thm:existence} may be extended up to some maximal time $T^*(\kappa_{\rm in})$ depending on the initial data. If the energy identity \eqref{Theta_energy} can be used (in conjunction with a Gr\"onwall argument) to obtain a maximal existence time which depends only on the norm $\norm{\kappa_{\rm in}}_{L^2}$ rather than the full initial data, or if it can be shown that, given $T$, the local solution $\kappa$ remains in $C([0,T];L^2(I))$ up to time $T$, then the solution can be extended globally. 
The main difficulty in proving these arises in bounding the tension $\tau$ (see e.g. bound \eqref{T_H1bound}), which seems to necessarily depend on $\norm{\kappa}_{\dot H^1}^2$ -- i.e. a rather high power of a high regularity norm for $\kappa$. }  
This makes the system \eqref{noforce_eqn1}--\eqref{noforce_BCs} somewhat surprising from a regularity point of view. Nevertheless, the energy inequality \eqref{Theta_energy} motivates our choice of function space for the following well-posedness result.

\begin{theorem}[Well-posedness]\label{thm:existence}
\emph{Local-in-time for large data:} For any $\kappa_{\rm in}\in L^2(I)$ there exists a time $0<T_0<1$ such that the system \eqref{noforce_eqn1}--\eqref{noforce_BCs} admits a unique mild solution $\kappa\in C([0,T_0];L^2(I))\cap C((0,T_0];H^1(I))$.

\emph{Global-in-time for small data:}
There exists a constant $0<c_0<1$ such that, given $\kappa_{\rm in}\in L^2(I)$ with $\norm{\kappa_{\rm in}}_{L^2(I)} \le c_0$, the system \eqref{noforce_eqn1}--\eqref{noforce_BCs} admits a unique mild solution $\kappa\in C([0,T];L^2(I))\cap C((0,T];H^1(I))$ for any $T>0$. The solution $\kappa$ satisfies
\begin{equation}\label{exp_bound}
\norm{\kappa}_{L^2(I)}+ \min\{t^{1/4},1\}\norm{\kappa}_{\dot H^1(I)} \le c e^{-t\lambda_1}\norm{\kappa_{\rm in}}_{L^2(I)}\,.
\end{equation}
Here $\lambda_1$ is the first eigenvalue of the linearized operator $\mc{L}$ (see \eqref{L_efuns}). 
\end{theorem}

The time weight in \eqref{exp_bound} allows us to consider $\kappa_{\rm in}\in L^2(I)$ rather than $H^1(I)$. The norm $\norm{\kappa_{\rm in}}_{L^2(I)}$ has physical significance, as it is exactly the initial bending energy of the filament. The proof of Theorem \ref{thm:existence} is given in Section \ref{sec:wellpo}. In principle, one can further show that $\kappa(s,t)$ belongs to $C^\infty(I)$ for all $t>0$ using a bootstrapping argument, although this requires a bit of additional care (see Remark \ref{rem:bootstrap}).

Noting that $\kappa\equiv 0$ corresponds to a straight filament, as an immediate corollary of the exponentially decaying bound \eqref{exp_bound}, we obtain:
\begin{corollary}[Stability of straight fiber]\label{cor:straight_stab}
The straight filament $\kappa\equiv 0$ is nonlinearly stable to small perturbations.
\end{corollary}

We now turn to the filament with active internal forcing $\kappa_0\neq0$ in \eqref{theta_eqn}-\eqref{BCs}. We take $\kappa_0(s,t)$ to be time-periodic with period $T>0$ to simulate a waving flagellum. Note that existence and uniqueness of solutions to \eqref{theta_eqn}--\eqref{BCs} with sufficiently small $\kappa_0$ and $\overline\kappa_{\rm in}$ follow by similar arguments to Theorem \ref{thm:existence} (see Lemma \ref{lem:wellpo_force}). Moreover, we show the existence of a unique $T$-periodic solution $\overline\kappa$: 

\begin{theorem}[Existence of unique periodic solution]\label{thm:periodic}
There exist constants $0<c_0<1$ and $c_1\ge0$ such that, given a $T$-periodic preferred curvature $\kappa_0(s,t)\in C^1([0,T]; H^1(I))$ satisfying 
\begin{align*}
T\bigg(\sup_{0\le t\le T}\norm{\dot\kappa_0}_{L^2(I)}\bigg)= \varepsilon\le c_0\,, \quad \sup_{0\le t\le T}\norm{\kappa_0}_{H^1}=c_1\varepsilon\le c_0\,, 
\end{align*}
 there exists a unique $T$-periodic mild solution $\overline\kappa\in C([0,T];H^1(I))$ to the system \eqref{theta_eqn}--\eqref{BCs} satisfying the bound
 \begin{equation}\label{kappaH1}
  \sup_{0\le t\le T}\norm{\overline\kappa}_{H^1(I)} \le c(T)\varepsilon\,.
 \end{equation}
 Furthermore, there exists $c_2\ge 0$ such that, for any $\overline\varphi_{\rm in}\in L^2(I)$ satisfying $\norm{\overline\varphi_{\rm in}}_{L^2}\le c_2\varepsilon$, the corresponding unique solution to \eqref{theta_eqn}-\eqref{BCs} satisfies
 \begin{equation}\label{per_conv}
 \norm{\overline\kappa(\cdot,nT)-\overline\varphi(\cdot,nT)}_{L^2(I)} \le ce^{-nT\lambda_1}\norm{\overline\kappa_{\rm in}-\overline\varphi_{\rm in}}_{L^2(I)}, \quad n=1,2,\dots \,, 
 \end{equation} 
 where $\lambda_1$ is the first eigenvalue of the linearized operator $\mc{L}$ (see \eqref{L_efuns}). 
\end{theorem}
The proof of Theorem \ref{thm:periodic} appears in Section \ref{subsec:unique_per}.
Note that by \eqref{per_conv}, nearby solutions converge very rapidly to the periodic solution. We consider this periodic solution in all further analysis and use it to calculate the fiber swimming speed. To do so, we first must recover the equations for the actual translational motion of the fiber.  

Given the unique periodic solution $(\overline\kappa,\overline\tau)=(\kappa-\kappa_0,\tau+\kappa_0^2)$ to \eqref{theta_eqn}--\eqref{BCs} and the initial tangent angle $\theta_{\rm in}$, we may uniquely calculate $\theta$ via the equation \eqref{thetadot} for $\dot\theta$. From \eqref{thetadot} we may also solve for the frame evolution: 
\begin{equation}\label{frame_ev}
\dot\be_{\rm t}(s,t) = \dot\theta(s,t) \be_{\rm n}(s,t)\,, \qquad \dot\be_{\rm n}(s,t)=-\dot\theta(s,t) \be_{\rm t}(s,t)\,.
\end{equation}

We may then obtain the full evolution of the fiber via 
\begin{equation}\label{fiber_move}
\frac{\p\X}{\p t}(s,t)
= \big(-(\kappa-\kappa_0)_{ss} +\kappa^3 +\tau\kappa\big) \be_{\rm n} +(1+\gamma)\big(\tau_s+ 2\kappa\kappa_s +\kappa(\kappa-\kappa_0)_s \big)\be_{\rm t}  \,. 
\end{equation}

The swimming velocity of the fiber $\bm{V}(t)=\int_0^1\frac{\p\X}{\p t}(s,t)\,ds$ may be calculated as
\begin{equation}\label{basepoint}
\bm{V}(t) = -\gamma\int_0^1 \bigg(\big(\overline\kappa\overline\kappa_s-\overline\tau_s-\kappa_0\overline\kappa_s\big)\be_{\rm t} +2\overline\kappa\kappa^2\be_{\rm n}\bigg) \, ds \, .
\end{equation}
(See Section \ref{subsec:swim}, equation \eqref{swim_calc} for the full calculation). 

For suitably small $\kappa_0$, we seek an expression for the average swimming speed of the fiber in the direction $\be_{\rm t}(0,0)$ over the course of one period $0\le t\le T$. In what follows, for any function $h(t)$, we will use the notation
\begin{equation}
\langle h \rangle := \frac{1}{T}\int_0^T h(t)\, dt
\end{equation}
to denote the time average over one period. We show the following.

\begin{theorem}[Small amplitude swimming]\label{thm:swimming}
Suppose that $\kappa_0(s,t)\in C^1([0,T]; H^3(I))$ is $T$-periodic in time and satisfies 
\begin{align*}
T\bigg(\sup_{0\le t\le T}\norm{\dot\kappa_0}_{L^2(I)}\bigg)= \varepsilon\le c_0\,, \quad \sup_{0\le t\le T}\norm{\kappa_0}_{H^1}=c_1\varepsilon\le c_0\,. 
\end{align*} 
Then a filament whose curvature evolution satisfies \eqref{theta_eqn}-\eqref{BCs} swims with velocity 
\begin{equation}
\bm{V}(t)=U(t)\, \be_{\rm t}(0,0) + \bm{r}_{\rm v}(t)\,,
\end{equation}
where $\sup_{0\le t\le T}\abs{U}\le c\varepsilon^2$ and $\sup_{0\le t\le T}\abs{\bm{r}_{\rm v}}\le c\varepsilon^3$, and $U$ is given by
\begin{equation}\label{swim_speed1}
U(t) = -\gamma\int_0^1 (\kappa_0)_s\overline\kappa \, ds\,.
\end{equation}

Writing $\kappa_0(s,t)=\sum_{m,k=1}^\infty\big(a_{m,k}\cos(\omega m t)-b_{m,k}\sin(\omega m t)\big)\psi_k(s)$ where $\psi_k(s)$ are the eigenfunctions of the linearized operator $\mc{L}$ (see \eqref{L_efuns}) and $\omega=\frac{2\pi}{T}$, we have that the time-averaged speed $\langle U\rangle$ may be written
\begin{equation}\label{swim_speed2}
\begin{aligned}
\langle U\rangle&= \frac{\gamma}{2}\sum_{m,k,\ell=1}^\infty\frac{\omega^2 m^2}{\omega^2m^2+\lambda_k^2} \bigg(\frac{\lambda_k}{\omega m}\big(a_{m,k}b_{m,\ell}-b_{m,k}a_{m,\ell} \big) \\
&\hspace{5cm}+ a_{m,k}a_{m,\ell}+b_{m,k}b_{m,\ell} \bigg) \int_0^1\psi_k(\psi_\ell)_s\,ds \, + r_{\rm u}\,,
\end{aligned}
\end{equation}
where $\lambda_k$ are the eigenvalues of $\mc{L}$ and $\abs{r_{\rm u}}\le c\varepsilon^3$.
\end{theorem}
The proof of Theorem \ref{thm:swimming} is given in Section \ref{subsec:swim}. We note that the additional $H^3(I)$ regularity on $\kappa_0$ is used to ensure that the fiber frame $(\be_{\rm t},\be_{\rm n})$ is changing very little over time (see equation \eqref{frame_ev}). As in previous works (see, e.g., \cite{lauga2020fluid}), we obtain that an $O(\varepsilon)$ forcing results in propulsion at most at $O(\varepsilon^2)$. 
From \eqref{swim_speed1}, we see that propulsion occurs due to a lag between the actual and preferred curvature of the fiber, which in turn depends on the rate of change of the preferred curvature via \eqref{theta_eqn}.
However, from \eqref{swim_speed2}, we see that further conditions on $\kappa_0(s,t)$ are needed to actually generate net movement over the course of one full period. Observations about the swimming speed \eqref{swim_speed2} are made more precise in Section \ref{sec:obs}, but may be summarized as follows:

\begin{obs}\label{obs:obs} 
\begin{enumerate}
\item Forcing only the lowest nonzero temporal frequency appears to yield the furthest filament displacement for a given amount of work. Thus in further applications we consider the filament's preferred curvature to be of the form
\begin{equation}\label{kap0_form}
\kappa_0(s,t)=F_1(s)\cos(\omega\,t)+F_2(s)\sin(\omega\,t)\,.
\end{equation}
\item If $F_1$ and $F_2$ are both even or both odd about $s=\frac{1}{2}$, the integral $\int_0^1\psi_k(\psi_\ell)_s\,ds$ in \eqref{swim_speed2} vanishes and the filament does not swim to leading order. \\

\item If either $F_1=0$ or $F_2=0$, or if $F_1=\pm F_2$, then the  term proportional to $\big(a_{m,k}b_{m,\ell}-b_{m,k}a_{m,\ell} \big)$ in \eqref{swim_speed2} vanishes. The filament may still swim at leading order, but its displacement will be very small, due to the size of $\lambda_k$ (see \eqref{L_efuns}).
\end{enumerate}
\end{obs}

Observation (1), which was also noted in \cite{lauga2007floppy}, comes from numerically solving a simple constrained optimization problem (see \eqref{optimization}). The requirements on $F_1$ and $F_2$ may be considered in terms of the Scallop Theorem for a swimmer in a linear, time-independent flow: violating observation (2) can be seen numerically to lead to a time-reversible filament deformation (see Section \ref{subsec:swim_num}), while violating (3) gives rise to a nearly time-reversible fiber deformation. 

A waveform fitting the above criteria is a traveling wave propagating along the filament; for example, $F_1=\sin(\omega s)$ and $F_2=\cos(\omega s)$. This is consistent with classical calculations by Pironneau and Katz \cite{pironneau1974optimal} which find that a filament should deform as a traveling wave to reach a given swimming speed while minimizing energy expenditure.

A further small numerical optimization problem (see \eqref{optimization2}) where $\norm{\kappa_0}_{L^2}$ is constrained as well yields a $\kappa_0$ \eqref{kap0_form} in the form of a traveling wave which outperforms the simple prescription $F_1=\sin(\omega s)$ and $F_2=\cos(\omega s)$. A notable feature of this optimal $\kappa_0$ is that $F_1$ and $F_2$ look like $\sin(\omega s)$ and $\cos(\omega s)$ away from the filament ends, but $F_1=F_2=0$ and $F_1'=F_2'=0$ at $s=0,1$. This is reminiscent of the results of \cite{neal2020doing}, which find that an inactive flagellar endpiece enhances swimming movement.

The above observations are explored in more detail in Section \ref{sec:obs}, and are further supplemented with numerical simulations in Section \ref{sec:numerics}. We implement a reformulated version of equation \eqref{classical} using a novel combination of the methods proposed by Moreau et al. \cite{moreau2018asymptotic} and Maxian et al. \cite{maxian2021integral}, which we validate against a direct discretization of \eqref{classical} (see Appendix \ref{app:num}). Using the observations surrounding the form \eqref{kap0_form} of $\kappa_0$, we consider combinations of $F_1(s)$ and $F_2(s)$ that give rise to non-swimmers, bad swimmers, and good swimmers, and create visualizations of their respective positions over time.

\section{Well-posedness: local for large data, global for small data}\label{sec:wellpo}

\subsection{Preliminaries}
We begin by stating a few useful lemmas. 
 
\begin{lemma}[Gagliardo--Nirenberg inequality]\label{lem:GN}
For any function $u\in H^1(I)$, the following inequality holds for each $2\le p\le\infty$:
\begin{align*}
\norm{u}_{L^p(I)}&\le c\norm{u}_{H^1(I)}^{\frac{p-2}{2p}}\norm{u}_{L^2(I)}^{\frac{p+2}{2p}} \,.
\end{align*}
 \end{lemma} 
If $u\in H^1_0(I)$, then we may replace $\norm{u}_{H^1(I)}$ in the above estimate with $\norm{u}_{\dot H^1(I)}$.

In addition to Lemma \ref{lem:GN}, we will need estimates for the semigroup generated by the operator $\mc{L}$ \eqref{L_op}. We first state some properties of the well-studied Dirichlet problem for the bilaplace equation in 1D
(see \cite{landau1986theory, wiggins1998flexive, wiggins1998trapping}). In particular, we have explicit expressions for the eigenfunctions and eigenvalues of $\mc{L}$. First, define an increasing sequence $\{\xi_k\}$ satisfying
\begin{equation}\label{eigenvals}
\cos(\xi_k)\cosh(\xi_k)=1\,, \qquad \xi_0=0\,.
\end{equation}
The first few values of $\xi_k$ are given by $\xi_1\approx 4.730$, $\xi_2\approx7.853$, $\xi_3\approx10.996$.
Note that since $\cosh(\xi)$ increases exponentially as $\abs{\xi}$ increases, we have $\xi_k\to\frac{(2k+1)\pi}{2}$ as $k\to\infty$. The eigenfunctions $\psi_k$ and eigenvalues $\lambda_k$ of $\mc{L}$ are given by 
\begin{equation}\label{L_efuns}
\begin{aligned}
\psi_k(s) &= \frac{\wh\psi_k(s)}{\|\wh\psi_k\|_{L^2}}\,, \qquad \lambda_k = \xi_k^4\,, \qquad k=1,2,\dots\\
\text{where } \wh\psi_k(s) &= \left(\cos(\xi_k)-\cosh(\xi_k)\right) \left(\cos(\xi_k s)- \cosh(\xi_k s) \right) \\
&\qquad + \left(\sin(\xi_k) +\sinh(\xi_k)\right) \left(\sin(\xi_k s)-\sinh(\xi_k s)\right)\, .
\end{aligned}
\end{equation}
We note in particular that the smallest eigenvalue of $\mc{L}$ is given by $\lambda_1\approx(4.73)^4\approx 500$. The first six eigenfunctions $\psi_k$ are plotted in Figure \ref{fig:eigenfunctions}. Note that from \eqref{L_efuns}, for odd $k$, $\psi_k(s)$ is even about $s=\frac{1}{2}$, while for even $k$, $\psi_k(s)$ is odd about $s=\frac{1}{2}$. 

\begin{figure}[!ht]
\centering
\includegraphics[scale=0.2]{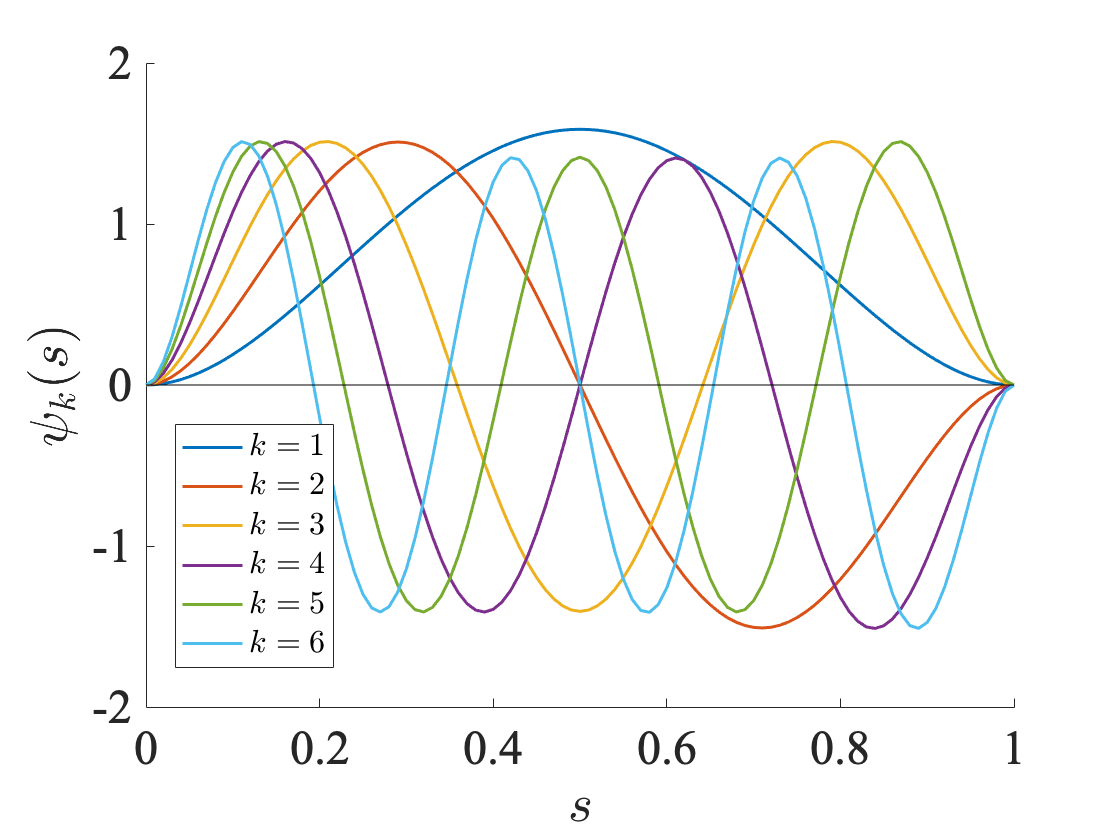}
\caption{The first six eigenfunctions of $\mc{L}$.}
\label{fig:eigenfunctions}
\end{figure}

Using \eqref{L_efuns}, we may define the domain of $\mc{L}$ by
\begin{equation}\label{domainL}
D(\mc{L}) = \bigg\{ u\in L^2(I) \;:\; u =\sum_k a_k\psi_k, \; a_k = \int_0^1u(s)\psi_k(s)ds; \; \sum_k\lambda_k^2a_k^2 <\infty \bigg\}\,.
\end{equation}

Furthermore, following \cite[Chapter 3.7]{sell2002dynamics}, we may define the domain of fractional powers $\mc{L}^r$, $0\le r\le1$, by  
\begin{equation}\label{domainLr}
D(\mc{L}^r) = \bigg\{ u\in L^2(I) \;:\; u =\sum_k a_k\psi_k, \; a_k = \int_0^1u(s)\psi_k(s)ds; \; \sum_k\lambda_k^{2r}a_k^2 <\infty \bigg\}\,.
\end{equation}

Noting the quartic growth of $\lambda_k=\xi_k^4$ as $k\to\infty$, we have the following inclusions:
\begin{equation}\label{inclusions}
D(\mc{L}^r)\subseteq H^{4r}(I)\,, \quad 0\le r \le 1\,; \quad D(\mc{L}^0)= L^2(I)\,.
\end{equation}

We will repeatedly use the following lemma.
\begin{lemma}[Semigroup estimates]\label{lem:smoothing}
For $0\le m+j\le 4$ we have
\begin{equation}
\norm{e^{-t\mc{L}}(\p_s^j u)}_{\dot H^m} \le c \max\{t^{-(m+j)/4},1\}e^{-t\lambda_1}\norm{u}_{L^2}\,,
\end{equation}
where $\lambda_1=\xi_1^4$ is the first eigenvalue of the operator $\mc{L}$ \eqref{L_efuns}. 
\end{lemma}

\begin{proof}
For $u\in L^2(I)$, write $u=\sum_{k=1}^\infty a_k \psi_k(s)$ where $a_k = \int_0^1 u(s)\psi_k(s)\,ds$. For $0\le r\le 1$, we have 
\begin{align*}
\norm{\mc{L}^r e^{-t\mc{L}}u}_{L^2(I)} &= \bigg\|\sum_{k=1}^\infty \lambda_k^r e^{-t\lambda_k} a_k \psi_k\bigg\|_{L^2} \le \sup_{\lambda_k}\bigg(\lambda_k^{r}e^{-t(\lambda_k-\lambda_1)}\bigg)e^{-t\lambda_1} \bigg\|\sum_{k=1}^\infty a_k \psi_k\bigg\|_{L^2} \\
&\le  c_r \max\{t^{-r},1\} e^{-t\lambda_1}\norm{u}_{L^2(I)}\,.
\end{align*}
Using \eqref{inclusions}, we then have 
\begin{equation}\label{semi_est1}
\norm{e^{-t\mc{L}}u}_{\dot H^{4r}} \le c_r \max\{t^{-r},1\}e^{-t\lambda_1}\norm{u}_{L^2(I)}\,, \qquad 0\le r\le 1\,.
\end{equation}

It remains to show the smoothing estimate, for which we use a duality argument.
For $u,w\in C_c^\infty(I)$, we have 
\begin{align*}
\sup_{\norm{u}_{L^2}=1}\norm{e^{-t\mc{L}}(\p_s^j u) }_{L^2} &= \sup_{\norm{u}_{L^2}=1}\sup_{\norm{w}_{L^2}=1}\big(w,e^{-t\mc{L}}\p_s^j u \big)_{L^2} \\
&=\sup_{\norm{u}_{L^2}=1}\sup_{\norm{w}_{L^2}=1}\big(\p_s^j e^{-t\mc{L}}w,u \big)_{L^2} 
\end{align*}
since $\mc{L}$ is self-adjoint. Using \eqref{semi_est1}, we then have 
\begin{align*}
\sup_{\norm{u}_{L^2}=1}\norm{e^{-t\mc{L}}(\p_s^j u) }_{L^2} \le \sup_{\norm{u}_{L^2}=1}\sup_{\norm{w}_{L^2}=1}\norm{\p_s^j e^{-t\mc{L}}w}_{L^2}\norm{u}_{L^2}  
\le c_j \max\{t^{-j/4},1\}e^{-t\lambda_1}\,.
\end{align*}
The same result holds for $u\in L^2(I)$ by density.
\end{proof}

Finally, we will require the following lemma to show local existence for large initial data. 
\begin{lemma}\label{lem:init_dat}
Fix $u\in L^2(I)$ and let $0<r\le 1$, $\varepsilon>0$. There exists $T_\varepsilon>0$ which depends on $u$ and satisfies
\begin{equation}\label{smalltimelem}
\sup_{t\in[0,T_\varepsilon]}\, \min\{t^{r},1\} e^{t\lambda_1}\norm{e^{-t\mc{L}}u}_{\dot H^{4r}(I)} \le \varepsilon\,.
\end{equation}
\end{lemma}

The smallness in Lemma \ref{lem:init_dat} is achieved by approximating $u\in L^2(I)$ by slightly smoother functions. 
\begin{proof}
Consider a sequence of functions $\phi_j\in D(\mc{L}^r)$ with $\phi_j\to u$ in $L^2$. Since $\phi_j\in D(\mc{L}^r)$, by \eqref{domainLr} we have 
\begin{align*}
\norm{e^{-t\mc{L}}\phi_j}_{\dot H^{4r}} \le \norm{\mc{L}^re^{-t\mc{L}}\phi_j}_{L^2} \le \norm{\sum_{k=1}^\infty\lambda_k^r e^{-t\lambda_k}(\phi_j,\psi_k)\psi_k}_{L^2}
\le e^{-t\lambda_1}\bigg(\sum_{k=1}^\infty\lambda_k^{2r} (\phi_j,\psi_k)^2\bigg)^{1/2} \le e^{-t\lambda_1}c_j
\end{align*}
for some constant $0<c_j<\infty$. Then, using Lemma \ref{lem:smoothing}, we have
\begin{align*}
\min\{t^{r},1\} e^{t\lambda_1}\norm{e^{-t\mc{L}}u}_{\dot H^{4r}} &\le \min\{t^{r},1\} e^{t\lambda_1}\norm{e^{-t\mc{L}}\phi_j}_{\dot H^{4r}}+ \min\{t^{r},1\} e^{t\lambda_1}\norm{e^{-t\mc{L}}(u-\phi_j)}_{\dot H^{4r}}  \\
&\le \min\{t^{r},1\}\,c_j+ c\norm{u-\phi_j}_{L^2}\,.
\end{align*}
Taking $j$ large enough that $c\norm{u-\phi_j}_{L^2}\le \frac{\varepsilon}{2}$ and taking $t$ small enough that $t^{r}\,c_j\le \frac{\varepsilon}{2}$, we obtain the bound \eqref{smalltimelem}.
\end{proof}

With Lemmas \ref{lem:GN}, \ref{lem:smoothing}, and \ref{lem:init_dat}, we may now turn to the proof of Theorem \ref{thm:existence}: well-posedness for the unforced system \eqref{noforce_eqn1}--\eqref{noforce_BCs}. First, a remark:

\begin{remark}\label{rmk}
To highlight the subtlety of the well-posedness theory for the system \eqref{noforce_eqn1}--\eqref{noforce_BCs}, we may consider the following reduced PDE system consisting of the most problematic terms of \eqref{noforce_eqn1}--\eqref{noforce_BCs}:
\begin{align*}
\dot\kappa =-\kappa_{ssss} + (\tau\kappa_s)_s\,, \quad \tau_{ss}=\kappa_s^2\,.
\end{align*}
We may rewrite this simplified system as a single equation:
\begin{align*}
\dot\kappa =-\kappa_{ssss} + \big(\p_{ss}^{-1}(\kappa_s^2)\kappa_s\big)_s\,,
\end{align*}
where $\p_{ss}^{-1}$ is the inverse second derivative with homogeneous Dirichlet conditions at $s=0,1$. This equation is invariant under the same scaling symmetry as the full system: $\kappa(s,t)\mapsto \lambda\kappa(\lambda s,\lambda^4t)$. Again, $\norm{\kappa}_{L^2}$ is subcritical for this equation. However, to actually make sense of the tension $\tau$, we need $\kappa_s\in L^2$. Taking initial data $\kappa_{\rm in}\in L^2(I)$ only, we expect that $\sup_t \,t^{1/4}\norm{\kappa_s}_{L^2(I)}<\infty$ due to the smoothing effect of $\mc{L}$. Thus every $\kappa_s$ which appears in the nonlinear term `costs' $t^{-1/4}$. From Lemma \ref{lem:smoothing}, using the smoothing estimate on the semigroup to account for the outside derivative in the nonlinear term also costs $t^{-1/4}$. The nonlinearity then contains three additional powers of $\norm{\kappa_s}_{L^2}$, costing a total of $t^{-3/4}$. Altogether, we must integrate $t^{-1}$ in the Duhamel formula, so the full power of the smoothing of $\mc{L}$ is required to close the contraction.
\end{remark}

\subsection{Tension equation}\label{subsec:tension}
We first consider the elliptic equation \eqref{noforce_eqn2} and derive the following estimates for the tension $\tau$ in terms of $\kappa$.

\begin{lemma}\label{lem:T_lemma}
Given $\kappa\in H^1(I)$, there exists a unique weak solution $\tau$ to equation \eqref{noforce_eqn2} satisfying
\begin{equation}\label{T_H1bound}
 \norm{\tau}_{H^1(I)} \le c\norm{\kappa}_{\dot H^1(I)}^2\left(\norm{\kappa}_{L^2(I)}+1\right)\,.
 \end{equation}

Given $\kappa,\varphi\in H^1(I)$, define $\tau^\kappa$, $\tau^\varphi$ to be the $H^1_0$ solution to \eqref{noforce_eqn2} using $\kappa$, $\varphi$ respectively on the right hand side. The difference $\tau^\kappa-\tau^\varphi$ then satisfies
\begin{equation}\label{wtT_bound}
\begin{aligned}
\|\tau^\kappa-\tau^\varphi\|_{H^1(I)}
&\le 
c\norm{\kappa-\varphi}_{L^2}\bigg(\norm{\kappa}_{\dot H^1}^2(\norm{\kappa}_{L^2}^2+1) +\norm{\varphi}_{\dot H^1}^2(\norm{\varphi}_{L^2}^2+1)\bigg) \\
&\qquad+ c\norm{\kappa-\varphi}_{\dot H^1}\big(\norm{\kappa}_{\dot H^1}+\norm{\varphi}_{\dot H^1}\big)\,.
\end{aligned}
\end{equation}
\end{lemma}
Note that the right hand sides of both \eqref{T_H1bound} and \eqref{wtT_bound} involve two powers of an $\dot H^1$ norm, and it is not clear that a lower power is possible. This factor of $\norm{\kappa}_{\dot H^1}^2$ presents an obstacle to showing global existence for large data, as demonstrated by Remark \ref{rmk}. 

\begin{proof}
Given $\kappa\in H^1(I)$, for $\phi\in H^1_0(I)$, we may define 
\begin{equation}\label{T_bilinear}
\mc{B}(\tau,\phi) := \int_0^1\bigg(\tau_{s}\phi_s+\frac{\kappa^2}{1+\gamma}\tau\phi \bigg)\,ds\,, 
\end{equation}
which, due to the sign of the potential term (note that $\kappa\in L^\infty(I)$) is always bounded and coercive on $H^1_0(I)$. We may define a weak solution of the tension equation as $\tau$ satisfying
\begin{equation}\label{T_weaksol}
\mc{B}(\tau,\phi) = -\int_0^1\frac{1}{1+\gamma}\bigg(\kappa^4\phi + (\kappa_s)^2\phi +(4+3\gamma)\kappa_s\kappa\phi_s \bigg) \,ds \quad \text{for all }\phi\in H^1_0(I)\,.
\end{equation}
By the Lax--Milgram lemma, there exists a unique weak solution $\tau\in H^1_0(I)$, and, using Lemma \ref{lem:GN}, $\tau$ satisfies 
\begin{align*}
\mc{B}(\tau,\tau) &\le c\bigg(\norm{\kappa}_{L^6}^3\norm{\kappa\tau}_{L^2}+\norm{\kappa_s}_{L^2}^2\norm{\tau}_{L^\infty} + \norm{\kappa_s}_{L^2}\norm{\tau_s}_{L^2}\norm{\kappa}_{L^\infty}\bigg) \\
&\le c\bigg(\norm{\kappa}_{\dot H^1}\norm{\kappa}_{L^2}^2\norm{\kappa\tau}_{L^2}+\norm{\kappa}_{\dot H^1}^2\norm{\tau}_{H^1} + \norm{\kappa}_{\dot H^1}^{3/2}\norm{\kappa}_{L^2}^{1/2}\norm{\tau}_{H^1}\bigg)\,.
\end{align*}
Note that, in terms of regularity, the limiting factor here is the term $\kappa_s^2\tau$, which requires 2 powers of $\norm{\kappa}_{\dot H^1}$ to estimate, while the additional regularity of $\tau$ is wasted.  

Using Young's inequality, we then have 
 \begin{equation}\label{BTT_est}
\mc{B}(\tau,\tau) 
\le c\bigg(\norm{\kappa}_{\dot H^1}^2\norm{\kappa}_{L^2}^4+\norm{\kappa}_{\dot H^1}^4 + \norm{\kappa}_{\dot H^1}^3\norm{\kappa}_{L^2}\bigg)\,,
 \end{equation}
and therefore, using that $\norm{\tau}_{H^1}^2\le c\norm{\tau_s}_{L^2}^2\le c\mc{B}(\tau,\tau)$, we obtain the bound \eqref{T_H1bound}.\\

We next show the Lipschitz estimate \eqref{wtT_bound} for $\tau$. Let $\tau^\kappa$, $\tau^\varphi$ be as defined in Lemma \ref{lem:T_lemma}. The difference $\wh \tau= \tau^\kappa-\tau^\varphi$ then satisfies
 \begin{align*}
 \wh \tau_{ss} &= \frac{\kappa^2}{1+\gamma}\tau^\kappa-\frac{\varphi^2}{1+\gamma}\tau^\varphi +\frac{1}{1+\gamma}\bigg(\kappa^4-\varphi^4 + (\kappa_s)^2-(\varphi_s)^2-\frac{4+3\gamma}{2}\big(\kappa^2-\varphi^2\big)_{ss}  \bigg) \\
 &= \frac{1}{2(1+\gamma)}\bigg((\kappa^2+\varphi^2)\wh \tau+(\kappa-\varphi)(\kappa+\varphi)(\tau^\kappa+\tau^\varphi)\bigg) \\
 &\quad + \frac{1}{1+\gamma}\bigg((\kappa-\varphi)(\kappa+\varphi)(\kappa^2+\varphi^2) + (\kappa_s-\varphi_s)(\kappa_s+\varphi_s)-\frac{4+3\gamma}{2}\big((\kappa-\varphi)(\kappa+\varphi)\big)_{ss}  \bigg)\,.
 \end{align*}

 Upon multiplying by $\wh \tau$ and integrating by parts in $s$, we obtain 
 \begin{align*}
 & \int_0^1 \bigg(\wh \tau_{s}^2+\frac{\kappa^2+\varphi^2}{2(1+\gamma)}\wh \tau^2\bigg)ds \\
 &\le 
 c\bigg(\norm{\kappa-\varphi}_{L^2}\norm{\kappa+\varphi}_{L^2}\|\tau^\kappa+\tau^\varphi\|_{L^\infty}\bigg)\|\wh \tau\|_{L^\infty} \\
 &\quad +c\bigg(\norm{\kappa-\varphi}_{L^\infty}\norm{\kappa+\varphi}_{L^\infty}\|(\kappa^2+\varphi^2)\wh \tau\|_{L^2} + \norm{\kappa_s-\varphi_s}_{L^2}\norm{\kappa_s+\varphi_s}_{L^2}\|\wh \tau\|_{L^\infty}\bigg)\\
 &\quad + c\bigg(\norm{\kappa-\varphi}_{L^\infty}\norm{\kappa_s+\varphi_s}_{L^2}+ \norm{\kappa_s-\varphi_s}_{L^2}\norm{\kappa+\varphi}_{L^\infty}\bigg)\|\wh \tau_s\|_{L^2}  \\
 &\le 
c\norm{\kappa-\varphi}_{L^2}(\norm{\kappa}_{L^2}+\norm{\varphi}_{L^2})\bigg(\norm{\kappa}_{\dot H^1}^2(\norm{\kappa}_{L^2}+1)+\norm{\varphi}_{\dot H^1}^2(\norm{\varphi}_{L^2}+1)\bigg)\|\wh \tau\|_{L^\infty}  \\
 &\quad +c\norm{\kappa-\varphi}_{\dot H^1}^{1/2}\norm{\kappa-\varphi}_{L^2}^{1/2}(\norm{\kappa}_{\dot H^1}^{1/2}\norm{\kappa}_{L^2}^{1/2}+\norm{\varphi}_{\dot H^1}^{1/2}\norm{\varphi}_{L^2}^{1/2})\|(\kappa^2+\varphi^2)\wh \tau\|_{L^2} \\
 &\quad+ c\norm{\kappa-\varphi}_{\dot H^1}(\norm{\kappa}_{\dot H^1}+\norm{\varphi}_{\dot H^1})\|\wh \tau\|_{H^1}\,.
 \end{align*}

 Here we have used Lemma \ref{lem:GN} as well as the bound \eqref{T_H1bound}. Using Young's inequality, we then have  
\begin{align*}
\|\wh \tau\|_{H^1}^2 &\le \int_0^1 \bigg(\wh \tau_s^2+\frac{\kappa^2+\varphi^2}{2(1+\gamma)}\wh \tau^2\bigg)ds \\
&\le 
c\norm{\kappa-\varphi}_{L^2}^2(\norm{\kappa}_{L^2}+\norm{\varphi}_{L^2})^2\bigg(\norm{\kappa}_{\dot H^1}^2(\norm{\kappa}_{L^2}+1)+\norm{\varphi}_{\dot H^1}^2(\norm{\varphi}_{L^2}+1)\bigg)^2 \\
 &\quad + c\norm{\kappa-\varphi}_{\dot H^1}^2\big(\norm{\kappa}_{\dot H^1}+\norm{\varphi}_{\dot H^1}\big)^2\,. 
\end{align*}

In particular, we have
\begin{align*}
\|\tau^\kappa-\tau^\varphi\|_{H^1(I)}&\le 
c\norm{\kappa-\varphi}_{L^2}(\norm{\kappa}_{L^2}+\norm{\varphi}_{L^2})\bigg(\norm{\kappa}_{\dot H^1}^2(\norm{\kappa}_{L^2}+1)+\norm{\varphi}_{\dot H^1}^2(\norm{\varphi}_{L^2}+1)\bigg) \\
&\qquad+ c\norm{\kappa-\varphi}_{\dot H^1}\big(\norm{\kappa}_{\dot H^1}+\norm{\varphi}_{\dot H^1}\big)\,,
\end{align*}
from which we obtain the bound \eqref{wtT_bound}.
\end{proof}

\subsection{Evolution equation}
We now consider the evolution equation \eqref{noforce_eqn1} for the curvature $\kappa(s,t)$. To prove the existence of a unique mild solution \eqref{theta_mild} in both the large data and small data settings, we will rely on a contraction mapping argument that makes use of the semigroup estimates of Lemma \ref{lem:smoothing}. We will also require Lemma \ref{lem:init_dat} in the large data setting. 

\begin{proof}[Proof of Theorem \ref{thm:existence}]
We begin by defining the function spaces $\mc{Y}_0$, $\mc{Y}_1$ as
\begin{align*}
\mc{Y}_0 &= \big\{ u\in C([0,T];L^2(I))\,:\, \norm{u}_{\mc{Y}_0}<\infty \big\}\,, \\
\mc{Y}_1 &= \big\{ u\in C((0,T];\dot H^1(I))\,:\, \norm{u}_{\mc{Y}_1}<\infty \big\} \,,
\end{align*}
where the norms $\norm{\cdot}_{\mc{Y}_0}$, $\norm{\cdot}_{\mc{Y}_1}$, respectively, are given by 
\begin{align*}
\norm{\cdot}_{\mc{Y}_0}:= \sup_{t\in [0,T]} e^{t\lambda_1}\norm{\cdot}_{L^2(I)}\,, \qquad \norm{\cdot}_{\mc{Y}_1}:=\sup_{t\in [0,T]} \min\{t^{1/4},1\}\,e^{t\lambda_1}\norm{\cdot}_{\dot H^1(I)} \,. 
\end{align*}

We will close a contraction mapping argument in $\mc{Y}_0\cap\mc{Y}_1$. We consider the intersection of two closed balls in $\mc{Y}_0$ and $\mc{Y}_1$:  
\begin{align*}
B_{M_0}(\mc{Y}_0)\cap B_{M_1}(\mc{Y}_1) = \big\{ u\in \mc{Y}_0\cap\mc{Y}_1\,:\, \norm{u}_{\mc{Y}_0}\le M_0,\; \norm{u}_{\mc{Y}_1}\le M_1 \big\}\,.
\end{align*}
Given either sufficiently small $T$ or sufficiently small initial data, we show that the mild solution formula \eqref{theta_mild} maps $B_{M_0}(\mc{Y}_0)\cap B_{M_1}(\mc{Y}_1)$ into itself for some choice of $M_0>0$ and $0<M_1<1$. 

For any $\varphi\in B_{M_0}(\mc{Y}_0)\cap B_{M_1}(\mc{Y}_1)$, let $\Psi$ denote the mapping
\begin{align*}
\Psi[\varphi]:= e^{-t\mc{L}}\kappa_{\rm in} + \int_0^t e^{-(t-t')\mc{L}}\big(\mc{R}[\varphi(s,t')]\big)_s dt' \,.
\end{align*}
We start by estimating the nonlinear terms of \eqref{theta_mild}. First, using \eqref{T_H1bound} and Lemma \ref{lem:GN}, we have
\begin{equation}\label{Rest_L2}
\begin{aligned}
\norm{\mc{R}[\kappa(\cdot,t)]}_{L^2} &\le c\bigg(\norm{\kappa}_{L^\infty}^2\norm{\kappa_s}_{L^2} +\norm{\tau}_{L^2}\norm{\kappa}_{L^\infty} +\norm{\tau}_{L^\infty}\norm{\kappa_s}_{L^2} \bigg)\\
&\le c\bigg(\norm{\kappa}_{L^2}\norm{\kappa}_{\dot H^1}^2 +\norm{\kappa}_{\dot H^1}^{5/2}(\norm{\kappa}_{L^2}^{3/2}+\norm{\kappa}_{L^2}^{1/2}) +\norm{\kappa}_{\dot H^1}^3(\norm{\kappa}_{L^2}+1) \bigg) \\
&\le c\norm{\kappa}_{\dot H^1}^3(\norm{\kappa}_{L^2}+1)\,.
\end{aligned}
\end{equation}
Here we note that the term $\tau\kappa_s$ from \eqref{remainder} is the limiting factor in terms of regularity (and note that writing $\tau_s\kappa+\tau\kappa_s=(\tau\kappa)_s$ and using the smoothing of the semigroup in Lemma \ref{lem:smoothing} will still lead to the same result). In particular, since we only have the $H^1$ bound \eqref{T_H1bound} for $\tau$, the regularity of $\tau$ in this estimate ends up getting wasted, and we end up with 3 powers of $\norm{\kappa}_{\dot H^1}$ on the right hand side of \eqref{Rest_L2}. 

Using Lemma \ref{lem:smoothing} and \eqref{Rest_L2}, we then have
\begin{align*}
&\norm{\int_0^t e^{-(t-t')\mc{L}}\big(\mc{R}[\kappa(\cdot,t')]\big)_s\, dt'}_{\dot H^1(I)}  \le c\int_0^t \max\{(t-t')^{-1/2},1\}e^{-(t-t')\lambda_1}\norm{\mc{R}[\kappa]}_{L^2(I)} dt' \\
&\hspace{1cm}\le c\int_0^t \max\{(t-t')^{-1/2},1\}e^{-(t-t')\lambda_1}\norm{\kappa}_{\dot H^1}^3(\norm{\kappa}_{L^2}+1) dt'  \\
&\hspace{1cm}\le c\int_0^t \max\{(t-t')^{-1/2},1\} e^{-(t-t')\lambda_1}\left(e^{-3t'\lambda_1}\max\{(t')^{-3/4},1\}\norm{\kappa}_{\mc{Y}_1}^3\left( e^{-t'\lambda_1}\norm{\kappa}_{\mc{Y}_0} +1 \right) \right)\,dt'  \\
&\hspace{1cm}\le ce^{-t\lambda_1}\int_0^t \max\{(t-t')^{-1/2},1\}\max\{(t')^{-3/4},1\} \,e^{-2t'\lambda_1}\,dt' \, \left(\norm{\kappa}_{\mc{Y}_1}^3\left( \norm{\kappa}_{\mc{Y}_0} +1 \right) \right)  \\
&\hspace{1cm}\le c\,\max\{t^{-1/4},1\}e^{-t\lambda_1}\left(\norm{\kappa}_{\mc{Y}_1}^3\left( \norm{\kappa}_{\mc{Y}_0} +1 \right) \right)\,. 
\end{align*} 
By an analogous series of estimates, we also have
\begin{align*}
\norm{\int_0^t e^{-(t-t')\mc{L}}\big(\mc{R}[\kappa(\cdot,t')]\big)_s\, dt'}_{L^2(I)} 
&\le c e^{-t\lambda_1}\left(\norm{\kappa}_{\mc{Y}_1}^3\left( \norm{\kappa}_{\mc{Y}_0} +1 \right) \right)\,.
\end{align*}

Then, using Lemma \ref{lem:smoothing}, for $\varphi\in B_{M_0}(\mc{Y}_0)\cap B_{M_1}(\mc{Y}_1)$ we obtain the following $\mc{Y}_0$ bound:
\begin{equation}\label{BM_Y0bound}
\begin{aligned}
\norm{\Psi[\varphi]}_{\mc{Y}_0} &\le \sup_{t\in[0,T]}e^{t\lambda_1}\norm{e^{-t\mc{L}}\kappa_{\rm in}}_{L^2}+ c\norm{\varphi}_{\mc{Y}_1}^3\left( \norm{\varphi}_{\mc{Y}_0} +1 \right)  \\
&\le \norm{\kappa_{\rm in}}_{L^2(I)}+cM_1^3(M_0+1) \le M_0\,,
\end{aligned}
\end{equation}
where we have taken $M_0=2\norm{\kappa_{\rm in}}_{L^2(I)}$ and $M_1$ small enough that $cM_1^3(M_0+1)\le \frac{M_0}{2}$.

Furthermore, to obtain a $\mc{Y}_1$ bound, we use that 
\begin{equation}\label{BM_Y1bound}
\begin{aligned}
\norm{\Psi[\varphi]}_{\mc{Y}_1} &\le \sup_{t\in[0,T]}\,\min\{t^{1/4},1\}\,e^{t\lambda_1}\norm{e^{-t\mc{L}}\kappa_{\rm in}}_{\dot H^1}+ c\norm{\varphi}_{\mc{Y}_1}^3\left( \norm{\varphi}_{\mc{Y}_0} +1 \right)  \\
&\le \sup_{t\in[0,T]}\,\min\{t^{1/4},1\}\,e^{t\lambda_1}\norm{e^{-t\mc{L}}\kappa_{\rm in}}_{\dot H^1}+cM_1^3(M_0+1) \\
&\le \sup_{t\in[0,T]}\,\min\{t^{1/4},1\}\,e^{t\lambda_1}\norm{e^{-t\mc{L}}\kappa_{\rm in}}_{\dot H^1}+\frac{M_1}{2} 
\end{aligned}
\end{equation}
for $M_1$ sufficiently small. It remains to show that 
\begin{align*}
\sup_{t\in[0,T]}\,\min\{t^{1/4},1\}\,e^{t\lambda_1}\norm{e^{-t\mc{L}}\kappa_{\rm in}}_{\dot H^1} \le \frac{M_1}{2}\,,
\end{align*}
for which we have two options:
\begin{enumerate}
\item \emph{Small time:} Using Lemma \ref{lem:init_dat}, we may choose a time $T_{M_1}=T_{M_1}(\kappa_{\rm in})$ such that 
\begin{align*}
\sup_{t\in[0,T_{M_1}]}\,\min\{t^{1/4},1\}\,e^{t\lambda_1}\norm{e^{-t\mc{L}}\kappa_{\rm in}}_{\dot H^1} \le \frac{M_1}{2}\,.
\end{align*}
Note that $T_{M_1}$ depends on $\kappa_{\rm in}$, particularly on how well $\kappa_{\rm in}\in L^2(I)$ can be approximated by slightly smoother functions. \\
\item \emph{Small data:} Using Lemma \ref{lem:smoothing}, we have 
\begin{align*}
\sup_{t\in[0,T]}\,\min\{t^{1/4},1\}\,e^{t\lambda_1}\norm{e^{-t\mc{L}}\kappa_{\rm in}}_{\dot H^1} \le c\norm{\kappa_{\rm in}}_{L^2} \,.
\end{align*}
For $\norm{\kappa_{\rm in}}_{L^2}$ sufficiently small, we may take $M_1=2c\norm{\kappa_{\rm in}}_{L^2}$ to obtain the desired bound.  
\end{enumerate} 

We now show that $\Psi$ is a contraction on $B_{M_0}(\mc{Y}_0)\cap B_{M_1}(\mc{Y}_1)$. For $\kappa,\varphi\in H^1(I)$ we first note that 
\begin{equation}\label{R_lip}
\begin{aligned}
&\norm{\mc{R}[\kappa(\cdot,t)]-\mc{R}[\varphi(\cdot,t)]}_{L^2(I)}\\
 &\le c\left(\norm{\kappa^2\kappa_s-\varphi^2\varphi_s}_{L^2}
+\norm{\tau^\kappa_s\kappa-\tau^\varphi_s\varphi}_{L^2}+\norm{\tau^\kappa\kappa_s-\tau^\varphi\varphi_s}_{L^2}\right) \\
&\le c\bigg(\norm{(\kappa^2+\varphi^2-\varphi\kappa)(\kappa_s-\varphi_s)}_{L^2} +\norm{(\kappa-\varphi)(\kappa\varphi_s+\varphi\kappa_s)}_{L^2} + \|(\tau^\kappa-\tau^\varphi)_s(\kappa+\varphi)\|_{L^2}\\
&\quad  + \|(\tau^\kappa+\tau^\varphi)_s(\kappa-\varphi)\|_{L^2} + \|(\tau^\kappa-\tau^\varphi)(\kappa_s+\varphi_s)\|_{L^2} + \|(\tau^\kappa+\tau^\varphi)(\kappa_s-\varphi_s)\|_{L^2} \bigg) \\
&\le c\bigg(\left(\norm{\kappa}_{L^\infty}^2+\norm{\varphi}_{L^\infty}^2\right)\norm{\kappa_s-\varphi_s}_{L^2} +\norm{\kappa-\varphi}_{L^\infty}\big(\norm{\kappa}_{L^\infty}\norm{\varphi_s}_{L^2}+\norm{\varphi}_{L^\infty}\norm{\kappa_s}_{L^2} \big)  \\
&\quad + \|(\tau^\kappa-\tau^\varphi)_s\|_{L^2}(\norm{\kappa}_{L^\infty}+\norm{\varphi}_{L^\infty}) + (\|\tau^\kappa_s\|_{L^2}+\|\tau^\varphi_s\|_{L^2})\norm{\kappa-\varphi}_{L^\infty}\\
&\quad + \|\tau^\kappa-\tau^\varphi\|_{L^\infty}(\norm{\kappa_s}_{L^2}+\norm{\varphi_s}_{L^2}) + (\|\tau^\kappa\|_{L^\infty}+\|\tau^\varphi\|_{L^\infty})\norm{\kappa_s-\varphi_s}_{L^2} \bigg)\\
&\le c\bigg( \norm{\kappa-\varphi}_{L^2}\left(\norm{\kappa}_{\dot H^1}^3(\norm{\kappa}_{L^2}^2+1) +\norm{\varphi}_{\dot H^1}^3(\norm{\varphi}_{L^2}^2+1)\right) \\
&\quad  + \norm{\kappa-\varphi}_{\dot H^1} \left(\norm{\kappa}_{\dot H^1}^2\left(\norm{\kappa}_{L^2}+1\right)+\norm{\varphi}_{\dot H^1}^2\left(\norm{\varphi}_{L^2}+1\right)\right)\bigg) \,.
\end{aligned}
\end{equation}
Here we have used Lemma \ref{lem:GN} in the third inequality, and in the fourth we used the estimates \eqref{wtT_bound} and \eqref{T_H1bound} to bound $\tau^\kappa$, $\tau^\varphi$ in terms of $\kappa$, $\varphi$.

Using \eqref{R_lip}, for $\kappa,\varphi\in B_{M_0}(\mc{Y}_0)\cap B_{M_1}(\mc{Y}_1)$ we have 
\begin{align*}
&\norm{\Psi[\kappa]-\Psi[\varphi]}_{\mc{Y}_1}\\
 &\quad \le \sup_{t\in[0,T]}\min\{t^{1/4},1\}\,e^{t\lambda_1}\int_0^t \norm{e^{-(t-t')\mc{L}}\big(\mc{R}[\kappa]-\mc{R}[\varphi]\big)_s}_{\dot H^1} \,dt'  \\
&\quad \le \sup_{t\in[0,T]} \min\{t^{1/4},1\}\,e^{t\lambda_1}\int_0^t \max\{(t-t')^{-1/2},1\}\, e^{-(t-t')\lambda_1}\norm{\mc{R}[\kappa]-\mc{R}[\varphi]}_{L^2}  \,dt'  \\
&\quad \le c\sup_{t\in[0,T]}\min\{t^{1/4},1\}\int_0^t \max\{(t-t')^{-1/2},1\}\max\{(t')^{-3/4},1\}\, e^{-2t'\lambda_1}\, dt'\; \bigg(\norm{\kappa-\varphi}_{\mc{Y}_0} \big(\norm{\kappa}_{\mc{Y}_1}^3(\norm{\kappa}_{\mc{Y}_0}^2+1) \\
&\hspace{2cm}+\norm{\varphi}_{\mc{Y}_1}^3(\norm{\varphi}_{\mc{Y}_0}^2+1) \big) +\norm{\kappa-\varphi}_{\mc{Y}_1} \big(\norm{\kappa}_{\mc{Y}_1}^2(\norm{\kappa}_{\mc{Y}_0}+1)+\norm{\varphi}_{\mc{Y}_1}^2(\norm{\varphi}_{\mc{Y}_0}+1) \big) \bigg)  \\
&\quad \le c\,M_1^3(M_0^2+1)\norm{\kappa-\varphi}_{\mc{Y}_0} + c\,M_1^2(M_0+1)\norm{\kappa-\varphi}_{\mc{Y}_1}\,,
\end{align*} 
and, by a similar series of estimates,
\begin{align*}
\norm{\Psi[\kappa]-\Psi[\varphi]}_{\mc{Y}_0}
& \le c\,M_1^3(M_0^2+1)\norm{\kappa-\varphi}_{\mc{Y}_0} + c\,M_1^2(M_0+1)\norm{\kappa-\varphi}_{\mc{Y}_1}\,.
\end{align*} 
For fixed $M_0$, taking $M_1$ small enough that both $c\,M_1^3(M_0^2+1)<\frac{1}{2}$ and $c\,M_1^2(M_0+1)<\frac{1}{2}$, by the contraction mapping theorem there exists a unique fixed point of the map $\Psi$ in $B_{M_0}(\mc{Y}_0)\cap B_{M_1}(\mc{Y}_1)$. 
Furthermore, in the case of small data, combining the bounds \eqref{BM_Y0bound} and \eqref{BM_Y1bound}, and using that $M_0=2\norm{\kappa_{\rm in}}_{L^2}$ and $M_1=2c\norm{\kappa_{\rm in}}_{L^2}$, we obtain the estimate \eqref{exp_bound}.
\end{proof}

\begin{remark}\label{rem:bootstrap}
Here we sketch a bootstrapping argument for showing that $\kappa$ is in fact smooth for all positive times.
By a calculation analogous to that above equation \eqref{BM_Y0bound}, it can be shown that $\kappa(s,t)$ in fact belongs to $H^2(I)$ for any positive time. Using this additional regularity on the right hand side of the equation \eqref{noforce_eqn2} for $\tau$, we may show that $\tau\in H^2(I)$ as well.
Then, applying $\mc{L}^{1-\epsilon/4}$ to the Duhamel formula \eqref{theta_mild} for $\kappa$:
\begin{align*}
\mc{L}^{1-\epsilon/4}\kappa(s,t) = \mc{L}^{1-\epsilon/4}e^{-(t-t_0)\mc{L}}\kappa\big|_{t=t_0} + \int_{t_0}^t \mc{L}^{1-\epsilon/4}e^{-(t-t')\mc{L}}(\mc{R}[\kappa(s,t')])_s \, dt'\,,
\end{align*}
since $(\mc{R}[\kappa(s,t')])_s\in L^2(I)$ (see \eqref{remainder}), we may use 
 Lemma \ref{lem:smoothing} to obtain that $\kappa\in C([t_0,T];H^{4-\epsilon})$ for any $\epsilon,t_0>0$. Note that we can also show that $\kappa\in L^2([t_0,T];H^4(I))$ by applying the full operator $\mc{L}$ above and estimating the integral term in $L^2$ in time. 

Passing to higher regularity becomes a bit tricky with the boundary conditions \eqref{noforce_BCs}: if we were instead working on the torus, we could simply apply $\mc{L}^k$, $k\ge1$, to the Duhamel formula, and use that $\mc{L}$ commutes with the semigroup $e^{-t\mc{L}}$.
 However, on the interval $I$, this requires $(\mc{R}[\kappa])_s\in D(\mc{L})$, which is not necessarily the case since, e.g., $\tau_s\kappa_s$ does not satisfy the boundary conditions \eqref{noforce_BCs}. 

 Instead, we can use that $\dot\kappa\in L^2([t_0,T];L^2(I))$ since the right hand side of the $\dot\kappa$ equation \eqref{noforce_eqn1} now belongs to $L^2([t_0,T];L^2(I))$. Taking a time derivative of the equation \eqref{noforce_eqn1} for $\kappa$, we have
 \begin{align*}
 \p_t(\dot\kappa) &= -\mc{L}\dot\kappa +(2+\gamma)(3\kappa^2\dot\kappa)_{ss} +(\tau\dot\kappa)_{ss}  + (2+\gamma)(\dot\tau\kappa)_{ss} +(1+\gamma)(\tau_s\dot\kappa-\dot\tau\kappa_s)_s \,.
 \end{align*}
Here $\dot\tau$ satisfies the elliptic equation
\begin{align*}
\dot\tau_{ss}-\frac{\kappa^2}{1+\gamma}\dot\tau &=\frac{1}{1+\gamma}\bigg(2\kappa\dot\kappa\tau + 4\kappa^3\dot\kappa + 2((\kappa_s\dot\kappa)_s-\kappa_{ss}\dot\kappa)-(4+3\gamma)(\dot\kappa_s\kappa+\kappa_s\dot\kappa)_s  \bigg)\,,
\end{align*}
and we have that at each time $t>t_0$, $\norm{\dot\tau}_{H^1(I)}\le c\big(\norm{\kappa}_{H^2}^3\norm{\dot\kappa}_{L^2} +\norm{\kappa}_{H^1}\norm{\dot\kappa}_{\dot H^1}\big)$.
  Using the Duhamel formula for $\dot\kappa$, we have
 \begin{align*}
 \dot\kappa &= e^{-(t-t_0)\mc{L}}\dot\kappa\big|_{t=t_0} +\int_{t_0}^t e^{-(t-t')\mc{L}}\bigg((\mc{R}_1[\kappa(s,t')])_{ss}+(\mc{R}_2[\kappa(s,t')])_s\bigg) \, dt'\,,
 \end{align*}
where $\mc{R}_1=(2+\gamma)3\kappa^2\dot\kappa +\tau\dot\kappa +(2+\gamma)\dot\tau\kappa$ and $\mc{R}_2=(1+\gamma)(\tau_s\dot\kappa-\dot\tau\kappa_s)$. Since we need $\dot\kappa\in H^1(I)$ to make sense of $\dot\tau$, we can again obtain estimates for $\dot\kappa$ in a time-weighted space analogous to $\mc{Y}_0\cap\mc{Y}_1$. In particular, we can estimate 
\begin{align*}
\sup_{t\in[t_0,T]} \bigg(\norm{\dot\kappa}_{L^2(I)} + \min\{(t-t_0)^{1/4},1\}\norm{\dot\kappa}_{\dot H^1(I)} \bigg)
\end{align*}
in a similar way as in estimates \eqref{BM_Y0bound} and \eqref{BM_Y1bound}, and the above argument begins to repeat. Repeating this, we can show that $\p_t^k\kappa\in L^2([kt_0,T];H^4(I))$. To obtain spatial regularity from the time regularity, we may use that, once $\dot\kappa\in L^2([2t_0,T];H^4(I))$, we have $\mc{L}\dot\kappa\in L^2([2t_0,T];L^2(I))$, and, using the equation \eqref{noforce_eqn1} for $\dot\kappa$, $\mc{L}(-\mc{L}\kappa+(2+\gamma)(\kappa^3)_{ss} +(2+\gamma)(\tau_s\kappa)_s+(\tau\kappa_s)_s) \in L^2([2t_0,T];L^2(I))$, i.e. $\kappa\in L^2([2t_0,T];H^8(I))$.
\end{remark}


\section{Periodic active forcing and swimming}\label{sec:periodic}
We now consider the system \eqref{theta_eqn}--\eqref{BCs} with nonzero $T$-periodic active forcing $\kappa_0(s,t)$ along the filament, corresponding to a nonzero preferred curvature, as derived in \cite{camalet2000generic,thomases2017role}.
 In particular, at a fixed time $t$, let $\kappa_0(s)$ denote the preferred curvature of the fiber. For a planar fiber, recalling that $\X_{ss}=\kappa(s)\be_{\rm n}(s)$ and $(\be_{\rm n})_s=-\kappa(s)\be_{\rm t}(s)$, we may compute the elastic energy of the fiber as
\begin{equation}\label{fvariational}
\begin{aligned}
G[\X] &= \frac{1}{2}\int_0^1\bigg(\big(\kappa(s)-\kappa_0(s)\big)^2 + \lambda(s)(\abs{\X_s}^2-1)\bigg)\, ds \\
&= \frac{1}{2}\int_0^1\bigg(\abs{\X_{ss}-\kappa_0\be_{\rm n}}^2 + \lambda(\abs{\X_s}^2-1) \bigg)\, ds\,, 
\end{aligned}
\end{equation}
along with the boundary conditions $(\X_{ss}-\kappa_0\be_{\rm n})\big|_{s=0,1}=0$, $(\X_{ss}-\kappa_0\be_{\rm n})_s\big|_{s=0,1}=0$, $\lambda|_{s=0,1}=0$. The Lagrange multiplier $\lambda(s)$ is equal to the tension $\tau$ when $\kappa_0\equiv 0$. Taking the variation and integrating by parts yields 
\begin{align*}
\frac{d}{d\varepsilon}\bigg|_{\varepsilon=0}G[\X+\varepsilon\bm{Y}]  &= -\int_0^1 {\bm Y}_s\cdot\bigg(\X_{sss}-(\kappa_0)_s\be_{\rm n} - \kappa_0(\be_{\rm n})_s -\lambda\X_s\bigg)\, ds \,. 
\end{align*}
Defining the tension of the fiber to be $\tau=\lambda-\kappa\kappa_0$, we obtain the form of forcing appearing on the right hand side of \eqref{classical}.

\subsection{Existence of solution with force}
Given the complicated way that the forcing terms due to $\kappa_0$ enter the system \eqref{theta_eqn}--\eqref{BCs}, we will first need to establish existence and uniqueness of solutions for $t\in[0,T]$ for sufficiently small $\kappa_0$. In particular, we show the following lemma.

\begin{lemma}[Well-posedness with forcing]\label{lem:wellpo_force}
There exists a constant $0<M<1$ such that, given a preferred curvature $\kappa_0(s,t)\in C^1([0,T]; H^1(I))$ and an initial condition $\overline\kappa_{\rm in}\in L^2(I)$ satisfying 
\begin{align*}
cT\bigg(\sup_{0\le t\le T}\norm{\dot\kappa_0}_{L^2(I)}\bigg)= c_1M\,, \quad \sup_{0\le t\le T}\norm{\kappa_0}_{H^1}=c_2M\,, \qquad \norm{\overline\kappa_{\rm in}}_{L^2}=c_3M\,,
\end{align*}
for $c_1>0$, $c_2>0$, $c_3\ge 0$, and $c_1+c_3\le \frac{1}{2}$,
there exists a unique mild solution $\overline\kappa\in C([0,T];L^2(I))\cap C((0,T];H^1(I))$ to the forced system \eqref{theta_eqn}--\eqref{BCs} satisfying
 \begin{equation}\label{force_thm_bd}
 \sup_{0\le t\le T}\bigg(\norm{\overline\kappa}_{L^2(I)}+ \min\{t^{1/4},1\}\norm{\overline\kappa}_{\dot H^1(I)}\bigg) \le c M\,.
 \end{equation} 
\end{lemma} 

\begin{proof}
We begin by deriving estimates for the forced tension equation \eqref{T_eqn}. Recalling the definition of the bilinear form $\mc{B}$ \eqref{T_bilinear}, where now we write $\kappa=\overline\kappa+\kappa_0$, we have that a weak solution of \eqref{T_eqn} satisfies 
\begin{equation}\label{T_weaksol_forced}
\begin{aligned}
\mc{B}(\overline\tau,\phi) &= -\int_0^1\frac{1}{1+\gamma}\bigg(\overline\kappa(\overline\kappa+\kappa_0)^2(\overline\kappa+2\kappa_0)\phi + \overline\kappa_s(\overline\kappa+\kappa_0)_s\phi\\
&\qquad+(1+\gamma)\big(\overline\kappa(\overline\kappa+2\kappa_0)\big)_s\phi_s+(2+\gamma)\overline\kappa_s(\overline\kappa+\kappa_0)\phi_s  \bigg) \,ds \qquad \forall\,\phi\in H^1_0(I)\,.
\end{aligned}
\end{equation}

Existence of a weak solution $\overline\tau\in H^1_0(I)$ again follows by Lax-Milgram, and, using Lemma \ref{lem:GN} as in the proof of the unforced estimate \eqref{T_H1bound}, we have that $\overline\tau$ satisfies
\begin{align*}
\mc{B}(\overline\tau,\overline\tau) &\le c\bigg( \norm{\overline\kappa}_{\dot H^1}\big(\norm{\overline\kappa}_{\dot H^1}\norm{\overline\kappa}_{L^2}+\norm{\kappa_0}_{H^1}\norm{\kappa_0}_{L^2}\big)\norm{(\overline\kappa+\kappa_0)\overline\tau}_{L^2}\\
&\qquad + \norm{\overline\kappa}_{\dot H^1}\big(\norm{\overline\kappa}_{\dot H^1}+\norm{\kappa_0}_{H^1}\big)\norm{\overline\tau}_{H^1}  \bigg)\,.
\end{align*}

Applying Young's inequality and again using that $\norm{\overline\tau}_{H^1}^2\le c\norm{\overline\tau_s}_{L^2}^2\le c\mc{B}(\overline\tau,\overline\tau)$, we obtain
\begin{equation}\label{Tf_bound}
\norm{\overline\tau}_{H^1}\le c\big( \norm{\overline\kappa}_{\dot H^1}^2+\norm{\kappa_0}_{H^1}^2\big)\big(\norm{\overline\kappa}_{L^2}+\norm{\kappa_0}_{L^2}+1\big) \,.
\end{equation}

We also need to amend the Lipschitz estimate of Lemma \ref{lem:T_lemma} to account for the presence of the forcing terms. As before, given $\overline\kappa$, $\overline\varphi\in H^1(I)$, we define $\overline\tau^\kappa$ and $\overline\tau^\varphi$ to be the $H^1_0$ solutions to the forced tension equation \eqref{T_eqn} using $\overline\kappa$ and $\overline\varphi$, respectively, on the right hand side. We have that the difference $\wh \tau:=\overline\tau^\kappa-\overline\tau^\varphi$ now satisfies
 \begin{align*}
 \wh \tau_{ss}-\frac{(\overline\kappa+\kappa_0)^2+(\overline\varphi+\kappa_0)^2}{2(1+\gamma)}\wh \tau
 &=\frac{1}{1+\gamma}\bigg(\frac{1}{2}(\overline\kappa-\overline\varphi)(\overline\kappa+\overline\varphi+2\kappa_0)(\overline\tau^\kappa+\overline\tau^\varphi)\\
 &\qquad +(\overline\kappa-\overline\varphi)\big(\overline\kappa^3+\overline\varphi^3+(\overline\kappa\overline\varphi+5\kappa_0^2)(\overline\kappa+\overline\varphi)+4\kappa_0(\overline\kappa^2+\overline\kappa\overline\varphi+\overline\varphi^2)+2\kappa_0^3\big)\\
 &\qquad + (\overline\kappa-\overline\varphi)_s(\overline\kappa+\overline\varphi+\kappa_0)_s-(1+\gamma)\big((\overline\kappa-\overline\varphi)(\overline\kappa+\overline\varphi+2\kappa_0)\big)_{ss}\\
  &\qquad -(2+\gamma)\big((\overline\kappa-\overline\varphi)_s(\overline\kappa+\kappa_0)+(\overline\kappa-\overline\varphi)\overline\varphi_s\big)_s\bigg)\,.
 \end{align*}
Proceeding similarly to the proof of Lemma \ref{lem:T_lemma} and using the bound \eqref{Tf_bound} to estimate the first term involving $\overline\tau^\kappa+\overline\tau^\varphi$, we have
\begin{align*}
\norm{\wh\tau}_{H^1}^2&\le c\int_0^1\bigg( \wh \tau_s^2+\frac{(\overline\kappa+\kappa_0)^2+(\overline\varphi+\kappa_0)^2}{2(1+\gamma)}\wh \tau^2\bigg)\,ds\\
&\le 
c\norm{\overline\kappa-\overline\varphi}_{L^2}
 \big( \norm{\overline\kappa}_{\dot H^1}^2+ \norm{\overline\varphi}_{\dot H^1}^2+\norm{\kappa_0}_{H^1}^2\big)\big(\norm{\overline\kappa}_{L^2}+\norm{\overline\varphi}_{L^2}+\norm{\kappa_0}_{L^2}+1\big)^2 \norm{\wh\tau}_{L^\infty}\\
 &\qquad +c\norm{\overline\kappa-\overline\varphi}_{L^2}\big(\norm{\overline\kappa}_{L^6}^3+\norm{\overline\varphi}_{L^6}^3+\norm{\kappa_0}_{L^6}^3 \big)\norm{\wh\tau}_{L^\infty}\\
 &\qquad + c\norm{\overline\kappa-\overline\varphi}_{\dot H^1}(\norm{\overline\kappa}_{\dot H^1}+\norm{\overline\varphi}_{\dot H^1}+\norm{\kappa_0}_{H^1})\norm{\wh\tau}_{H^1}\,.
\end{align*}
Using Lemma \ref{lem:GN} to absorb the $L^6$ terms into the first expression and applying Young's inequality, we obtain the Lipschitz bound
\begin{equation}\label{Tf_lip_bound}
\begin{aligned}
\norm{\overline\tau^\kappa-\overline\tau^\varphi}_{H^1} &\le
c\norm{\overline\kappa-\overline\varphi}_{L^2}
 \big( \norm{\overline\kappa}_{\dot H^1}^2+ \norm{\overline\varphi}_{\dot H^1}^2+\norm{\kappa_0}_{H^1}^2\big)\big(\norm{\overline\kappa}_{L^2}+\norm{\overline\varphi}_{L^2}+\norm{\kappa_0}_{L^2}+1\big)^2\\
&\qquad
 + c\norm{\overline\kappa-\overline\varphi}_{\dot H^1}\big(\norm{\overline\kappa}_{\dot H^1}+\norm{\overline\varphi}_{\dot H^1}+\norm{\kappa_0}_{H^1}\big) \,.
\end{aligned}
\end{equation}

Equipped with estimates \eqref{Tf_bound} and \eqref{Tf_lip_bound}, we now consider the forced evolution equation \eqref{theta_eqn}. We take $\overline\kappa_{\rm in}(s)=\overline\kappa(s,0)$ to be the initial difference from the fiber's preferred curvature and recall the definition \eqref{L_op} of the operator $\mc{L}$. We may then write the mild solution formula for \eqref{theta_eqn}--\eqref{BCs}, which now consists of three terms: 
\begin{equation}\label{theta_mild_f}
\overline\Psi[\overline\kappa](s,t) = e^{-t\mc{L}}\overline\kappa_{\rm in}(s) -\int_0^t e^{-(t-t')\mc{L}}\,\dot\kappa_0(s,t')\,dt'+ \int_0^t e^{-(t-t')\mc{L}}(\overline{\mc{R}}[\overline\kappa(s,t')])_s \, dt' \,.
\end{equation}
Here $(\overline{\mc{R}}[\overline\kappa(s,t')])_s$ contains all nonlinear terms in equation \eqref{theta_eqn}; in particular, 
\begin{equation}\label{Rf_def}
\begin{aligned}
\overline{\mc{R}}[\overline\kappa]&=  
3(2+\gamma)\overline\kappa(\overline\kappa+2\kappa_0)\overline\kappa_s+(5+3\gamma)\kappa_0^2\overline\kappa_s+ (5+2\gamma)\overline\kappa^2(\kappa_0)_s  \nonumber \\
&\hspace{3cm}
+2(3+\gamma)\overline\kappa\kappa_0(\kappa_0)_s +(2+\gamma)\overline\tau_s(\overline\kappa+\kappa_0)+\overline\tau(\overline\kappa+\kappa_0)_s  \,.
\end{aligned}
\end{equation}

We show that \eqref{theta_mild_f} admits a unique fixed point. We define $\mc{Y}_T$ as the space of functions for which the following norm is finite: 
\begin{equation}\label{YT_def}
\norm{\cdot}_{\mc{Y}_T} := \sup_{0\le t\le T}\bigg(\norm{\cdot}_{L^2(I)}+\min\{t^{1/4},1\}\norm{\cdot}_{\dot H^1(I)}\bigg)\,.
\end{equation} 
Considering a ball of radius $M$ in $\mc{Y}_T$,  
\begin{align*}
B_M(\mc{Y}_T)=\{ u\in \mc{Y}_T\, : \, \norm{u}_{\mc{Y}_T}\le M\}\,,
\end{align*}
we have that $\overline\Psi$ is bounded from $B_M(\mc{Y}_T)$ to $B_M(\mc{Y}_T)$ for some $M>0$. In particular, by Lemma \ref{lem:smoothing}, the second term of $\overline\Psi[\overline\kappa]$ involving $\dot\kappa_0$ satisfies
\begin{equation}\label{dotkap_est}
\begin{aligned}
\norm{\int_0^t e^{-(t-t')\mc{L}}\,\dot\kappa_0 \, dt'}_{\dot H^m(I)} 
&\le c\int_0^t \max\{(t-t')^{-m/4},1\}\,e^{-(t-t')\lambda_1}\norm{\dot\kappa_0}_{L^2(I)} dt' \\
&\le c\,\max\{t^{1-m/4},1\}\bigg(\sup_{0\le t\le T}\norm{\dot\kappa_0}_{L^2(I)} \bigg)\,.
\end{aligned}
\end{equation} 

Furthermore, we may bound the nonlinear terms from $(\overline{\mc{R}}[\overline\kappa])_s$ following similar steps to the unforced estimate \eqref{Rest_L2} for $\mc{R}[\kappa]$. Using Lemma \eqref{lem:GN} and the tension estimate \eqref{Tf_bound}, we have
\begin{equation}\label{Rfest_L2}
\begin{aligned}
\norm{\overline{\mc{R}}[\overline\kappa(\cdot,t)]}_{L^2} 
&\le c\big(\norm{\overline\kappa}_{L^\infty}^2+\norm{\kappa_0}_{L^\infty}^2+\norm{\overline\tau}_{H^1}\big)\big(\norm{\overline\kappa}_{\dot H^1}+\norm{\kappa_0}_{H^1}\big) \\
&\le c\big( \norm{\overline\kappa}_{\dot H^1}^3+\norm{\kappa_0}_{H^1}^3\big)\big(\norm{\overline\kappa}_{L^2}+\norm{\kappa_0}_{L^2}+1\big) \, .
\end{aligned}
\end{equation}

By Lemma \ref{lem:smoothing}, we then have 
\begin{equation}\label{Rf_Hm_est}
\begin{aligned}
&\norm{\int_0^t e^{-(t-t')\mc{L}}\big(\overline{\mc{R}}[\overline\kappa(\cdot,t')]\big)_s dt'}_{\dot H^m(I)} \\
&\quad \le c\int_0^t \max\{(t-t')^{-(m+1)/4},1\}\,e^{-(t-t')\lambda_1}\norm{\overline{\mc{R}}[\overline\kappa]}_{L^2} dt' \\
&\quad \le c\int_0^t \max\{(t-t')^{-(m+1)/4},1\}\,e^{-(t-t')\lambda_1}\big( \norm{\overline\kappa}_{\dot H^1}^3+\norm{\kappa_0}_{H^1}^3\big)\big(\norm{\overline\kappa}_{L^2}+\norm{\kappa_0}_{L^2}+1\big)\, dt' \\
 &\quad \le c\int_0^t \max\{(t-t')^{-(m+1)/4},1\}\,\max\{(t')^{-3/4},1\}\,e^{-(t-t')\lambda_1} \,dt' \, \big(\norm{\overline\kappa}_{\mc{Y}_T}^3 \\
 &\hspace{9cm}+\norm{\kappa_0}_{\mc{Y}_T}^3\big)\left( \norm{\overline\kappa}_{\mc{Y}_T}+\norm{\kappa_0}_{\mc{Y}_T} +1 \right)  \\
 &\quad \le c\,\max\{t^{-m/4},1\}\big(\norm{\overline\kappa}_{\mc{Y}_T}^3+\norm{\kappa_0}_{\mc{Y}_T}^3\big)\left( \norm{\overline\kappa}_{\mc{Y}_T}+\norm{\kappa_0}_{\mc{Y}_T} +1 \right)\, . 
\end{aligned}
\end{equation} 

Combining the estimates \eqref{dotkap_est} and \eqref{Rf_Hm_est} for $\dot\kappa_0$ and $\overline{\mc{R}}[\overline\kappa]$, we then have 
\begin{equation}\label{Phif_est}
\begin{aligned}
\norm{\overline\Psi[\overline\kappa]}_{\mc{Y}_T} &\le \norm{\overline\kappa_{\rm in}}_{L^2} 
+cT\bigg(\sup_{0\le t\le T}\norm{\dot\kappa_0}_{L^2}\bigg)+ c\big(\norm{\overline\kappa}_{\mc{Y}_T}^3+\norm{\kappa_0}_{\mc{Y}_T}^3\big)\left( \norm{\overline\kappa}_{\mc{Y}_T}+\norm{\kappa_0}_{\mc{Y}_T} +1 \right)
\end{aligned}
\end{equation}
for $\overline\kappa\in B_M(\mc{Y}_T)$. Noting that $\norm{\kappa_0}_{\mc{Y}_T}\le c(T)\big(\sup_{0\le t\le T}\norm{\kappa_0}_{H^1}\big)$, we take the preferred curvature $\kappa_0(s,t)$ and initial condition $\overline\kappa_{\rm in}(s)$ to satisfy 
\begin{equation}\label{f_bounds}
cT\bigg(\sup_{0\le t\le T}\norm{\dot\kappa_0}_{L^2(I)}\bigg)= c_1M, \quad \sup_{0\le t\le T}\norm{\kappa_0}_{H^1}=c_2M, \qquad \norm{\overline\kappa_{\rm in}}_{L^2}=c_3M,
\end{equation}
where $c_1>0$, $c_2>0$, $c_3\ge 0$, and $c_1+c_3\le \frac{1}{2}$. We then have
\begin{equation}\label{Psif_Mest}
\norm{\overline\Psi[\overline\kappa]}_{\mc{Y}_T} \le \left(c_1+c_3+c(M^3+M^2)\right) M \le M
\end{equation}
for $M$ sufficiently small. \\

We now show that $\overline\Psi$ is a contraction on $B_M(\mc{Y}_T)$. Given $\overline\kappa$, $\overline\varphi\in B_M(\mc{Y}_T)$, we first amend the unforced Lipschitz estimate \eqref{R_lip} for $\mc{R}[\kappa]$ to account for the forcing terms in $\overline{\mc{R}}[\overline\kappa]$. Using the tension bound \eqref{Tf_bound} and tension Lipschitz estimate \eqref{Tf_lip_bound}, as well as Lemma \ref{lem:GN}, we have 
\begin{align*}
&\norm{\overline{\mc{R}}[\overline\kappa(\cdot,t)]-\overline{\mc{R}}[\overline\varphi(\cdot,t)]}_{L^2(I)} \\
&\qquad \le c\bigg(\norm{\overline\kappa-\overline\varphi}_{L^\infty}\big(\norm{\overline\kappa}_{L^\infty}+\norm{\overline\varphi}_{L^\infty}+\norm{\kappa_0}_{L^\infty}\big)\big(\norm{\overline\kappa}_{\dot H^1}+\norm{\kappa_0}_{H^1}\big)   \\
&\qquad\quad + \norm{\overline\kappa-\overline\varphi}_{\dot H^1}\big(\norm{\overline\varphi}_{L^\infty}^2+\norm{\kappa_0}_{L^\infty}^2\big) +\norm{\overline\tau^\kappa-\overline\tau^\varphi}_{H^1}\big(\norm{\overline\kappa}_{\dot H^1}+\norm{\kappa_0}_{H^1}\big) + \norm{\overline\kappa-\overline\varphi}_{\dot H^1}\norm{\overline\tau^\varphi}_{H^1}\bigg) \\
&\qquad \le c\bigg(\norm{\overline\kappa-\overline\varphi}_{L^2}
 \big( \norm{\overline\kappa}_{\dot H^1}^3+ \norm{\overline\varphi}_{\dot H^1}^3+\norm{\kappa_0}_{H^1}^3\big)\big(\norm{\overline\kappa}_{L^2}+\norm{\overline\varphi}_{L^2}+\norm{\kappa_0}_{L^2}+1\big)^2\\
&\qquad\quad
 + \norm{\overline\kappa-\overline\varphi}_{\dot H^1}\big(\norm{\overline\kappa}_{\dot H^1}^2+\norm{\overline\varphi}_{\dot H^1}^2+\norm{\kappa_0}_{H^1}^2\big)\big(\norm{\overline\varphi}_{L^2}+\norm{\kappa_0}_{L^2}+1\big) \bigg) \,.
\end{align*}

Then, using the above estimate along with Lemma \ref{lem:smoothing}, we have 
\begin{align*}
&\norm{\int_0^t e^{-(t-t')\mc{L}}\left(\overline{\mc{R}}[\overline\kappa]-\overline{\mc{R}}[\overline\varphi]\right)_s\,dt'}_{\mc{Y}_T}\\
 &\quad \le \sup_{t\in[0,T]} \int_0^t \bigg(\max\{(t-t')^{-1/4},1\}\\
 &\hspace{2cm}+\min\{t^{1/4},1\}\max\{(t-t')^{-1/2},1\}\bigg)e^{-(t-t')\lambda_1}\norm{\overline{\mc{R}}[\overline\kappa]-\overline{\mc{R}}[\overline\varphi] }_{L^2(I)}\,dt' \\
&\quad \le c\sup_{t\in[0,T]}\int_0^t e^{-(t-t')\lambda_1}\bigg(\max\{(t-t')^{-1/4},1\}\\
 &\hspace{2cm}+\min\{t^{1/4},1\}\max\{(t-t')^{-1/2},1\}\bigg) \max\{(t')^{-3/4},1\}\,dt' \, \bigg( \norm{\overline\kappa}_{\mc{Y}_T}^5 +\norm{\overline\varphi}_{\mc{Y}_T}^5+\norm{\kappa_0}_{\mc{Y}_T}^5 \\
&\hspace{5cm}
+ \norm{\overline\kappa}_{\mc{Y}_T}^2 + \norm{\overline\varphi}_{\mc{Y}_T}^2 +\norm{\kappa_0}_{\mc{Y}_T}^2 \bigg)\norm{\overline\kappa-\overline\varphi}_{\mc{Y}_T} \\
&\quad \le c\norm{\overline\kappa-\overline\varphi}_{\mc{Y}_T}\bigg( M^5+ M^2+\norm{\kappa_0}_{\mc{Y}_T}^5+\norm{\kappa_0}_{\mc{Y}_T}^2 \bigg)\,.
\end{align*}

Using the bound \eqref{f_bounds} on $\kappa_0$, we therefore have
\begin{equation}\label{Psif_contraction}
\norm{\overline\Psi[\overline\kappa]-\overline\Psi[\overline\varphi]}_{\mc{Y}_T}
\le c(M^5+M^2)\norm{\overline\kappa-\overline\varphi}_{\mc{Y}_T}\,.
\end{equation}
Taking $M$ small enough that $c(M^5+M^2)\le\frac{1}{2}$, we obtain a contraction on $B_M(\mc{Y}_T)$, yielding Lemma \ref{lem:wellpo_force}. 
\end{proof}


\subsection{Existence of unique periodic solution}\label{subsec:unique_per}
We now show that given a $T$-periodic preferred curvature $\kappa_0$ satisfying the smallness assumptions of Lemma \ref{lem:wellpo_force}, there exists a unique $\overline\kappa_{\rm in}$ such that the solution $\overline\kappa$ to \eqref{theta_eqn}--\eqref{BCs} is $T$-periodic in time, and solutions starting from nearby initial data converge to this $T$-periodic $\overline\kappa$ over time.

\begin{proof}[Proof of Theorem \ref{thm:periodic}]
Given an initial condition $\overline\kappa_{\rm in}(s)$ and $T$-periodic preferred curvature $\kappa_0(s,t)$ satisfying the conditions of Lemma \ref{lem:wellpo_force}, we consider the unique solution $\overline\kappa(s,t)$ to the system \eqref{theta_eqn}--\eqref{BCs} at time $T$. Let $\Phi^T[\overline\kappa_{\rm in}]$ denote the time $T$ map $\overline\kappa_{\rm in}\mapsto\overline\kappa(s,T)$: 
\begin{equation}\label{PhiT_map}
\Phi^T[\overline\kappa_{\rm in}] = e^{-T\mc{L}}\overline\kappa_{\rm in} -\int_0^Te^{-(T-t')\mc{L}}\,\dot\kappa_0\, dt'+ \int_0^Te^{-(T-t')\mc{L}}\big(\overline{\mc{R}}[\overline\kappa]\big)_s\, dt' \,.
\end{equation}
We consider the map $\Phi^T$ from $B_{M'}(L^2(I))$ to $B_{M'}(L^2(I))$ and show that there exists a unique fixed point. Here $B_{M'}(L^2(I))$ is the ball 
\begin{align*}
B_{M'}(L^2(I))=\{ u\in L^2\, : \, \norm{u}_{L^2}\le M'\}
\end{align*}
where the constant $M'$ is given by
\begin{equation}\label{Mprime}
M'=\min\bigg(2c_1M, c_2M, \frac{M}{4}\bigg)
\end{equation} 
for $c_1$, $c_2$, and $M$ as in Lemma \ref{lem:wellpo_force}. 
 
We will again make use of the space $\mc{Y}_T$ \eqref{YT_def}, and note that for $u(s,t)\in \mc{Y}_T$, we have 
\begin{align*}
\norm{u(\cdot,T)}_{L^2(I)} \le \norm{u}_{\mc{Y}_T}\,.
\end{align*}

Taking $\overline\kappa_{\rm in}\in B_{M'}(L^2(I))$, we first note that, using the bounds \eqref{f_bounds} and the definition \eqref{Mprime} of $M'$, the estimate \eqref{Psif_Mest} on $\overline\Psi$ implies that the second and third terms on the right hand side of \eqref{PhiT_map} satisfy 
\begin{align*}
\norm{\int_0^Te^{-(T-t')\mc{L}}\,\dot\kappa_0\, dt'}_{L^2} &\le \frac{M'}{2}\,, \qquad
\norm{\int_0^Te^{-(T-t')\mc{L}}(\overline{\mc{R}}[\overline\kappa])_s\, dt'}_{L^2} \le \frac{M'}{4} \,.
\end{align*}

Furthermore, by Lemma \ref{lem:smoothing}, we have 
\begin{align*}
\norm{e^{-T\mc{L}}\overline\kappa_{\rm in}}_{L^2} &\le e^{-T\lambda_1}\norm{\overline\kappa_{\rm in}}_{L^2} \le e^{-T\lambda_1}M'\,.
\end{align*}
Taking the period $T$ large enough that $e^{-T\lambda_1}<\frac{1}{4}$, the time $T$ map satisfies 
\begin{align*}
\norm{\Phi^T[\overline\kappa_{\rm in}]}_{L^2(I)} \le M' \,.
\end{align*}

Next we need a Lipschitz estimate for $\Phi^T$. Given two initial conditions $\overline\kappa_{\rm in}$, $\overline\varphi_{\rm in}\in L^2(I)$, we begin by noting that, by Lemma \ref{lem:smoothing}, the difference in the linear evolution alone satisfies 
\begin{align*}
\norm{e^{t\lambda_1}\,e^{-t\mc{L}}(\overline\kappa_{\rm in}-\overline\varphi_{\rm in})}_{\mc{Y}_T} &= \sup_{0\le t\le T}e^{t\lambda_1}\bigg(\norm{e^{-t\mc{L}}(\overline\kappa_{\rm in}-\overline\varphi_{\rm in})}_{L^2}+ \min\{t^{1/4},1\}\norm{e^{-t\mc{L}}(\overline\kappa_{\rm in}-\overline\varphi_{\rm in})}_{\dot H^1}\bigg) \\
&\le \sup_{0\le t\le T}e^{t\lambda_1}\bigg(c e^{-t\lambda_1}\norm{\overline\kappa_{\rm in}-\overline\varphi_{\rm in}}_{L^2}\bigg)\\
&= c\norm{\overline\kappa_{\rm in}-\overline\varphi_{\rm in}}_{L^2(I)}\,.
\end{align*} 
Furthermore, using the estimate \eqref{Psif_contraction}, we may obtain the following bound on the difference in the evolution of the nonlinear terms: 
\begin{align*}
\norm{e^{t\lambda_1}\int_0^te^{-(t-t')\mc{L}}\big(\overline{\mc{R}}[\overline\kappa]-\overline{\mc{R}}[\overline\varphi]\big)_s\, dt'}_{\mc{Y}_T}
&\le \frac{1}{2}\norm{e^{t\lambda_1}\big(\overline\kappa-\overline\varphi\big)}_{\mc{Y}_T} \,.
\end{align*}

We then have 
\begin{align*}
\norm{e^{t\lambda_1}\big(\Phi^t[\overline\kappa_{\rm in}]-\Phi^t[\overline\varphi_{\rm in}]\big)}_{\mc{Y}_T} 
&\le c\norm{\overline\kappa_{\rm in}-\overline\varphi_{\rm in}}_{L^2}+ \frac{1}{2}\norm{e^{t\lambda_1}(\overline\kappa-\overline\varphi)}_{\mc{Y}_T} \,.
\end{align*}

Noting that $\Phi^t[\overline\kappa_{\rm in}](s)=\overline\kappa(s,t)$, we may subtract $\frac{1}{2}\norm{e^{t\lambda_1}(\overline\kappa-\overline\varphi)}_{\mc{Y}_T}$ from both sides to obtain 
\begin{equation}\label{periodic_bound}
\norm{e^{t\lambda_1}\big(\Phi^t[\overline\kappa_{\rm in}]-\Phi^t[\overline\varphi_{\rm in}]\big)}_{\mc{Y}_T} 
\le c\norm{\overline\kappa_{\rm in}-\overline\varphi_{\rm in}}_{L^2}\,.
\end{equation}
In particular, at time $T$ we have 
\begin{align*}
\norm{e^{T\lambda_1}\big(\Phi^T[\overline\kappa_{\rm in}]-\Phi^T[\overline\varphi_{\rm in}]\big)}_{L^2} \le \norm{e^{t\lambda_1}\big(\Phi^t[\overline\kappa_{\rm in}]-\Phi^t[\overline\varphi_{\rm in}]\big)}_{\mc{Y}_T} 
\le c\norm{\overline\kappa_{\rm in}-\overline\varphi_{\rm in}}_{L^2}\,,
\end{align*}
which we may rewrite as
\begin{equation}\label{periodic_contraction}
\norm{\Phi^T[\overline\kappa_{\rm in}]-\Phi^T[\overline\varphi_{\rm in}]}_{L^2(I)} \le c e^{-T\lambda_1}\norm{\overline\kappa_{\rm in}-\overline\varphi_{\rm in}}_{L^2(I)}\, .
\end{equation}
Taking the period $T$ of the preferred curvature $\kappa_0(s,t)$ to be large enough that $c e^{-T\lambda_1}<1$, equation \eqref{periodic_contraction} is a contraction on $L^2(I)$, yielding the unique periodic solution of Theorem \ref{thm:periodic}. \\

As a slight abuse of notation, we now let $\overline\kappa_{\rm in}$ denote the unique initial condition such that $\Phi^T[\overline\kappa_{\rm in}]=\overline\kappa_{\rm in}$.
The estimate \eqref{kappaH1} then follows from estimate \eqref{force_thm_bd} of Lemma \ref{lem:wellpo_force} since we must have $\overline\kappa_{\rm in}\in H^1(I)$. In particular, using the periodicity of $\overline\kappa=\Phi^t[\overline\kappa_{\rm in}]$, we have 
\begin{align*}
\norm{\overline\kappa_{\rm in}}_{H^1(I)} &= \norm{\overline\kappa(\cdot,T)}_{H^1(I)} \le c(T)\left(\norm{\overline\kappa(\cdot,T)}_{L^2(I)} + \min\{T^{1/4},1\}\norm{\overline\kappa(\cdot,T)}_{\dot H^1(I)} \right) \\
&\le c(T)\norm{\overline\kappa}_{\mc{Y}_T} \le c(T)M \, . 
\end{align*}
Furthermore, the bound \eqref{periodic_bound} implies that for any initial condition $\overline\varphi_{\rm in}\in L^2(I)$ satisfying $\norm{\overline\varphi_{\rm in}}_{L^2}\le M'$, the iterated time $T$ map $\Phi^{nT}[\overline\varphi_{\rm in}]$ satisfies 
\begin{align*}
 \norm{\overline\kappa_{\rm in}-\Phi^{nT}[\overline\varphi_{\rm in}]}_{L^2(I)} \le ce^{-nT\lambda_1}\norm{\overline\kappa_{\rm in}-\overline\varphi_{\rm in}}_{L^2(I)}, \quad n=1,2,\dots \,, 
 \end{align*} 
 which yields the convergence estimate \eqref{per_conv}. 
\end{proof}

\subsection{Swimming}\label{subsec:swim}
Finally, we calculate the small-amplitude fiber swimming speed stated in Theorem \ref{thm:swimming}. The proof will rely on two auxiliary lemmas, which we introduce below. 


To show the swimming expression \eqref{swim_speed1}, we will need to use the smallness and additional $H^3(I)$ regularity of $\kappa_0$ to show that the fiber tangent vector $\be_{\rm t}(s,t)$ is not varying greatly in space and time. In particular: 
\begin{lemma}\label{lem:middle}
Given a $T$-periodic preferred curvature $\kappa_0(s,t)\in C^1([0,T]; H^3(I))$ satisfying 
\begin{align*}
 T\bigg(\sup_{0\le t\le T}\norm{\dot\kappa_0}_{L^2(I)}\bigg)= \varepsilon \,, \quad \sup_{0\le t\le T}\norm{\kappa_0}_{H^1}=c_1\varepsilon 
  \end{align*}
 for some $0<\varepsilon<1$, let $(\overline\kappa,\overline\tau)=(\kappa-\kappa_0,\tau+\kappa_0^2)$ denote the unique $T$-periodic mild solution to \eqref{theta_eqn}-\eqref{BCs}. The corresponding evolution of the frame vector $\be_{\rm t}(s,t)$, given by \eqref{frame_ev} and \eqref{thetadot}, then satisfies
\begin{equation}\label{U0}
\begin{aligned}
\sup_{0\le t\le T}\norm{\be_{\rm t}(\cdot,t)- \be_{\rm t}(0,0)}_{L^2(I)} \le c\varepsilon\,.
\end{aligned}
\end{equation}
\end{lemma}

In addition to Lemma \ref{lem:middle}, in order to show the swimming expansion \eqref{swim_speed2}, we will need to consider a nearby simpler problem.  
We define $(\overline\kappa^{\rm lin},\overline\tau^{\rm lin})\in C([0,T];H^1(I)\times H^1(I))$ to be the unique periodic solution to the linear evolution equation 
\begin{align}
\dot{\overline\kappa}^{\rm lin} &=-\overline\kappa^{\rm lin}_{ssss}-\dot\kappa_0  \label{theta_eqn2_2}\\
(1+\gamma)\bigg(\overline\tau^{\rm lin}_{ss}+\big(\overline\kappa^{\rm lin}(\overline\kappa^{\rm lin}+2\kappa_0)\big)_{ss}+\big(\overline\kappa^{\rm lin}_s(\overline\kappa^{\rm lin}+\kappa_0)\big)_s\bigg) &= -\overline\kappa^{\rm lin}_{ss}(\overline\kappa^{\rm lin}+\kappa_0)   \label{T_eqn2_2} \\
\overline\kappa^{\rm lin}\big|_{s=0,1}=0, \; \overline\kappa^{\rm lin}_s\big|_{s=0,1}&=0, \; \overline\tau^{\rm lin}\big|_{s=0,1}=0 \,,  \label{BCs_eqn2_2}
\end{align}
where $\overline\kappa^{\rm lin}(s,0):=\overline\kappa^{\rm lin}_{\rm in}$.
The existence of a unique periodic mild solution $(\overline\kappa^{\rm lin},\overline\tau^{\rm lin})$ follows by a simple modification of Section \ref{subsec:unique_per}, and can be shown to satisfy analogous estimates to \eqref{Tf_bound} and \eqref{kappaH1} for the full system \eqref{theta_eqn}-\eqref{BCs}. 

For sufficiently small $\kappa_0$, we have that the unique periodic solution $(\overline\kappa,\overline\tau)$ to \eqref{theta_eqn}-\eqref{BCs} is close to the unique periodic $(\overline\kappa^{\rm lin},\overline\tau^{\rm lin})$: 
\begin{lemma}\label{lem:lin_close}
Given a $T$-periodic preferred curvature $\kappa_0(s,t)$ satisfying 
 \begin{align*}
T\bigg(\sup_{0\le t\le T}\norm{\dot\kappa_0}_{L^2(I)}\bigg)= \varepsilon \,, \quad \sup_{0\le t\le T}\norm{\kappa_0}_{H^1}=c_1\varepsilon 
 \end{align*}
 for some $0<\varepsilon\ll1$, the unique $T$-periodic mild solutions $(\overline\kappa,\overline\tau)$ and $(\overline\kappa^{\rm lin},\overline\tau^{\rm lin})$ to \eqref{theta_eqn}-\eqref{BCs} and \eqref{theta_eqn2_2}-\eqref{BCs_eqn2_2}, respectively, satisfy the following:
\begin{align}
\sup_{0\le t\le T}\|\overline\kappa-\overline\kappa^{\rm lin}\|_{H^1} &\le c\varepsilon^3\, ; \label{kap_eps_bd}\\
\sup_{0\le t\le T}\|\overline\tau-\overline\tau^{\rm lin}\|_{H^1} &\le c\varepsilon^4 \,. \label{tau_eps_bd}
\end{align}
\end{lemma}

 The proofs of Lemmas \ref{lem:middle} and \ref{lem:lin_close} are given below.

\begin{proof}[Proof of Lemma \ref{lem:middle}]
First, using the Duhamel formula \eqref{theta_mild_f} for $\overline\kappa(s,t)$, by Lemma \ref{lem:smoothing} and estimate \eqref{Rf_Hm_est}, we may show that 
\begin{align*}
 \min\{t^{m/4},1\}\norm{\overline\kappa(\cdot,t)}_{\dot H^m(I)}\le c\norm{\overline\kappa_{\rm in}}_{L^2(I)}
 \end{align*} 
 for $0\le m\le 3$, due to the additional regularity of $\kappa_0$. By the periodicity of $\overline\kappa$, we in fact have 
 \begin{align*}
  \norm{\overline\kappa}_{\dot H^m(I)}\le c(T,m)\norm{\overline\kappa_{\rm in}}_{L^2(I)} \le c(T,m)\varepsilon \,.
  \end{align*} 
Then, using the equation \eqref{thetadot} for $\dot\theta$, we have
\begin{align*}
\sup_{0\le t\le T}\norm{\be_{\rm t}(\cdot,t)-\be_{\rm t}(\cdot,0)}_{L^2(I)}&\le c\sup_{0\le t\le T}\|\dot\theta\|_{L^2(I)} \\
&\le c\sup_{0\le t\le T}\big(\norm{\overline\kappa_{sss}}_{L^2(I)}+\norm{\dot\kappa_0}_{L^2(I)}\big)+c\varepsilon^2\\
&\le c(T)\varepsilon\,.
\end{align*}
Furthermore, using the frame relation $(\be_{\rm t})_s = \kappa\be_{\rm n}$, we have that 
\begin{align*}
\norm{\be_{\rm t}(s,0)-\be_{\rm t}(0,0)}_{L^2(I)} \le c\norm{\kappa}_{L^2(I)} \le c\varepsilon\,.
\end{align*}
Combining the above two estimates yields Lemma \ref{lem:middle}.
\end{proof}

\begin{proof}[Proof of Lemma \ref{lem:lin_close}]
We may use Duhamel's formula to write the difference $\overline\kappa-\overline\kappa^{\rm lin}$ as
\begin{align*}
\overline\kappa-\overline\kappa^{\rm lin} = e^{-t\mc{L}}(\overline\kappa_{\rm in}-\overline\kappa_{\rm in}^{\rm lin}) + \int_0^t e^{-(t-t')\mc{L}}\big(\overline{\mc{R}}[\overline\kappa(s,t')]\big)_s\,dt',
\end{align*}
where $\overline{\mc{R}}[\overline\kappa]$ is as in \eqref{Rf_def}. Recalling that $\overline\kappa,\overline\kappa^{\rm lin}\in C([0,T];H^1(I))$ by periodicity, we may replace the space $\mc{Y}_T$ of \eqref{YT_def} by $C([0,T];H^1(I))$ in the estimate \eqref{Rf_Hm_est} on the remainder terms $\overline{\mc{R}}[\overline\kappa]$. We may use this estimate along with Lemma \ref{lem:smoothing} to obtain
\begin{equation}\label{kapkaplin1}
\sup_{0\le t\le T}\|\overline\kappa-\overline\kappa^{\rm lin}\|_{H^1(I)} 
\le \| \overline\kappa_{\rm in}-\overline\kappa_{\rm in}^{\rm lin}\|_{L^2(I)} +c\varepsilon^3 \,.
\end{equation}

Furthermore, using the $T$-periodicity of both $\overline\kappa$ and $\overline\kappa^{\rm lin}$, we have 
\begin{align*}
\|\overline\kappa_{\rm in}-\overline\kappa_{\rm in}^{\rm lin} \|_{L^2(I)} &= 
\norm{e^{-T\mc{L}}(\overline\kappa_{\rm in}-\overline\kappa_{\rm in}^{\rm lin}) + \int_0^T e^{-(T-t')\mc{L}}\big(\overline{\mc{R}}[\overline\kappa(s,t')]\big)_s\,dt'}_{L^2(I)} \\
&\le ce^{-T\lambda_1}\|\overline\kappa_{\rm in}-\overline\kappa_{\rm in}^{\rm lin}\|_{L^2(I)} + c\varepsilon^3,
\end{align*}
where we have again used Lemma \ref{lem:smoothing}. In particular, provided that $T$ is large enough that $ce^{-T\lambda_1}\le \frac{1}{2}$, we have
\begin{equation}\label{kapkaplin2}
\|\overline\kappa_{\rm in}-\overline\kappa_{\rm in}^{\rm lin} \|_{L^2}\le c\varepsilon^3.
\end{equation}
Combining \eqref{kapkaplin1} and \eqref{kapkaplin2}, we obtain the desired estimate for $\overline\kappa-\overline\kappa^{\rm lin}$. \\

For the tension estimate \eqref{tau_eps_bd}, letting $\wt\kappa=\overline\kappa-\overline\kappa^{\rm lin}$, the difference $\wt\tau:=\overline\tau-\overline\tau^{\rm lin}$ satisfies 
\begin{align*}
(1+\gamma)\bigg(\wt\tau_{ss}&+\big(\wt\kappa(\overline\kappa+2\kappa_0)\big)_{ss}+ \big(\overline\kappa^{\rm lin}\wt\kappa\big)_{ss}
+\big(\wt\kappa_s(\overline\kappa+\kappa_0)\big)_s
+\big(\overline\kappa^{\rm lin}_s\wt\kappa\big)_s\bigg)\\
 &= -\wt\kappa_{ss}(\overline\kappa+\kappa_0) -\overline\kappa^{\rm lin}_{ss}\wt\kappa + (\overline\kappa+\kappa_0)^2\overline\tau +\overline\kappa(\overline\kappa+\kappa_0)^2(\overline\kappa+2\kappa_0) \,.
\end{align*}

Multiplying by $\wt\tau$ and integrating by parts, we obtain 
\begin{align*}
\norm{\wt\tau_s}_{L^2}^2&\le c\bigg(\norm{\big(\wt\kappa(\overline\kappa+2\kappa_0)\big)_s}_{L^2}+ \|\big(\overline\kappa^{\rm lin}\wt\kappa\big)_s\|_{L^2}
+\norm{\wt\kappa_s(\overline\kappa+\kappa_0)}_{L^2}
+\|\overline\kappa^{\rm lin}_s\wt\kappa\|_{L^2}\bigg)\norm{\wt\tau_s}_{L^2}\\
 & + c\bigg(\norm{\wt\kappa_s}_{L^2}\norm{(\overline\kappa+\kappa_0)_s}_{L^2} +\|\overline\kappa^{\rm lin}_s\|_{L^2}\norm{\wt\kappa}_{L^2} + \norm{\overline\kappa+\kappa_0}_{L^2}^2\norm{\overline\tau+ \overline\kappa(\overline\kappa+2\kappa_0)}_{L^\infty}\bigg)\norm{\wt\tau}_{L^\infty} \,.
\end{align*}

Using Young's inequality as well as Lemma \ref{lem:GN}, we have
\begin{align*}
\norm{\wt\tau}_{H^1}^2&\le \bigg(\norm{\wt\kappa}_{H^1}^2\big(\norm{\overline\kappa}_{H^1}^2+\|\overline\kappa^{\rm lin}\|_{H^1}^2+\norm{\kappa_0}_{H^1}^2\big)\\
&\qquad + (\norm{\overline\kappa}_{H^1}^4+\norm{\kappa_0}_{H^1}^4)\norm{\overline\tau}_{H^1}^2+ (\norm{\overline\kappa}_{H^1}^6+\norm{\kappa_0}_{H^1}^6)\norm{\overline\kappa}_{H^1}^2 \bigg) 
\le c\varepsilon^8\,,
\end{align*}
 where we have used \eqref{Tf_bound}, \eqref{kappaH1}, and \eqref{kap_eps_bd} to obtain the final inequality.
\end{proof}

Equipped with Lemmas \ref{lem:middle} and \ref{lem:lin_close}, we may now prove Theorem \ref{thm:swimming}.

\begin{proof}[Proof of Theorem \ref{thm:swimming}]
We begin by calculating the form of the swimming velocity $\bm{V}(t)=\int_0^1\frac{\p\X}{\p t}(s,t)\,ds$ stated in equation \eqref{basepoint}.
Using the formulation \eqref{classical} and recalling that $\X_{sss}=-\kappa^2\be_{\rm t}+\kappa_s\be_{\rm n}$, $(\be_{\rm t})_s=\kappa\be_{\rm n}$, and $(\be_{\rm n})_s=-\kappa\be_{\rm t}$, we have
\begin{equation}\label{swim_calc}
\begin{aligned}
\int_0^1\frac{\p\X}{\p t}(s,t)\,ds &= -\int_0^1\big({\bf I}+\gamma\be_{\rm t}\be_{\rm t}^{\rm T}\big) \big(-\kappa^2\be_{\rm t}+\kappa_s\be_{\rm n}-\tau\be_{\rm t}-(\kappa_0)_s\be_{\rm n}\big)_s \, ds \\
&= -\gamma\int_0^1 \big(-2\kappa\kappa_s-\tau_s-\kappa\kappa_s+\kappa(\kappa_0)_s\big)\be_{\rm t} \, ds \\
&= -\gamma\int_0^1 \big(-2(\kappa-\kappa_0)\kappa_s- 2\kappa_0(\kappa-\kappa_0)_s-\overline\tau_s-\kappa(\kappa-\kappa_0)_s\big)\be_{\rm t} \, ds \\ 
&= -\gamma\int_0^1 \bigg(\big(2(\kappa-\kappa_0)_s(\kappa-\kappa_0)-\overline\tau_s-\kappa(\kappa-\kappa_0)_s\big)\be_{\rm t} +2(\kappa-\kappa_0)\kappa^2\be_{\rm n}\bigg) \, ds \\ 
&= -\gamma\int_0^1 \bigg(\big((\kappa-\kappa_0)_s(\kappa-\kappa_0)-\overline\tau_s-\kappa_0(\kappa-\kappa_0)_s\big)\be_{\rm t} +2(\kappa-\kappa_0)\kappa^2\be_{\rm n}\bigg) \, ds  \,. 
\end{aligned}
\end{equation}
Here in the second line we have used the boundary conditions in \eqref{classical}, in third line we have replaced $\tau$ with $\overline\tau=\tau+\kappa_0^2$, in the fourth line we have integrated $-2(\kappa-\kappa_0)\kappa_s\be_{\rm t}$ by parts, and in the final line we have added and subtracted $\kappa_0(\kappa-\kappa_0)_s\be_{\rm t}$. \\

Then we have that the swimming velocity $\bm{V}(t)$ satisfies
\begin{align*}
\bm{V}(t) &= -\gamma\int_0^1 \bigg(\big(\overline\kappa\overline\kappa_s-\overline\tau_s-\kappa_0\overline\kappa_s\big)\be_{\rm t}(s,t) +2\overline\kappa\kappa^2\be_{\rm n}(s,t)\bigg) \, ds \\
&= -\gamma\int_0^1 \big(\overline\kappa\overline\kappa_s-\overline\tau_s-\kappa_0\overline\kappa_s\big)\be_{\rm t}(0,0) \, ds + \bm{r}_{\rm v}(t)\\
&= \gamma\int_0^1 \kappa_0\overline\kappa_s \, ds \, \be_{\rm t}(0,0) + \bm{r}_{\rm v}(t)
\end{align*}
where, using Lemma \ref{lem:middle}, we have
\begin{align*}
\abs{\bm{r}_{\rm v}(t)} &= \abs{\gamma\int_0^1 \bigg(\big(\overline\kappa\overline\kappa_s-\overline\tau_s-\kappa_0\overline\kappa_s\big)(\be_{\rm t}(s,t)-\be_{\rm t}(0,0)) +2\overline\kappa\kappa^2\be_{\rm n}(s,t)\bigg) \, ds}\\
&\le c\,\gamma \bigg(\big(\norm{\overline\kappa}_{H^1}^2+\norm{\overline\tau}_{H^1}+\norm{\kappa_0}_{H^1}\norm{\overline\kappa}_{H^1}\big)\norm{\be_{\rm t}(\cdot,t)-\be_{\rm t}(0,0)}_{L^2} +2\norm{\overline\kappa}_{H^1}\norm{\kappa}_{H^1}^2\bigg) \le c\varepsilon^3 \,.
\end{align*}
Integration by parts yields the swimming expression \eqref{swim_speed1}. \\

We now consider the equation \eqref{theta_eqn2_2} satisfied by $\overline\kappa^{\rm lin}$. We aim to use $\overline\kappa^{\rm lin}$, along with Lemma \ref{lem:lin_close}, to better understand when expression \eqref{swim_speed1} leads to nonnegligible net motion over one period in time.

Defining $\omega=\frac{2\pi}{T}$, we may first expand $\overline\kappa^{\rm lin}(s,t)$ and $\kappa_0(s,t)$ as a Fourier series in time:
\begin{equation}\label{expansion}
\kappa_0=\sum_{m=1}^\infty A_m(s)\cos(\omega m \,t)-B_m(s)\sin(\omega m \,t)\,, \quad 
\overline\kappa^{\rm lin}=\sum_{m=1}^\infty C_m(s)\cos(\omega m \,t)-D_m(s)\sin(\omega m \,t)\,.
\end{equation}
 Then, using equation \eqref{theta_eqn2_2}, the coefficients $A_m$, $B_m$, $C_m$, and $D_m$ satisfy
 \begin{align*}
 C_m &= -\frac{1}{\omega m} (D_m)_{ssss} - A_m\,,\qquad 
 D_m = \frac{1}{\omega m} (C_m)_{ssss}-B_m\,.
 \end{align*}

Further expanding $A_m$, $B_m$, $C_m$, and $D_m$ in eigenfunctions of $\mc{L}$ \eqref{L_efuns}, i.e.
\begin{align*}
A_m &= \sum_{k=1}^\infty a_{m,k} \psi_k(s)\,, \; 
B_m = \sum_{k=1}^\infty b_{m,k} \psi_k(s)\,, \; 
C_m = \sum_{k=1}^\infty c_{m,k} \psi_k(s)\,, \; 
D_m = \sum_{k=1}^\infty d_{m,k} \psi_k(s)\,,
\end{align*}
we may solve for the coefficients $c_{m,k}$, $d_{m,k}$ in terms of $a_{m,k}$, $b_{m,k}$ as 
\begin{equation}\label{coeffsCD}
c_{m,k} = \frac{\omega^2m^2}{\omega^2m^2+\lambda_k^2}\bigg(\frac{\lambda_k}{\omega m}b_{m,k}-a_{m,k}\bigg)\,, \quad
d_{m,k} = \frac{\omega^2m^2}{\omega^2m^2+\lambda_k^2}\bigg(-\frac{\lambda_k}{\omega m}a_{m,k}-b_{m,k}\bigg)\,.
\end{equation}
Here $\lambda_k=\xi_k^4$ are the eigenvalues of $\mc{L}$, given by \eqref{eigenvals}. Inserting the above expansions into the swimming expression \eqref{swim_speed1} (and replacing $\overline\kappa$ with $\overline\kappa^{\rm lin}$, by Lemma \ref{lem:lin_close}), we may write 
\begin{align*}
-\gamma&\int_0^1 \langle(\kappa_0)_s \overline\kappa^{\rm lin}\rangle \, ds  = 
-\gamma\sum_{m,k,\ell=1}^\infty\bigg(\frac{1}{2}a_{m,\ell}c_{m,k}+\frac{1}{2}b_{m,\ell}d_{m,k} \bigg)\int_0^1\psi_k(\psi_\ell)_s\,ds \\
&= \frac{\gamma}{2}\sum_{m,k,\ell=1}^\infty\frac{\omega^2 m^2}{\omega^2m^2+\lambda_k^2} \bigg(\frac{\lambda_k}{\omega m}\big(a_{m,k}b_{m,\ell}-b_{m,k}a_{m,\ell} \big) + a_{m,k}a_{m,\ell}+b_{m,k}b_{m,\ell} \bigg) \int_0^1\psi_k(\psi_\ell)_s\,ds \, .
\end{align*}
\end{proof}

\section{Observations and numerics}\label{sec:obs_num}
In this section, we begin by exploring the observations \ref{obs:obs} about the swimming speed predicted by the expression \eqref{swim_speed2}. We proceed to perform a small numerical optimization of the swimming speed subject to a fixed amount of work performed by the filament. Finally, we outline a numerical method for a reformulation of equation \eqref{classical} and verify the observations numerically.

\subsection{Observations}\label{sec:obs}
We first consider the form of the eigenfunctions $\psi_k$ of $\mc{L}$ (see \eqref{L_efuns}). Since $\psi_{2k}(s)$ is odd and $\psi_{2k-1}(s)$ is even about $s=\frac{1}{2}$ for each $k=1,2,3,\dots$, we have that eigenfunctions with the same parity do not contribute to net motion: 
\begin{equation}\label{vanish_int}
\int_0^1\psi_{2k}(\psi_{2\ell})_s\,ds =0\, ,\quad  \int_0^1\psi_{2k-1}(\psi_{2\ell-1})_s\,ds=0\,, \qquad k,\ell= 1,2,3,\dots\,.
\end{equation}

If $\kappa_0(s,t)$ is always odd about $s=\frac{1}{2}$, we may write 
\begin{align*}
\kappa_0=\sum_{m=1}^\infty \sum_{k=1}^\infty \big(a_{m,2k}\cos(\omega m \,t)+b_{m,2k}\sin(\omega m \,t)\big)\psi_{2k}(s)\,.
\end{align*}
Likewise, if $\kappa_0(s,t)$ is always even about $s=\frac{1}{2}$, we may write 
\begin{align*}
\kappa_0=\sum_{m=1}^\infty \sum_{k=1}^\infty \big(a_{m,2k-1}\cos(\omega m \,t)+b_{m,2k-1}\sin(\omega m \,t)\big)\psi_{2k-1}(s)\,.
\end{align*}
 In either of the above cases, due to \eqref{vanish_int}, we have that the leading order term of the swimming expression \eqref{swim_speed2} satisfies
 \begin{align*}
-\gamma\int_0^1 \langle(\kappa_0)_s \overline\kappa^{\rm lin}\rangle \, ds =0\,.
\end{align*}

Furthermore, writing
\begin{align*}
\kappa_0=\sum_{m=1}^\infty A_m(s)\cos(\omega m \,t)+B_m(s)\sin(\omega m \,t)\,,
\end{align*}
we see that if $A_m=\pm B_m$ or either $A_m=0$ or $B_m=0$ for all $m$, then the first term $b_{m,k}a_{m,\ell}-a_{m,k}b_{m,\ell}$ in the swimming expression \eqref{swim_speed2} vanishes. Recalling that the eigenvalues $\lambda_k=\xi_k^4$ of $\mc{L}$ satisfy $\xi_k\to\frac{(2k+1)\pi}{2}$ as $k\to\infty$ and that the smallest eigenvalue $\lambda_1\approx(4.73)^4\approx 500$ (see \eqref{L_efuns}), we note that this initial term has the greatest relative contribution to filament propulsion. (The remaining terms $a_{m,k}a_{m,\ell}+b_{m,k}b_{m,\ell}$ become relatively more important for large $m$, but the optimimzation results below prompt us to consider only $m$ small). 
In particular, if this initial term vanishes, the fiber may still have nonzero net motion over the course of one period, but we expect its displacement to be very small.\\


We may also consider optimizing the swimming speed \eqref{swim_speed2} given certain constraints on the filament deformation, such as a fixed amount of work \cite{spagnolie2010optimal,lauga2020fluid}. We begin by calculating the average work done by the swimmer over one period. Starting with the formulation \eqref{classical}, we have
\begin{align*}
W &= \int_0^1 \bigg\langle\frac{\p\X}{\p t}\cdot\big(\X_{sss}-\tau\X_s-(\kappa_0)_s\be_{\rm n} \big)_s\bigg\rangle \, ds = -\int_0^1 \bigg\langle\frac{\p\X_s}{\p t}\cdot\big(\X_{sss}-(\kappa_0)_s\be_{\rm n} \big) \bigg\rangle\, ds \\
&=-\int_0^1 \langle \dot\theta (\kappa-\kappa_0)_s \rangle \, ds = \int_0^1 \langle \dot\kappa \overline\kappa \rangle \, ds = \int_0^1 \langle\dot\kappa_0 \overline\kappa \rangle \, ds \,.
\end{align*}
Here we have used that $\overline\kappa=\kappa-\kappa_0$ is $T$-periodic. Using Lemma \ref{lem:lin_close}, we have that
\begin{align*}
|W - W^{\rm lin}| \le c\varepsilon^3\, , 
\end{align*}
where, using the expansion \eqref{expansion} and the form of the coefficients \eqref{coeffsCD}, we may write 
\begin{equation}\label{work}
\begin{aligned}
W^{\rm lin} &= \int_0^1\langle \dot\kappa_0\overline\kappa^{\rm lin} \rangle\, ds\,dt 
= \sum_{m,k,\ell=1}^\infty \frac{\omega m}{2}\bigg(b_{m,\ell}c_{m,k}-a_{m,\ell}d_{m,k}\bigg)\int_0^1\psi_k\psi_\ell \,ds \\
&= \sum_{m,k=1}^\infty \frac{\lambda_k}{2}\frac{\omega^2m^2}{\omega^2m^2+\lambda_k^2}\big(a_{m,k}^2+b_{m,k}^2\big) \,.
\end{aligned}
\end{equation}
Here we seek to numerically optimize the leading order swimming speed \eqref{swim_speed2} over a (small) finite number of spatial and temporal modes $a_{m,k}$, $b_{m,k}$, subject to a fixed amount of work \eqref{work}. For a given $m_{\max}$, $k_{\max}$, we aim to construct the preferred curvature 
\begin{align*}
\kappa_0(s) = \sum_{m=1}^{m_{\max}}\sum_{k=1}^{k_{\max}} \big(a_{m,k}\cos(\omega m\, t)-b_{m,k}\sin(\omega m\, t)\big)\psi_k(s) 
\end{align*}
which results in the fastest swimming speed.
Note that we do not necessarily expect the true optimal preferred curvature to be representable by a finite number of spatial modes $k$, as we cannot construct functions with nonhomogeneous boundary conditions using finitely many eigenfunctions $\psi_k$ of $\mc{L}$. We instead aim to visualize general trends in the emerging optimal $\kappa_0$.

We define the finite-mode swimming speed and work
\begin{align}
U_{m_{\max},k_{\max}} &:= \frac{\gamma}{2}\sum_{m=1}^{m_{\max}}\sum_{k,\ell=1}^{k_{\max}} \frac{\omega^2 m^2}{\omega^2m^2+\lambda_k^2} \bigg(\frac{\lambda_k}{\omega m}\big(a_{m,k}b_{m,\ell}-b_{m,k}a_{m,\ell} \big) \nonumber \\
&\hspace{5cm} + a_{m,k}a_{m,\ell}+b_{m,k}b_{m,\ell} \bigg) \int_0^1\psi_k(\psi_\ell)_s\,ds  \\
W_{m_{\max},k_{\max}} &:= \sum_{m=1}^{m_{\max}}\sum_{k=1}^{k_{\max}} \frac{\lambda_k}{2}\frac{\omega^2m^2}{\omega^2m^2+\lambda_k^2}\big(a_{m,k}^2+b_{m,k}^2\big)
\end{align}
and solve the following constrained optimization problem for the coefficients $a_{m,k}$, $b_{m,k}$:
\begin{equation}\label{optimization}
\begin{aligned}
\min_{a_{m,k},b_{m,k}} \quad &U_{m_{\max},k_{\max}} \\
\text{subject to }\quad &W_{m_{\max},k_{\max}} = {\rm const.}
\end{aligned}
\end{equation}
Note that we are seeking the fastest swimming speed in the negative direction with respect to the unit tangent vector $\be_{\rm t}(0,0)=\be_x$, i.e. leftward.
The optimization is performed using Matlab's \texttt{fmincon} starting from random initial $a_{m,k}$, $b_{m,k}$ satisfying the constraint, where $\text{const.}=1$.  
We use $m_{\max}=3$ and $k_{\max}=10$. For each $m=1,2,3$, the optimal temporal eigenmodes $A_{m}(s)=\sum_{k=1}^{k_{\max}}a_{m,k}\psi_k(s)$ and $B_{m}(s)=\sum_{k=1}^{k_{\max}}b_{m,k}\psi_k(s)$ are plotted in Figure \ref{subfig:optimize1}. As in \cite[Section B]{lauga2007floppy}, we find that forcing only the lowest temporal mode $m=1$ leads to the optimal swimming speed for a given amount of work. Notably, the optimal $A_1(s)$ is perfectly odd about $s=\frac{1}{2}$ while $B_1(s)$ is perfectly even. The phase difference between $A_1(s)$ and $B_1(s)$ is consistent with the classical optimization paper \cite{pironneau1974optimal} which finds that, to reach a given swimming speed while minimizing energy expenditure, the filament should deform as a traveling wave.  

\begin{figure}[!ht]
\centering
	\begin{subfigure}[b]{0.32\textwidth}
		\includegraphics[scale=0.14]{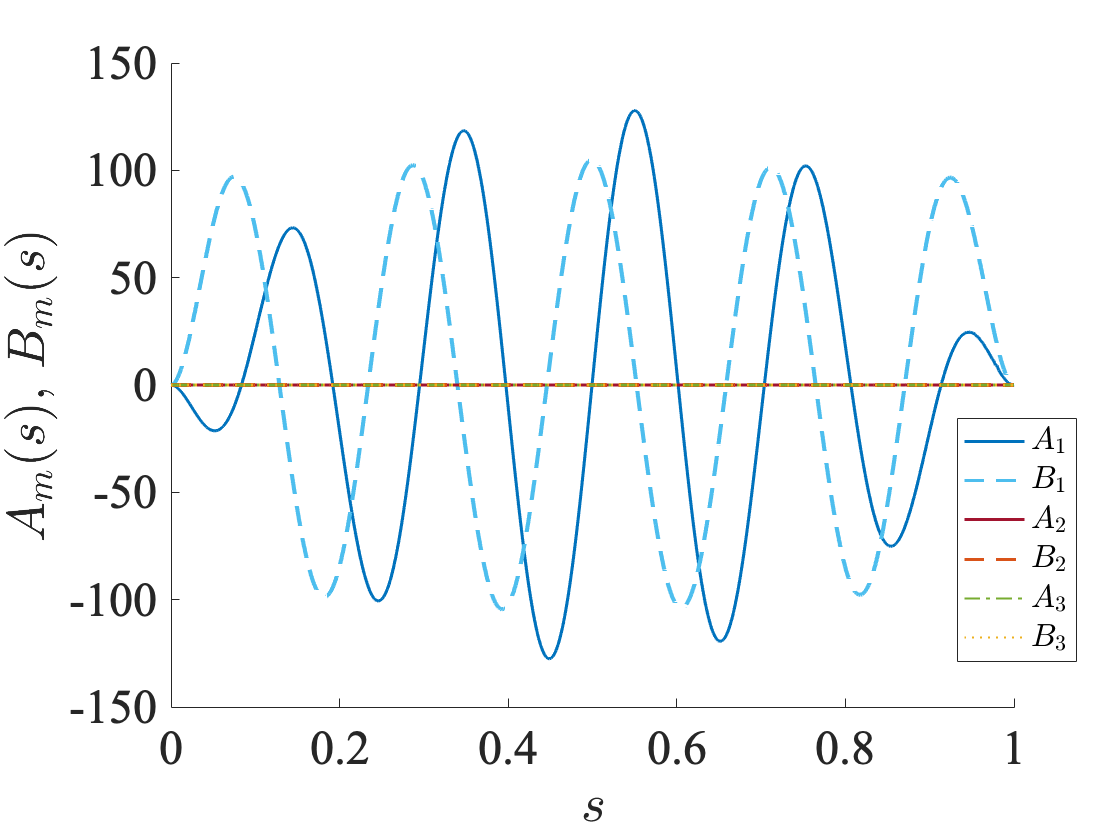}
		\caption{}
		\label{subfig:optimize1}
	\end{subfigure}
		\begin{subfigure}[b]{0.32\textwidth}
		\includegraphics[scale=0.14]{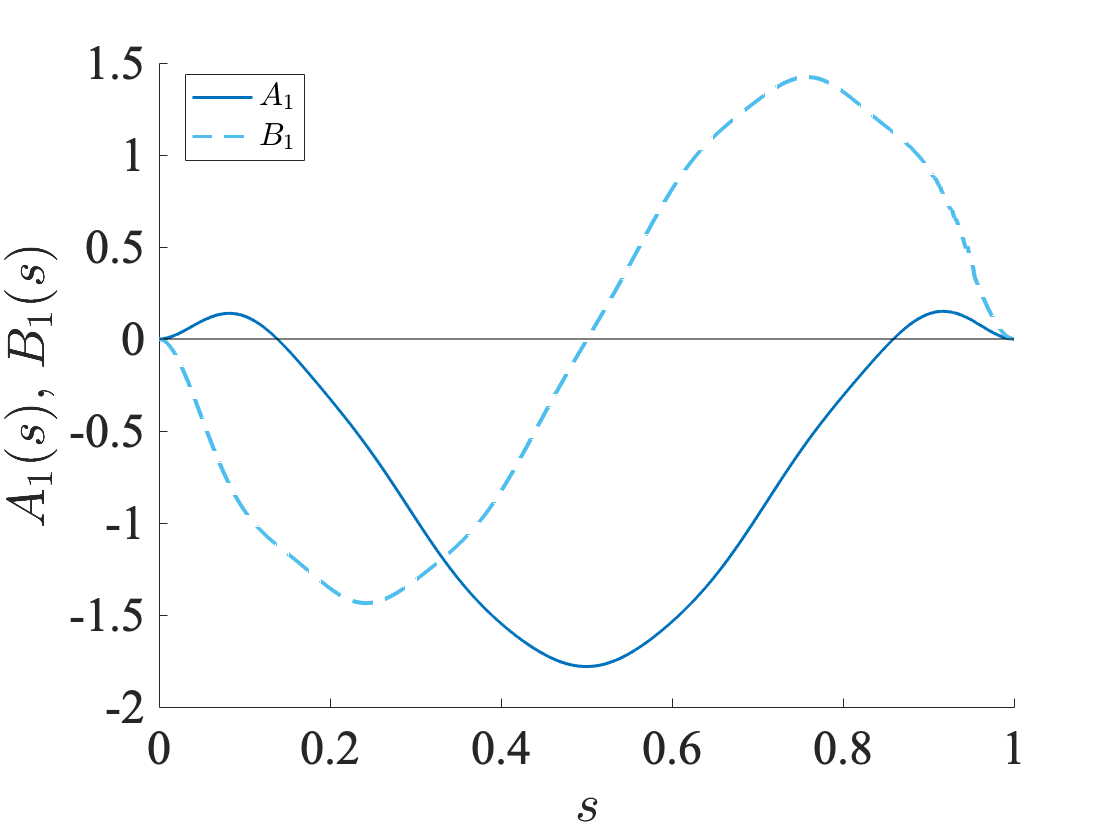}
		\caption{}
		\label{subfig:optimize2}
	\end{subfigure}
		\begin{subfigure}[b]{0.32\textwidth}
		\includegraphics[scale=0.14]{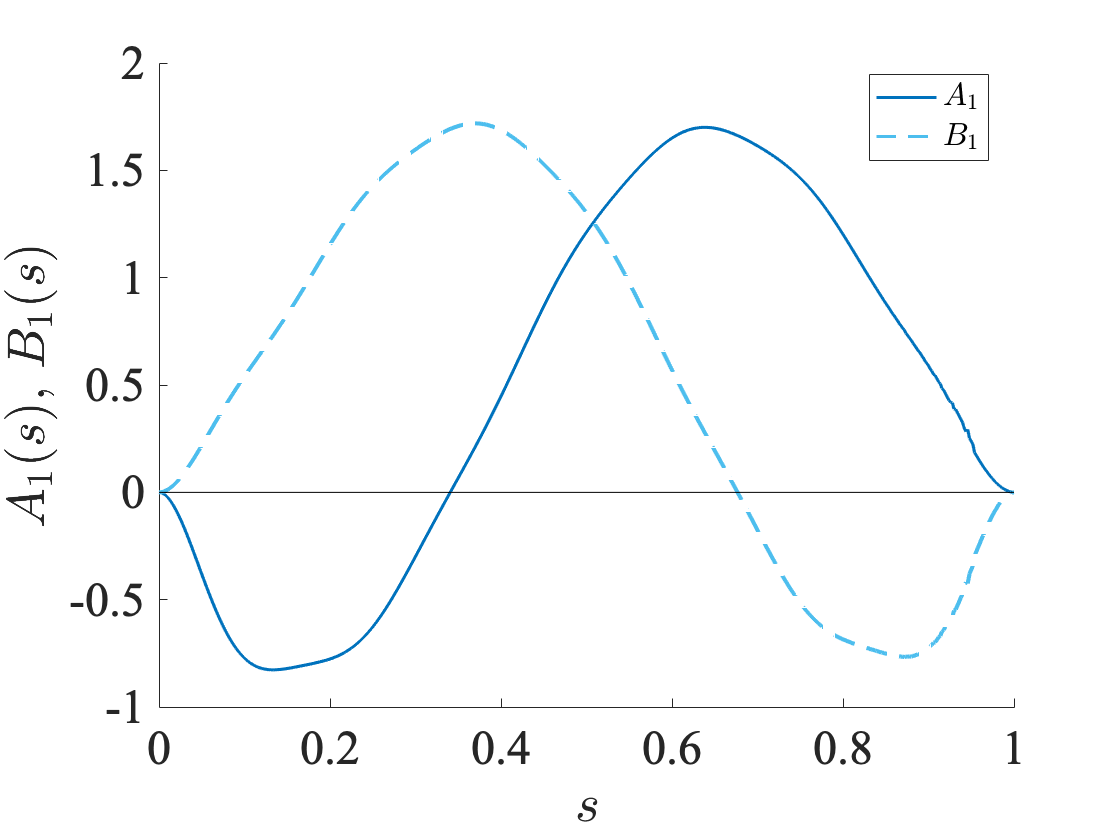}
		\caption{}
		\label{subfig:optimize3}
	\end{subfigure}
\caption{(A) For the optimization problem \eqref{optimization}, the optimal temporal eigenmodes $A_m(s)=\sum_{k=1}^{k_{\max}}a_{m,k}\psi_k(s)$ and $B_{m}(s)=\sum_{k=1}^{k_{\max}}b_{m,k}\psi_k(s)$ for $k_{\max}=10$ are plotted for each $m=1,2,3$. Forcing only the lowest temporal mode $m=1$ leads to the fastest swimming speed for a given amount of work. 
(B)--(C) When the bending energy $\norm{\kappa_0}_{L^2}$ is also constrained (see \eqref{optimization2}), the optimal eigenmodes $A_1$ and $B_1$ are far less oscillatory. Here we plot two different optima for $k_{\max}=12$ which both yield the same swimming speed of $-0.1008$.  }
\label{fig:optimize}
\end{figure}

Note that for small $m$, the expression for average work \eqref{work} is essentially the $H^{-4}$ norm of $\kappa_0$ along the length of the filament, due to the $\sim\lambda_k^{-1}$ coefficient of the summand. Thus, given a fixed amount of work, we expect the optimal $\kappa_0$ to oscillate at the highest allowable frequency in $s$, as seen in Figure \ref{subfig:optimize1}. Constraining only the work done by the filament thus may not lead to the most physically realistic optimum. As noted in \cite{lighthill1975mathematical,spagnolie2010optimal}, the bending energy of the filament can also be taken into account.

Thus we also consider the additional constraint $\norm{\kappa_0}_{L^2}=$\,const., and solve the following optimization problem for the coefficients $a_{m,k}$, $b_{m,k}$: 
\begin{equation}\label{optimization2}
\begin{aligned}
\min_{a_{m,k},b_{m,k}} \quad &U_{m_{\max},k_{\max}} \\
\text{subject to }\quad &W_{m_{\max},k_{\max}} = 1\\
 & \sum_{m,k}a_{m,k}^2=\sum_{m,k}b_{m,k}^2 = 1\,.
\end{aligned}
\end{equation}
Due to the results of optimization problem \eqref{optimization}, we take $m_{\max}=1$ and optimize only over spatial modes $k$. Two different results of the optimization for $k_{\max}=12$ are plotted in Figures \ref{subfig:optimize2} and \ref{subfig:optimize3}. Both yield the same swimming speed $U_{m_{\max},k_{\max}}= -0.1008$.

Case (B) is similar to a simple sine and cosine, i.e. perhaps the most familiar traveling wave example. However, since $A_1$ and $B_1$ are constructed of finitely many eigenfunctions $\psi_k(s)$ (see \eqref{L_efuns}), they are constrained to have homogeneous boundary conditions. We also note the striking symmetry in the solutions of case (C).   

In the next section we try swimming using both of the optimal curvatures constructed in Figures \ref{subfig:optimize2} and \ref{subfig:optimize3} in order to visualize their performance. 
We also explore other forms of $\kappa_0(s,t)$ to verify the observations about the swimming expression \eqref{swim_speed2} noted above.


\subsection{Numerical method}\label{sec:numerics}
Here we numerically verify the swimming predictions of the previous section. We develop a simple numerical method for simulating the motion of an inextensible filament via resistive force theory based on a combination of the methods of \cite{maxian2021integral} and \cite{moreau2018asymptotic}. We apply the method to the planar swimmer studied in the previous sections, but note that the formulation itself is applicable more generally.  

While the curvature evolution equations \eqref{theta_eqn}-\eqref{BCs} in terms of $\kappa$ and $\tau$ are convenient for analysis, a direct numerical implementation can be cumbersome due in part to complicated dependence on the forcing $\kappa_0$ as well as the need to solve for the actual fiber movement separately \eqref{fiber_move}. Numerically, however, we can make use of reformulations that may simplify the computation. As noted, our future goal is to do numerical analysis comparing the many existing formulations of filament inextensibility.  

To introduce the numerical method, we return to the original formulation of inextensible fiber dynamics \eqref{classical}, which we reiterate here for convenience:
\begin{equation}\label{original_eqn}
\begin{aligned}
\frac{\p\X}{\p t}(s,t) &= -\big({\bf I}+\gamma\X_s\X_s^{\rm T}\big)\big(\X_{sss}-\tau\X_s-(\kappa_0)_s\be_{\rm n}\big)_s \\
\abs{\X_s}^2&=1 \\
(\X_{ss}-\kappa_0\be_{\rm n})\big|_{s=0,1}&=0\,, \quad (\X_{sss}-\tau\X_s-(\kappa_0)_s\be_{\rm n})\big|_{s=0,1}=0\,.
\end{aligned}
\end{equation}

We enforce inextensibility directly by parameterizing $\X(s,t)$ as
\begin{equation}\label{Xparam}
\X(s,t)=\X_0(t) + \int_0^s\be_{\rm t}(s',t)ds' \, ,
\end{equation}
where $\be_{\rm t}\in S^2$; in particular, $\dot\be_{\rm t}\cdot\be_{\rm t}=0$. 

Along the fiber we may rewrite \eqref{original_eqn} as the coupled system
\begin{align}
\dot \X_0 +\int_0^s \dot\be_{\rm t}(s')\,ds' &= -({\bf I}+\gamma\be_{\rm t}\be_{\rm t}^{\rm T})\bm{h}(s) \label{form1} \\
\big({\bf I}-\be_{\rm t}(s)\be_{\rm t}(s)^{\rm T}\big)\int_0^s\bm{h}(s')ds' &= \big({\bf I}-\be_{\rm t}(s)\be_{\rm t}(s)^{\rm T}\big)\big(\X_{sss}-(\kappa_0)_s\be_{\rm n}\big)\,, \label{form2}
\end{align}
where $\bm{h}(s)$ is the hydrodynamic force density along the filament, and only the normal components of the integrated force density \eqref{form2} are specified. To satisfy the boundary conditions of \eqref{original_eqn}, we must also ensure that the total hydrodynamic force and torque along the fiber vanish: 
\begin{equation}\label{FandT}
\int_0^1\bm{h}(s) \,ds=0, \qquad \int_0^1(\X(s)-\X(0))\times\bm{h}(s) \,ds=0\,.
\end{equation}

The formulation \eqref{form1}-\eqref{form2}, combined with total force vanishing \eqref{FandT}, is equivalent to the formulation in \cite{maxian2021integral}, which uses slender body theory rather than resistive force theory in equation \eqref{form1}. As noted in \cite{maxian2021integral}, equations \eqref{form1}-\eqref{form2} along with \eqref{FandT} are enough to recover the rate of work exerted by the filament on the fluid in \eqref{original_eqn}. In particular, the rate of work may be calculated using \eqref{form1}-\eqref{form2} as 
\begin{align*}
\dot W = \int_0^1 \bm{h}(s)\cdot \bigg(\dot\X_0 +\int_0^s \dot\be_{\rm t}(s')\,ds' \bigg) ds &= -\int_0^1\bigg(\int_0^s \bm{h}(s')ds'\bigg)\cdot \dot\be_{\rm t}(s) \, ds \\
&= -\int_0^1(\X_{sss}-(\kappa_0)_s\be_{\rm n})\cdot \dot\be_{\rm t}(s) \, ds\,,
\end{align*}
where we have integrated by parts and used that $\be_{\rm t}\cdot\dot\be_{\rm t}=0$. Starting instead from \eqref{original_eqn} and again integrating by parts and using $\X_s=\be_{\rm t}\perp \dot\be_{\rm t}$, we also obtain
\begin{align*}
\dot W = \int_0^1 \big(\X_{sss}-\tau\X_s-(\kappa_0)_s\be_{\rm n}\big)_s \cdot \frac{\p\X}{\p t}\, ds = -\int_0^1 \big(\X_{sss}-(\kappa_0)_s\be_{\rm n}\big) \cdot \dot\be_{\rm t}\, ds\,.
\end{align*}

Since we are using resistive force theory, equation \eqref{form1} may be directly inverted to solve for $\bm{h}$:
\begin{equation}\label{h_solve}
\bm{h} = -({\bf I}-\frac{\gamma}{1+\gamma}\be_{\rm t}\be_{\rm t}^{\rm T})\bigg(\dot \X_0 +\int_0^s \dot\be_{\rm t}(s')\,ds'\bigg)\,,
\end{equation}
and this expression may be inserted into \eqref{form2} to condense \eqref{form1} and \eqref{form2} into a single equation along the filament. For a planar fiber, recalling that we may parameterize $\be_{\rm t}=(\cos\theta,\sin\theta)^{\rm T}$, $\be_{\rm n}=(-\sin\theta,\cos\theta)^{\rm T}$, and that $\X_{sss}=-(\theta_s)^2\be_{\rm t}+\theta_{ss}\be_{\rm n}$, we may rewrite \eqref{form1} and \eqref{form2} as
\begin{equation}\label{ourform}
\be_{\rm n}(s,t)\cdot\int_0^s({\bf I}-\frac{\gamma}{1+\gamma}\be_{\rm t}\be_{\rm t}^{\rm T})\bigg(\dot\X_0 +\int_0^{s'} \dot\be_{\rm t}(\bars)\,d\bars\bigg)ds' = -\theta_{ss}+(\kappa_0)_s\,.
\end{equation}
Equation \eqref{ourform} will be the main evolution equation for the fiber, for which the unknowns are the basepoint $\X_0$ and the tangent angle $\theta(s)$. We combine \eqref{ourform} with the total force balance \eqref{FandT} to enforce the boundary condition $(-\theta_{ss}+(\kappa_0)_s)\big|_{s=1}=0$:
\begin{equation}\label{ourform2}
\int_0^1({\bf I}-\frac{\gamma}{1+\gamma}\be_{\rm t}\be_{\rm t}^{\rm T})\bigg(\dot\X_0 +\int_0^s \dot\be_{\rm t}(s')\,ds'\bigg) \,ds =0.
\end{equation}
The torque condition in \eqref{FandT} leads to the boundary condition $(\theta_s-\kappa_0)\big|_{s=0,1}=0$, which will be enforced directly in the discretization of $\theta_{ss}$ on the right hand side of \eqref{ourform}.  \\

The fiber is discretized into $N$ segments between $N+1$ points $\X_i$ along the fiber, and $\theta_i$, $i=1,\dots,N$, is taken to be the angle between segment $i$ and the $x$-axis (see Figure \ref{fig:disc}). The equation \eqref{ourform} is enforced at the midpoint of each segment at $\X_{i-\frac{1}{2}}:=\frac{\X_{i-1}+\X_i}{2}$, $i=1,\dots,N$. 

\begin{figure}[!ht]
\centering
\includegraphics[scale=0.8]{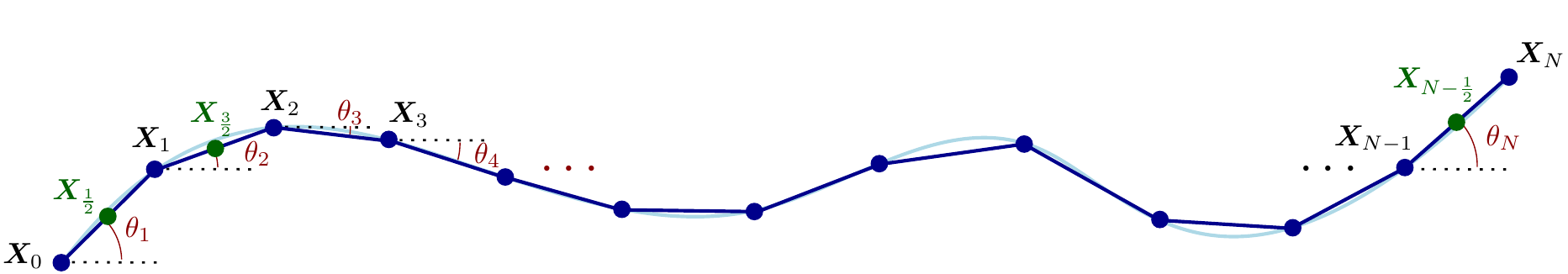}
\caption{
	Filament discretization used to implement \eqref{our_disc0}--\eqref{our_disc2}. 
}
\label{fig:disc}
\end{figure}

As mentioned, the boundary condition $(\theta_s-\kappa_0)\big|_{s=0,1}=0$ is enforced in the discretization of $\theta_{ss}$ on the right hand side of \eqref{ourform} via the approximation 
\begin{equation}\label{thetass_approx}
\theta_{ss}\big|_{s=0}\approx N^2(2\theta_2 -2\theta_1) - 2N\kappa_{0,1}\,,\qquad 
\theta_{ss}\big|_{s=1}\approx N^2(2\theta_{N-1} -2\theta_N) + 2N\kappa_{0,N}\, .
\end{equation}

At the $s=1$ endpoint, equations \eqref{ourform} and \eqref{ourform2} coincide to yield the boundary condition $\theta_{ss}\big|_{s=1}=(\kappa_0)_s\big|_{s=1}$, which, using \eqref{thetass_approx}, yields an equation for $\theta_N$: 
\begin{equation}\label{our_disc0}
\theta_N = \theta_{N-1}+ \frac{1}{N}\kappa_{0,N} - \frac{1}{2N^2}(\kappa_0)_{s,N} \,.  
\end{equation}

 The full discretized system is then given by 
\begin{align}
\frac{1}{N}\begin{pmatrix}
-\sin\theta_j \\
\cos\theta_j
\end{pmatrix}\cdot\sum_{i=1}^j\bm{M}_{\rm RFT}(\theta_i)\dot{\X}_{i-\frac{1}{2}} &= - N^2(\theta_{j-1}-2\theta_j+\theta_{j+1}) + (\kappa_0)_{s,j} \,, \quad j=2,\dots,N-1   \label{our_disc1}\\
\frac{1}{N}\begin{pmatrix}
-\sin\theta_1 \\
\cos\theta_1
\end{pmatrix}\cdot\bm{M}_{\rm RFT}(\theta_1)\dot{\X}_{\frac{1}{2}} &= -N^2(2\theta_2 -2\theta_1) + 2N\kappa_{0,1}    \label{our_disc1_1}\\
\frac{1}{N}\sum_{i=1}^{N-1}\bm{M}_{\rm RFT}(\theta_i)\dot{\X}_{i-\frac{1}{2}} &=0 \,,  \label{our_disc2}
\end{align}
along with equation \eqref{our_disc0}.
Here the evolution $\dot{\X}_{i-\frac{1}{2}}$ of each segment midpoint is parameterized as 
\begin{align*}
\dot{\X}_{i-\frac{1}{2}} = \begin{pmatrix}
\dot x_0\\
\dot y_0
\end{pmatrix} 
+ \frac{1}{2N}\begin{pmatrix}
-\sin\theta_i\\
\cos\theta_i
\end{pmatrix}\dot\theta_i 
+\frac{1}{N}\sum_{k=1}^{i} \begin{pmatrix}
-\sin\theta_k\\
\cos\theta_k
\end{pmatrix}\dot\theta_k\,, \quad i=1,\dots,N\,,
\end{align*}
and the $2N\times2N$ matrix $\bm{M}_{\rm RFT}(\theta_i)$ is given by 
\begin{align*}
\bm{M}_{\rm RFT}(\theta_i) = 
\begin{pmatrix}
1-\frac{\gamma}{1+\gamma}\cos^2\theta_i & -\frac{\gamma}{1+\gamma}\cos\theta_i\sin\theta_i \\
-\frac{\gamma}{1+\gamma}\cos\theta_i\sin\theta_i & 1-\frac{\gamma}{1+\gamma}\sin^2\theta_i
\end{pmatrix}\,.
\end{align*}
The expressions \eqref{our_disc1} and \eqref{our_disc1_1} include $N-1$ equations enforced at $\X_{\frac{1}{2}},\dots,\X_{N-\frac{3}{2}}$; the expression \eqref{our_disc2} involves two equations enforced at $\X_{N-\frac{1}{2}}$; and the equation \eqref{our_disc0} gives one additional equation to uniquely determine the $N+2$ unknowns $x_0$, $y_0$, $\theta_1$, \dots, $\theta_N$. We use a standard ODE solver to evolve \eqref{our_disc0}--\eqref{our_disc2} in time. \\

The formulation \eqref{our_disc0}--\eqref{our_disc2} is very similar to that of Moreau et al. \cite{moreau2018asymptotic}, which also avoids the need to solve for the unknown tension by projecting away from the tangential direction along the fiber. At the continuous level, the formulation in \cite{moreau2018asymptotic} is essentially based on the identity 
\begin{equation}\label{cross_ID}
\X_s\times(\X_{sss}-\tau(s)\X_s-(\kappa_0)_s\be_{\rm n}) = \big(\X_s\times(\X_{ss}-\kappa_0\be_{\rm n})\big)_s\,.
\end{equation} 

Integrating \eqref{cross_ID} with respect to arclength from $s$ to 1 ($0\le s\le 1$) and using the boundary conditions in \eqref{original_eqn}, the identity \eqref{cross_ID} yields
\begin{equation}\label{ID_result}
\begin{aligned}
\X_s(s)\times(\X_{ss}(s)-\kappa_0(s)\be_{\rm n}) &= -\int_s^1 \X_{s'}(s')\times\big(\X_{sss}(s')-\tau(s')\X_{s}(s')-\kappa_0(s')\be_{\rm n}\big)\,ds' \\
&= \int_s^1 \big(\X(s')-\X(s)\big)\times \big(\X_{sss}-\tau\X_{s}-\kappa_0\be_{\rm n}\big)_{s'} \,ds' \,.
\end{aligned}
\end{equation}

In the case of a planar filament parameterized by the tangent angle $\theta$,  we have $\X_s\times\X_{ss}=\theta_s\be_z$. Combining \eqref{ID_result} with the definition of $\bm{h}$ \eqref{h_solve}, we obtain (in the $\be_z$ direction): 
\begin{equation}\label{moreau}
\begin{aligned}
\theta_s(s) - \kappa_0(s) &= \int_s^1\big(\X(s')-\X(s)\big)\times \bm{h}(s') \,ds' \\
&= -\int_s^1\big(\X(s')-\X(s)\big)\times\bigg( ({\bf I}-\frac{\gamma}{1+\gamma}\be_{\rm t}\be_{\rm t}^{\rm T})\bigg(\dot\X_0 +\int_0^{s'} \dot\be_{\rm t}(\bars)\,d\bars\bigg) \bigg) \,ds'\,.
\end{aligned}
\end{equation}
The equation \eqref{moreau} is also enforced at each of the $N$ midpoints $\X_{i-\frac{1}{2}}$ along the discretized fiber (see Figure \ref{fig:disc}). To close the discretized system, equation \eqref{moreau} is accompanied by the total force constraint 
\begin{equation}\label{Fbalance0}
\int_0^1({\bf I}-\frac{\gamma}{1+\gamma}\be_{\rm t}\be_{\rm t}^{\rm T})\bigg(\dot\X_0 +\int_0^s \dot\be_{\rm t}(s')\,ds'\bigg) \, ds =0\,,
\end{equation}
which gives two additional equations. \\

In Appendix \ref{app:num}, we validate the method \eqref{our_disc0}--\eqref{our_disc2} against a direct discretization of the classical expression \eqref{original_eqn}. A more thorough comparison of different numerical implementations of fiber inextensibility is the intended subject of future work. 

\subsection{Swimmer numerics}\label{subsec:swim_num}
We use the numerical method \eqref{our_disc0}--\eqref{our_disc2} to test the swimming predictions from Section \ref{sec:obs}. In each simulation, the initial condition is a straight fiber lying on the $x$-axis from $x=0$ to $x=1$. We discretize the fiber into $N=100$ segments and prescribe a preferred curvature of the form 
\begin{align*}
\kappa_0(s,t) = F_1(s)\cos(\omega\,t)+F_2(s)\sin(\omega\,t)\,, \qquad \norm{F_1}_{L^2(I)}=\norm{F_2}_{L^2(I)}=1\,.
\end{align*}
Temporal forcing is only at the lowest mode, in accordance with the conclusions of the previous section. We take $\gamma=1$ and $\omega=2\pi$, and consider different waveforms $F_1(s)$ and $F_2(s)$. \\

{\bf Non-swimmers:}
We begin by considering combinations of $F_1$ and $F_2$ which are predicted to produce no net motion over each time period. In accordance with the observations of Section \ref{sec:obs}, we consider examples where $F_1$ and $F_2$ are either both even or both odd about $s=\frac{1}{2}$. Specifically, we take
\begin{enumerate}
\item  $F_1 = \cos(4\pi s)$, $F_2 = \cos(2\pi s)$ \; (both even about $s=\frac{1}{2}$) \\
\item  $F_1 = \sin(4\pi s)$, $F_2 = \sin(2\pi s)$ \; (both odd about $s=\frac{1}{2}$)
\end{enumerate}
and normalize both such that $\norm{F_1}_{L^2}=\norm{F_2}_{L^2}=1$. The two combinations of $F_1$ and $F_2$ are plotted in Figure \ref{fig:nonswimmers}, along with snapshots of the spatial positions of both swimmer (1) and swimmer (2) over the course of 50 time units. As predicted, neither swimmer exhibits net displacement in the $x$-direction. 

\begin{figure}[!ht]
\centering
	\begin{subfigure}[b]{0.32\textwidth}
		\includegraphics[scale=0.14]{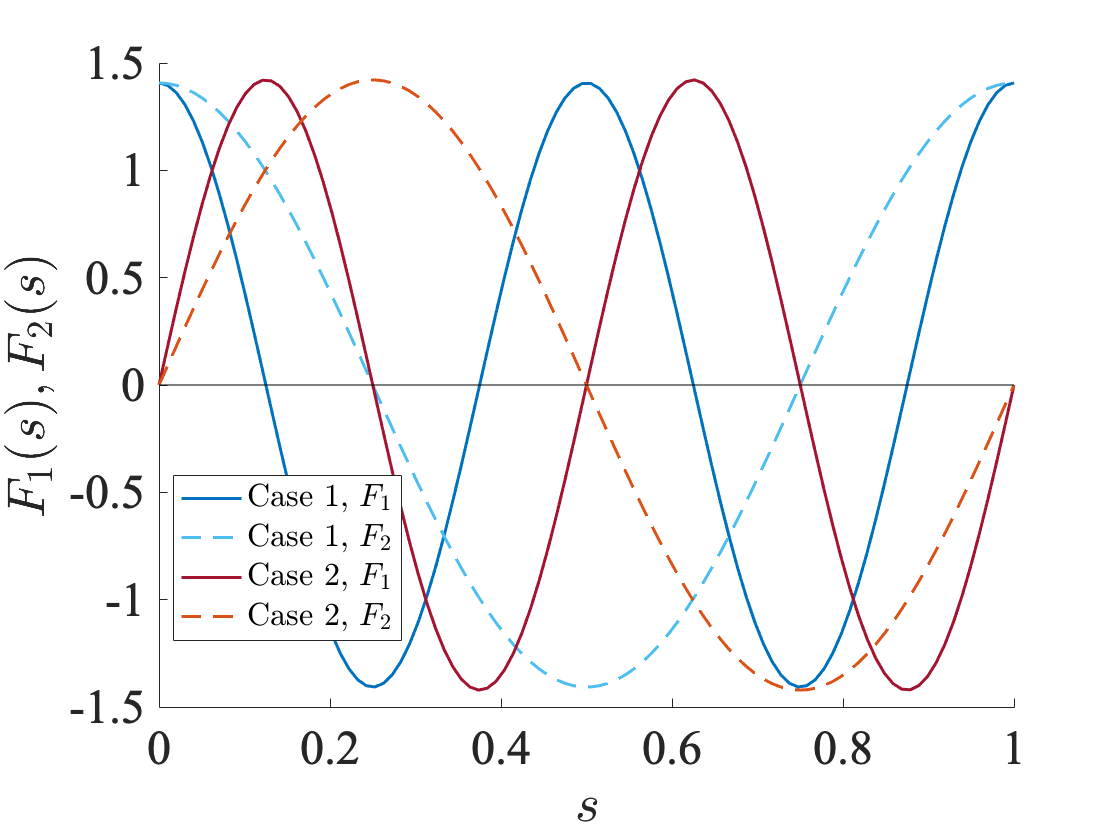}
		\caption{}
	\end{subfigure}
	\begin{subfigure}[b]{0.32\textwidth}
		\includegraphics[scale=0.14]{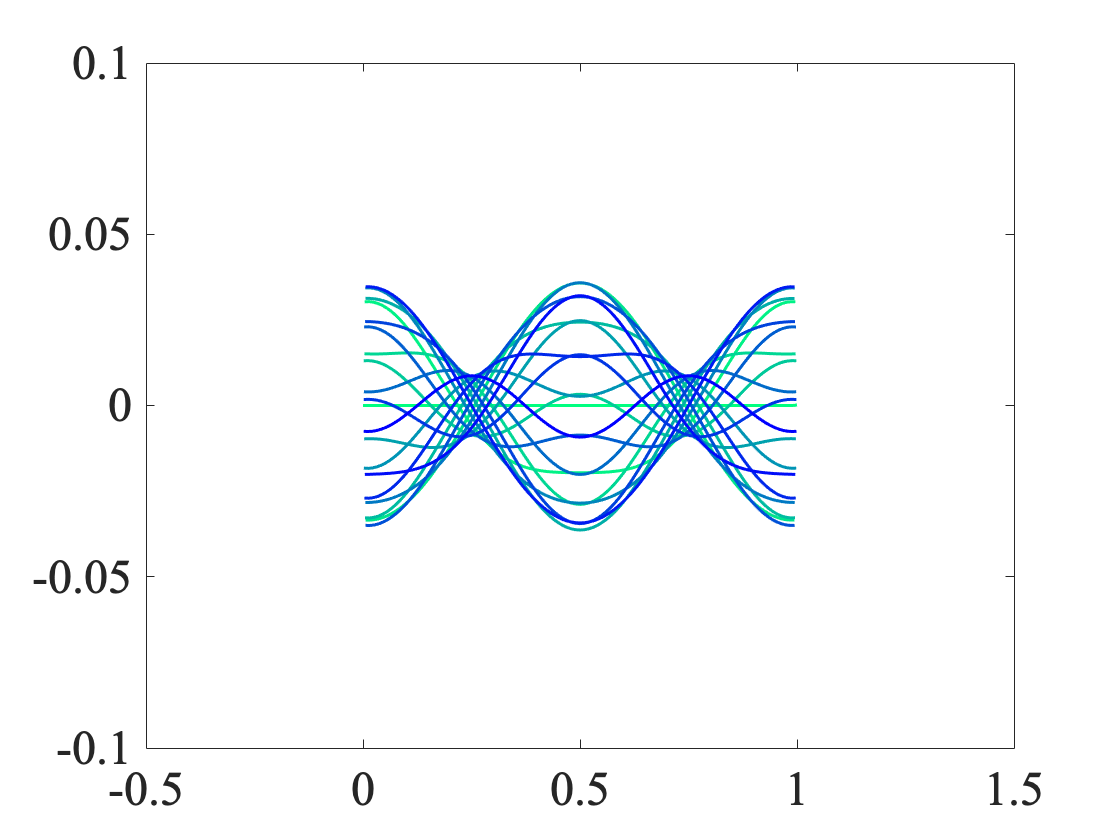}
		\caption{}
	\end{subfigure}
	\begin{subfigure}[b]{0.32\textwidth}
		\includegraphics[scale=0.14]{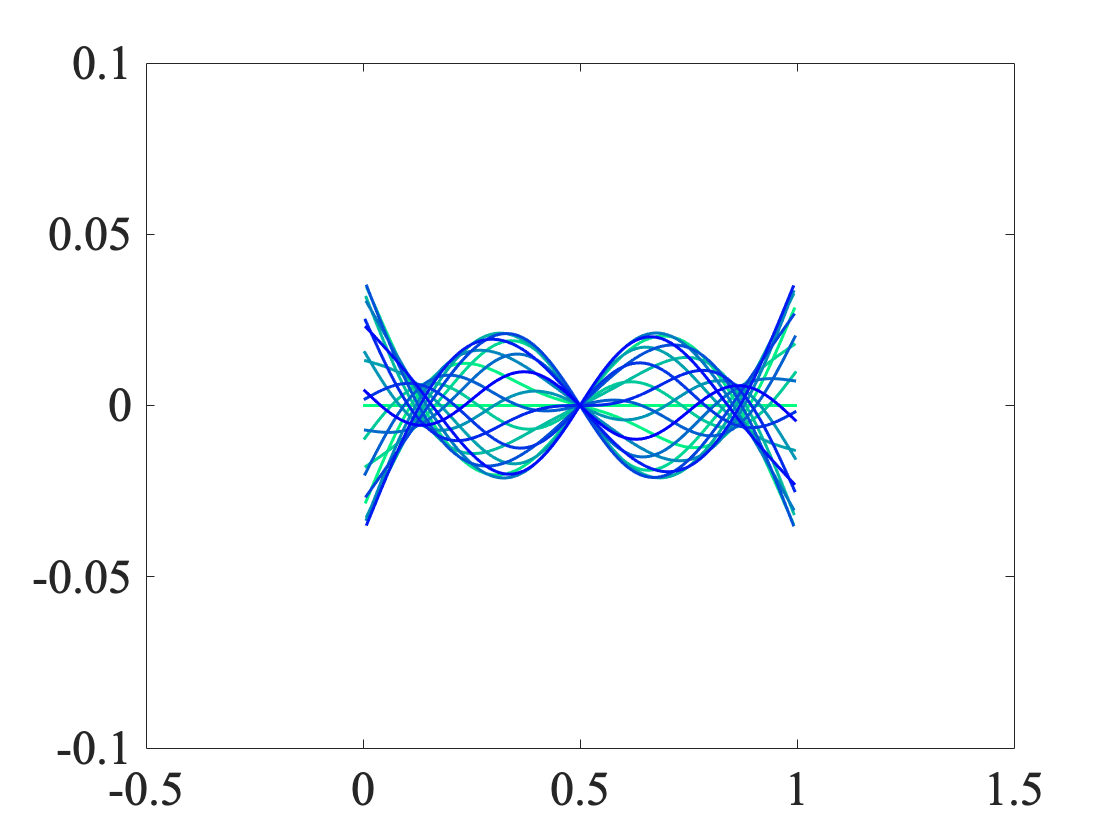}
		\caption{}
	\end{subfigure}
	\includegraphics[scale=0.2]{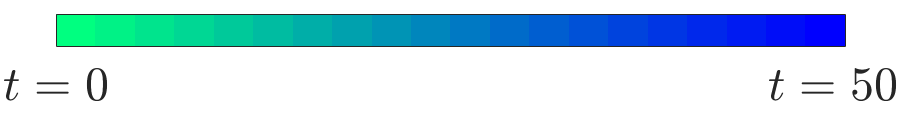}
\caption{(A) Plot of the two combinations of non-swimming $F_1$ and $F_2$ considered here. (B) Snapshots of the swimmer position in case (1) over the course of 50 time units. (C) Snapshots of the swimmer position in case (2) over the course of 50 time units. }
\label{fig:nonswimmers}
\end{figure}

{\bf Bad swimmers:} 
We next consider swimmers predicted in Section \ref{sec:obs} to possibly swim, but very poorly. We choose $F_2(s)$ to be neither even nor odd about $s=\frac{1}{2}$, and consider $F_1=\pm F_2$ and $F_1=0$:
\begin{enumerate}
\item[(3)] $F_1 = \cos(2\pi s)+ \sin(2\pi s)$, $F_2=F_1$ \\
\item[(4)] $F_1=0$, $F_2 = \cos(2\pi s)+ \sin(2\pi s)$ \\
\item[(5)] $F_1 = \cos(2\pi s)+ \sin(2\pi s)$, $F_2=-F_1$\,.
\end{enumerate}
As before, we normalize $\norm{F_2}_{L^2}=1$. The three combinations of $F_1$ and $F_2$ are plotted in Figure \ref{fig:badswimmers}, along with snapshots of their corresponding fiber positions over the course of 50 time units. Note that all three exhibit a very small rightward displacement.
\begin{figure}[!ht]
\centering
	\begin{subfigure}[b]{0.32\textwidth}
		\includegraphics[scale=0.14]{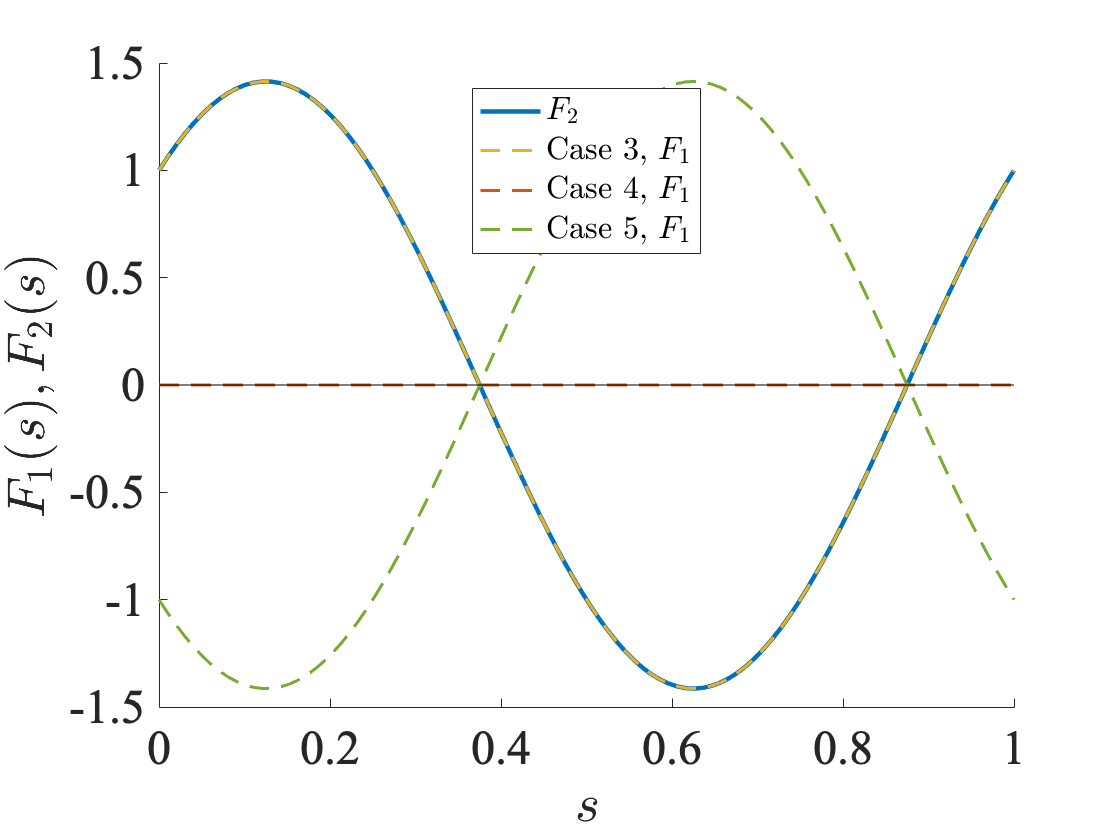}
		\caption{}
	\end{subfigure}
	\begin{subfigure}[b]{0.32\textwidth}
		\includegraphics[scale=0.14]{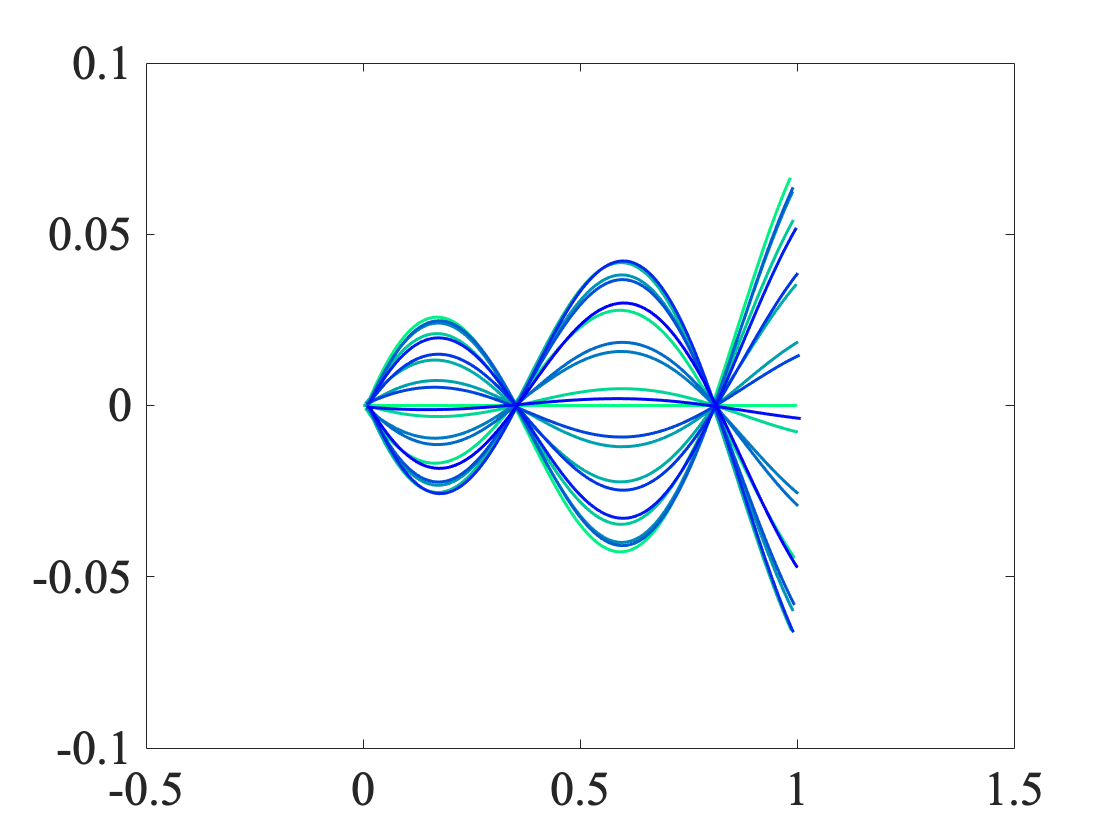}
		\caption{}
	\end{subfigure} \\
	\begin{subfigure}[b]{0.32\textwidth}
		\includegraphics[scale=0.14]{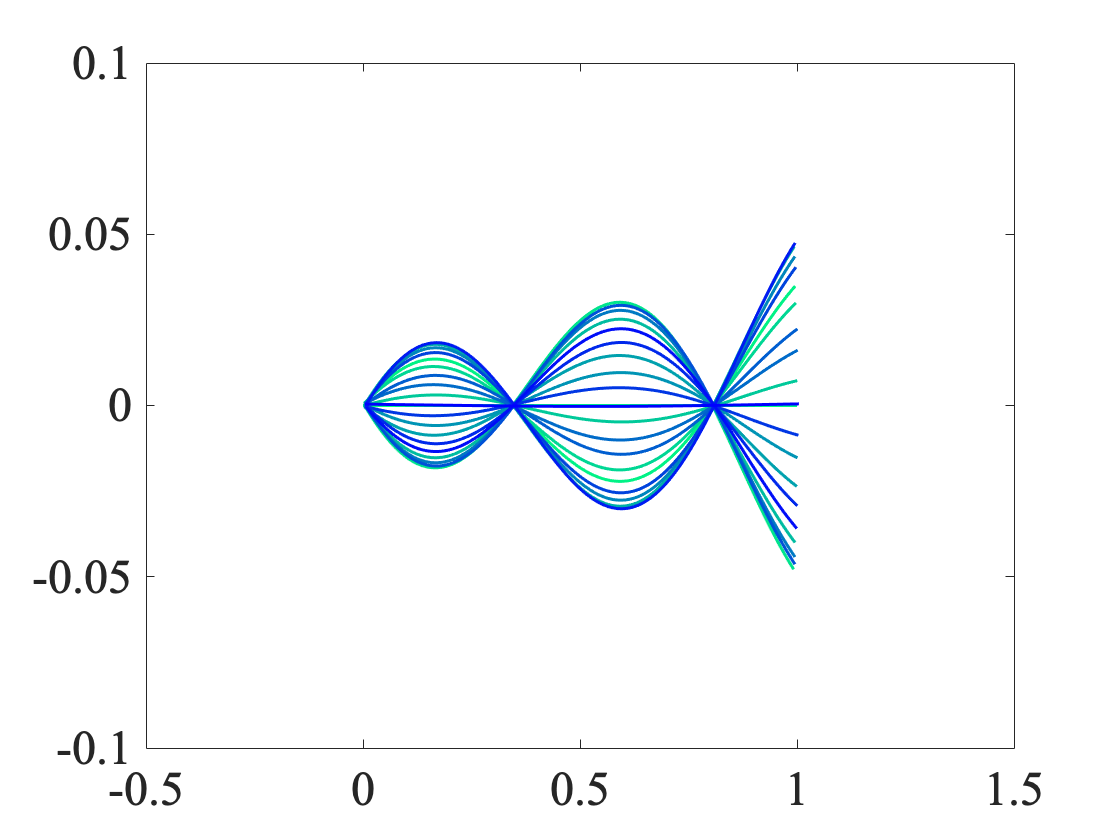}
		\caption{}
	\end{subfigure}
		\begin{subfigure}[b]{0.32\textwidth}
		\includegraphics[scale=0.14]{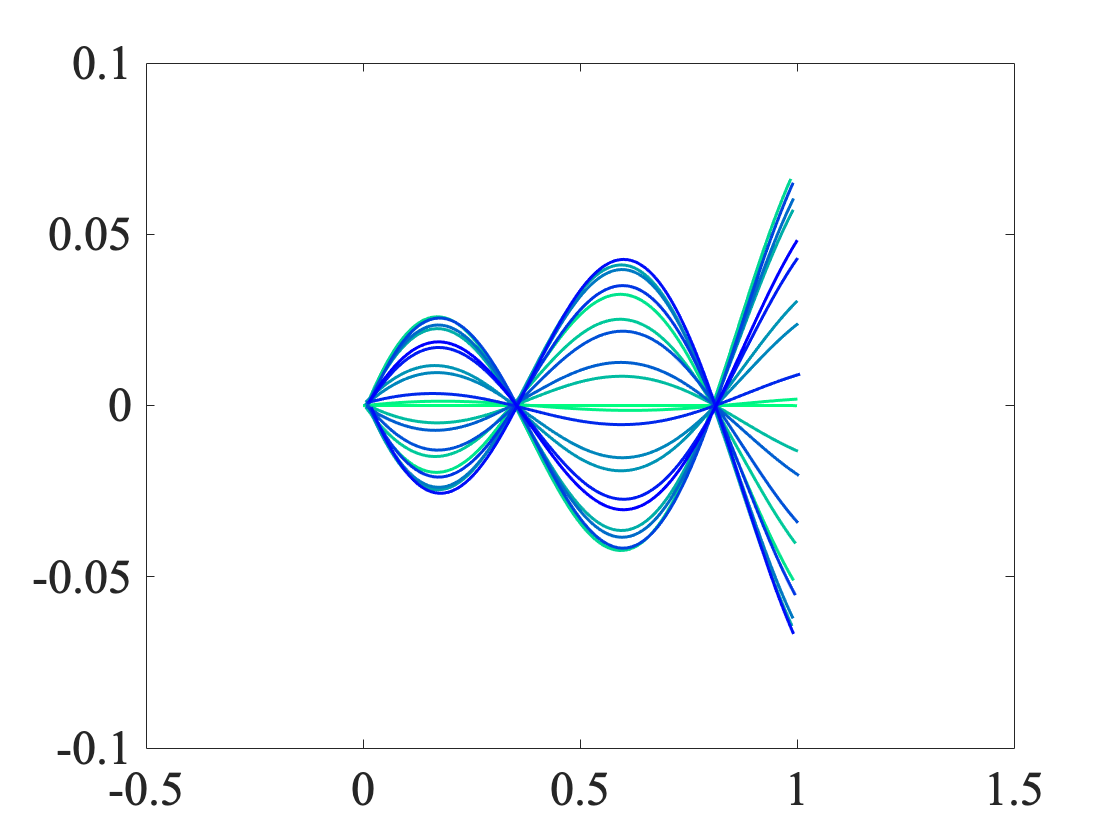}
		\caption{}
	\end{subfigure}
	\includegraphics[scale=0.2]{colorbar.png}
\caption{(A) Plot of $F_1$ and $F_2$ for the three cases of bad swimmers considered here. (B)--(D) Snapshots of the swimmer position in cases (3)--(5), respectively, over the course of 50 time units. }
\label{fig:badswimmers}
\end{figure}

{\bf Good swimmers:}
Finally, we consider four examples of swimmers which do exhibit nonnegligible net motion, in accordance with the observations of Section \ref{sec:obs}. We consider the following four combinations of $F_1$ and $F_2$:
\begin{enumerate}
\item[(6)] $F_1 = \cos(2\pi s)$, $F_2 = \sin(2\pi s)$ \\
\item[(7)] $F_1 = s^2$, $F_2 = (s-1)^2$ \\
\item[(8)] $F_1=A_1(s)$, $F_2=B_1(s)$ from Figure \ref{subfig:optimize2} \\
\item[(9)] $F_1=A_1(s)$, $F_2=B_1(s)$ from Figure \ref{subfig:optimize3} \,.
\end{enumerate}
In each case, $F_1$ and $F_2$ are normalized to have $\norm{F_1}_{L^2}=\norm{F_2}_{L^2}=1$. Snapshots of each swimmer's position over 50 time units are plotted in Figure \ref{fig:goodswimmers}. Notably, all four swimmers swim relatively well, but the two optimal swimmers from Figure \ref{fig:optimize} swim the farthest. 

\begin{figure}[!ht]
\centering
	\begin{subfigure}[b]{0.32\textwidth}
		\includegraphics[scale=0.14]{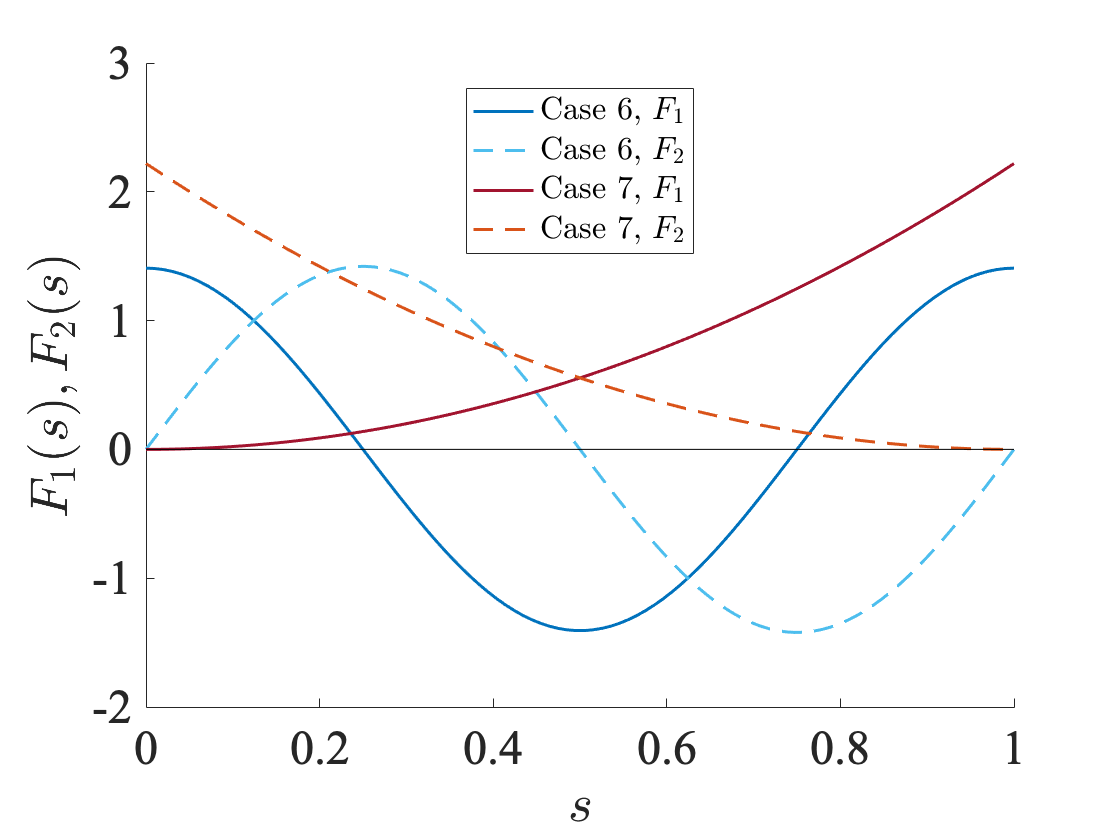}
		\caption{}
	\end{subfigure}
	\begin{subfigure}[b]{0.32\textwidth}
		\includegraphics[scale=0.14]{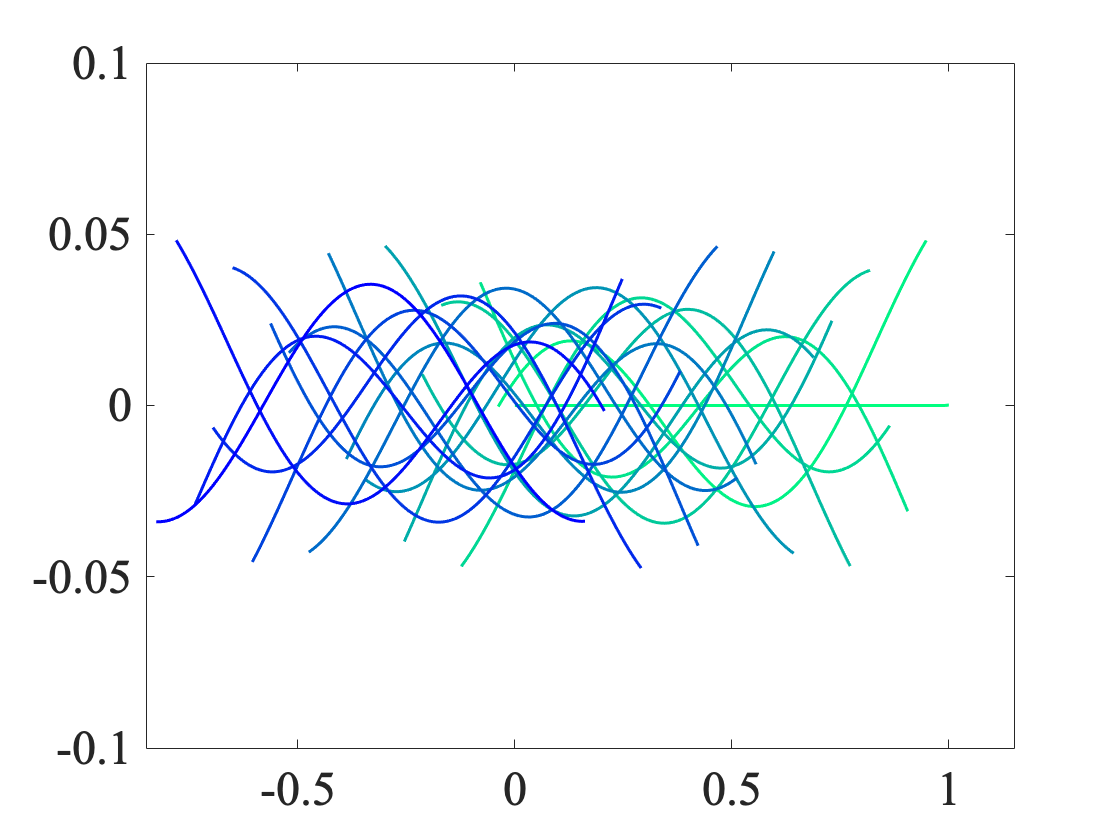}
		\caption{}
	\end{subfigure} 
	\begin{subfigure}[b]{0.32\textwidth}
		\includegraphics[scale=0.14]{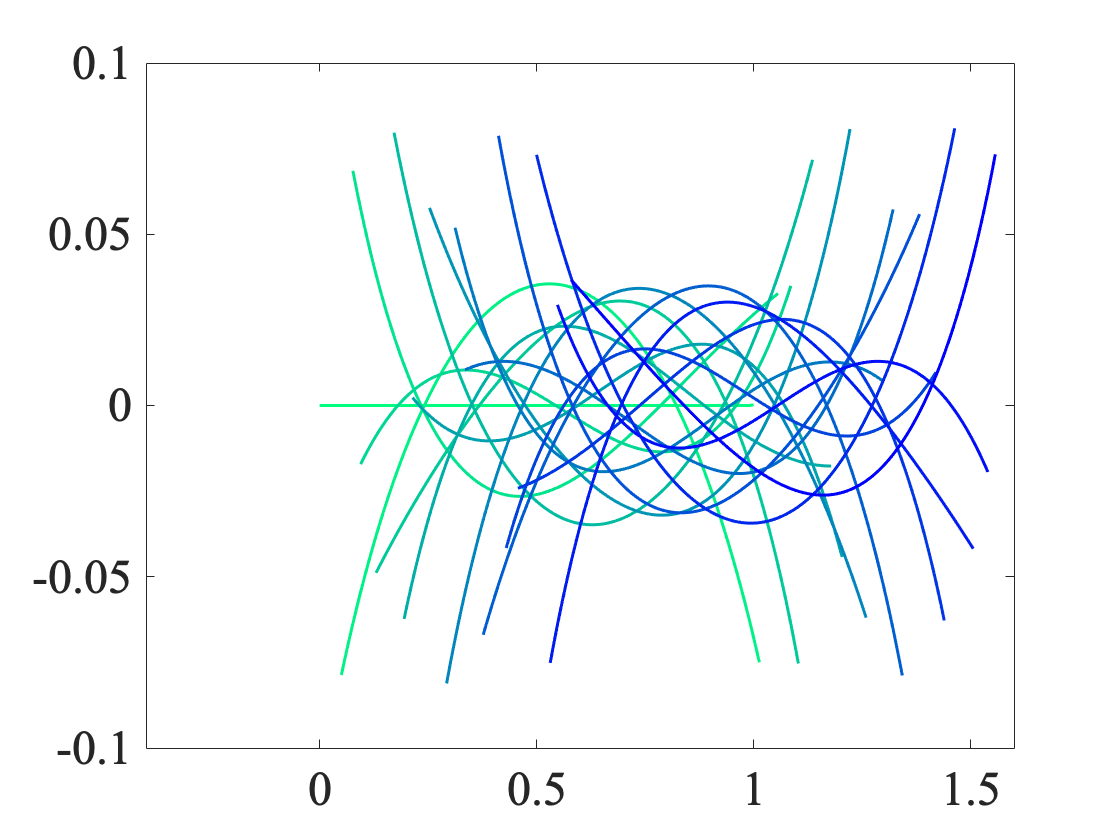}
		\caption{}
	\end{subfigure} \\
		\begin{subfigure}[b]{0.32\textwidth}
		\includegraphics[scale=0.14]{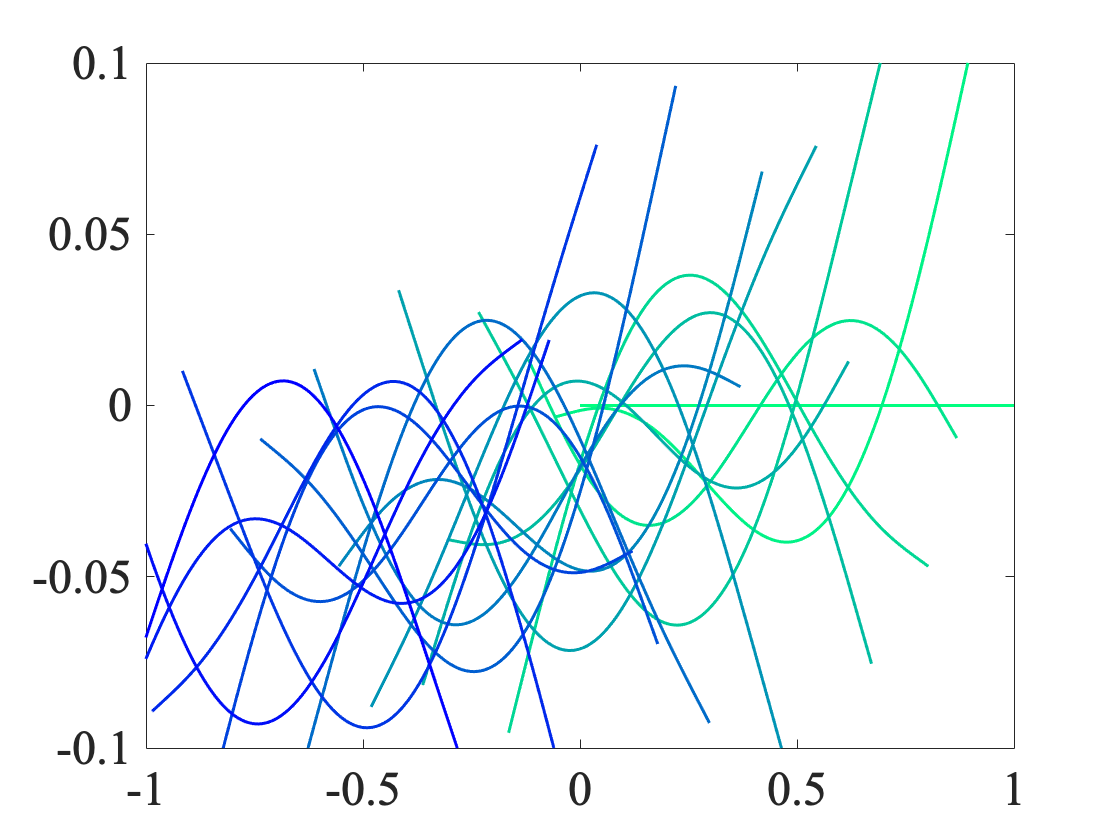}
		\caption{}
	\end{subfigure}
			\begin{subfigure}[b]{0.32\textwidth}
		\includegraphics[scale=0.14]{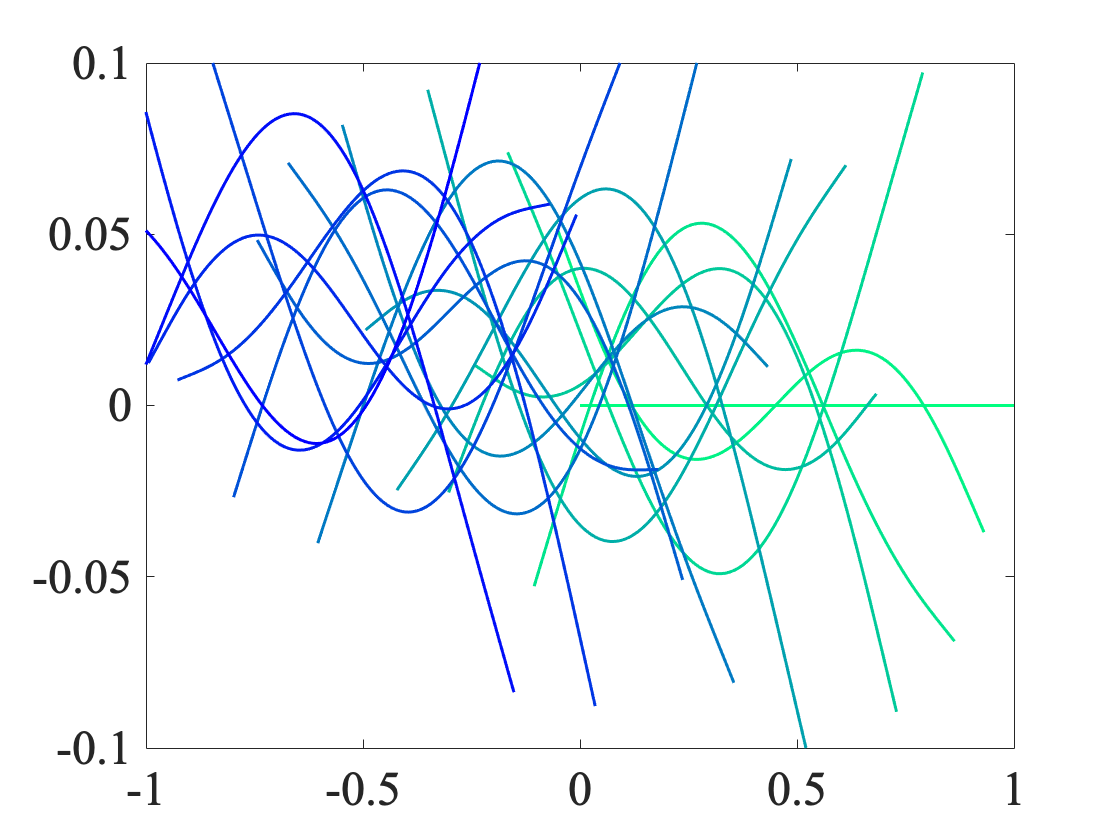}
		\caption{}
	\end{subfigure}
	\includegraphics[scale=0.2]{colorbar.png}
\caption{(A) Plot of $F_1$ and $F_2$ for cases (6) and (7) of good swimmers.
(B)--(E) Snapshots of the swimmer position in cases (6)--(9), respectively, over the course of 50 time units.}
\label{fig:goodswimmers}
\end{figure}

We may compare the swimming distance predicted by the expression \eqref{swim_speed1} to the observed displacement of the swimmer over a few periods. We numerically calculate the integral 
\begin{align*}
 \int_{t=5}^{t=10} U(t)\,dt = -\gamma \int_{t=5}^{t=10}\int_0^1 (\kappa_0)_s(\kappa-\kappa_0)\, ds \,dt
\end{align*}
over the course of 5 time units between $t=5$ and $t=10$. In Table \ref{tab:comparison}, we compare this integral expression to the actual horizontal displacement of the swimmer observed in numerical simulations, which we calculate as $x_0\big|_{t=10}-x_0\big|_{t=5}$. Note that in cases (8) and (9), the distance is far less than the $5\times(-0.1008)$ predicted by the linear theory. The prescribed $\kappa_0$ is not small, so nonlinear effects impact the swimming speeds seen here.  
\begin{table}
  \begin{center}
\def~{\hphantom{0}}
  \begin{tabular}{l|ccccc}
$F_1(s)$  &  Predicted displacement & Observed displacement \\[3pt]
\hline \\[-6pt]
Case 6 & -0.06033 & -0.06013 \\
Case 7 & 0.02643 & 0.02652 \\
Case 8 & -0.1201 & -0.1204 \\
Case 9 & -0.1226 & -0.1220 \\[3pt]
  \end{tabular}
  \caption{Predicted swimming distance over 5 time units (using expression \eqref{swim_speed1}) versus swimming distance $x_0\big|_{t=10}-x_0\big|_{t=5}$ observed in numerical simulations.  }
  \label{tab:comparison}
  \end{center}
\end{table} 

Finally, in one representative case each of non-swimmer, bad swimmer, and good swimmer, we plot the integrand $(\kappa_0)_s(\kappa-\kappa_0)$ of the swimming expression \eqref{swim_speed1} as a function of arclength $s$ over the course of one time period. Figure \ref{fig:intterm} displays snapshots of this integrand for forcing functions (1), (3), and (6), respectively. 

\begin{figure}[!ht]
\centering
	\begin{subfigure}[b]{0.32\textwidth}
		\includegraphics[scale=0.14]{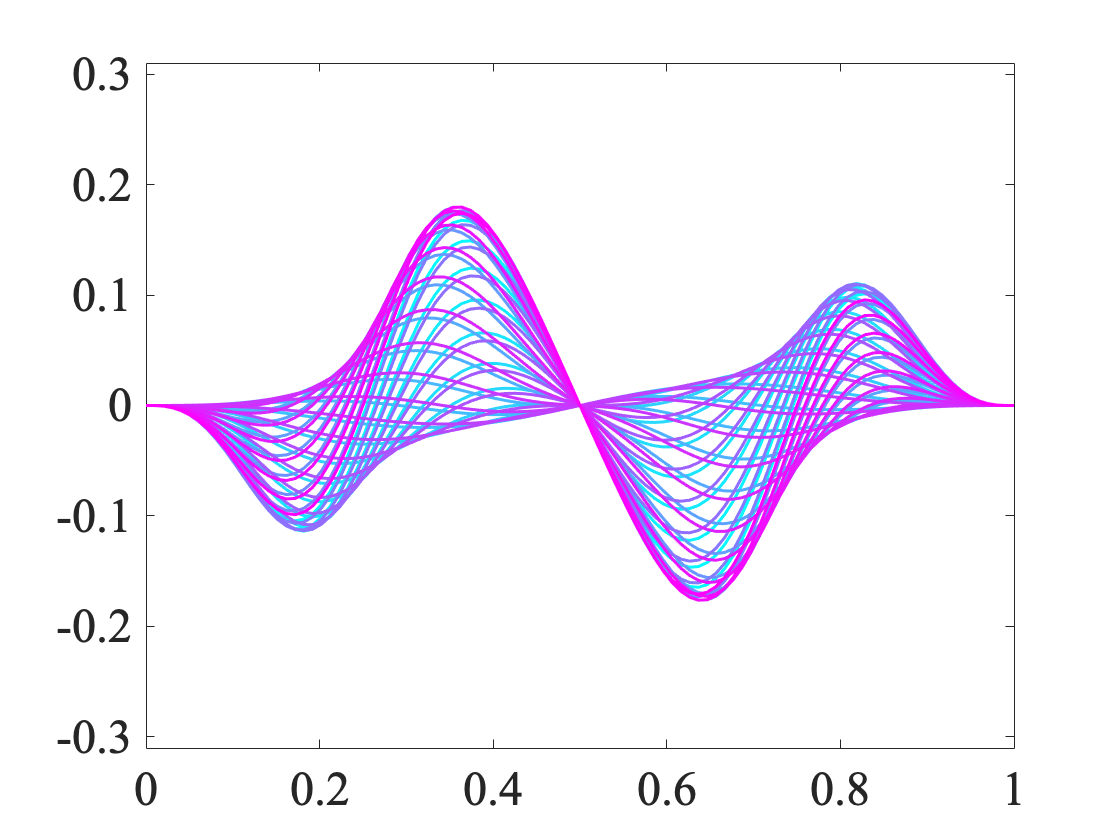}
		\caption{}
	\end{subfigure} 
	\begin{subfigure}[b]{0.32\textwidth}
		\includegraphics[scale=0.14]{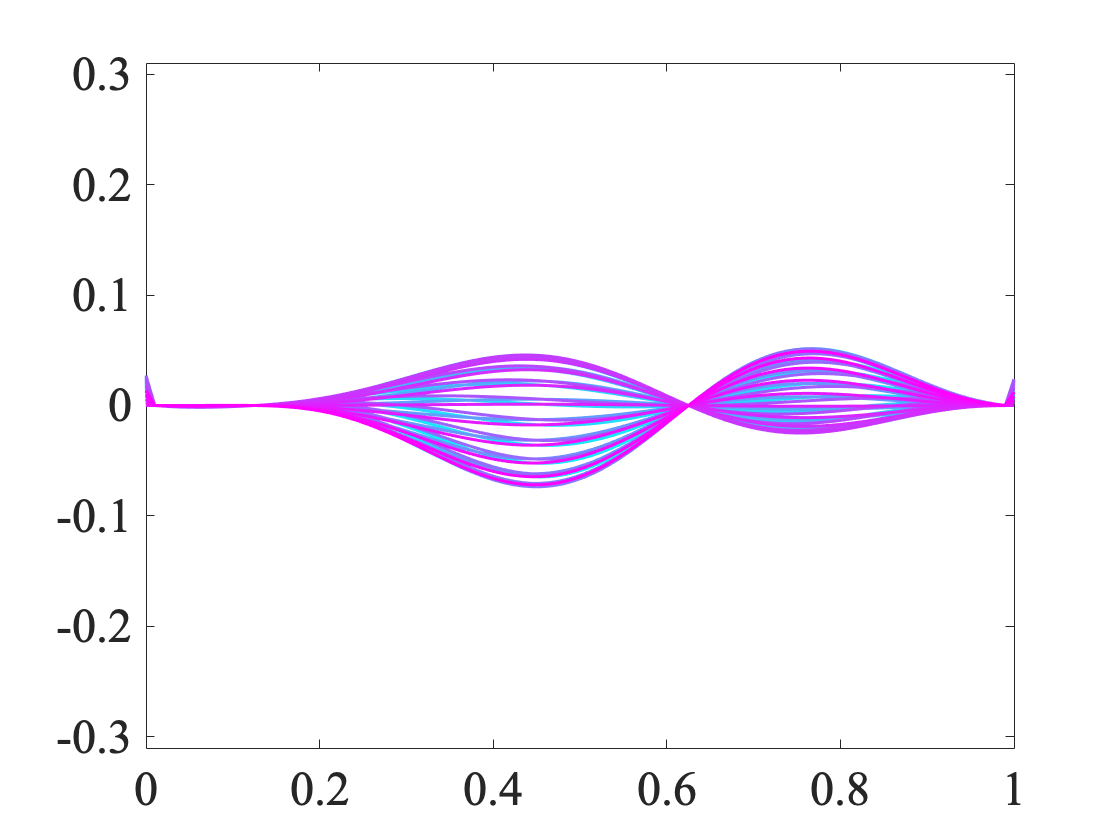}
		\caption{}
	\end{subfigure}
		\begin{subfigure}[b]{0.32\textwidth}
		\includegraphics[scale=0.14]{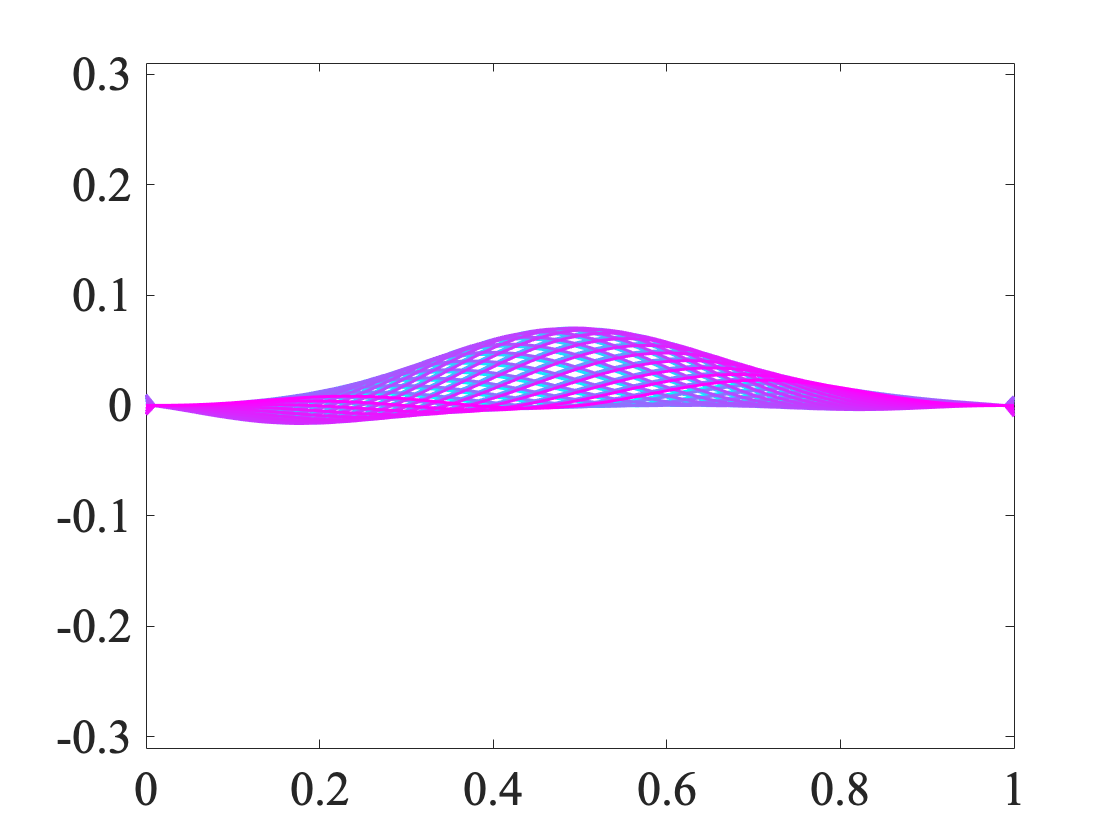}
		\caption{}
	\end{subfigure}
\includegraphics[scale=0.2]{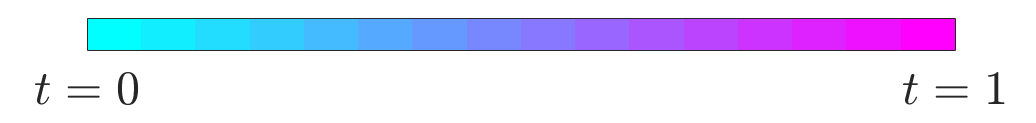}
\caption{Plots of the expression $(\kappa_0)_s(\kappa-\kappa_0)$ along the length of the filament over the course of one period in time for one example of no swimming, bad swimming, and good swimming. Shown are: (A) $F_1 = \cos(4\pi s)$, $F_2 = \cos(2\pi s)$ (forcing number 1), (B) $F_1 = \cos(2\pi s)+ \sin(2\pi s)$, $F_2=F_1$ (forcing number 3) and (C) $F_1 = \cos(2\pi s)$, $F_2 = \sin(2\pi s)$ (forcing number 6). }
\label{fig:intterm}
\end{figure}

In case (1), we note that the integrand $(\kappa_0)_s(\kappa-\kappa_0)$ is odd about $s=\frac{1}{2}$ at each time step and hence integrates to zero along the length of the filament, leading to no net motion due to spatial symmetry.
In case (3), we see that $(\kappa_0)_s(\kappa-\kappa_0)$ does not integrate to zero in $s$ but very nearly integrates to zero in time. The (near) lack of displacement here can be attributed to time symmetry. 
Finally, in case (6) the function $(\kappa_0)_s(\kappa-\kappa_0)$ clearly integrates to a positive number in space and time, leading to net motion (note that the displacement is negative according to the sign of \eqref{swim_speed1}).

\appendix
\section{Validation of numerical method}\label{app:num}
We validate the numerical method for the formulation described in Section \ref{sec:numerics} through comparison with a direct discretization of the classical inextensible fiber formulation \eqref{original_eqn}. 

The implementation of \eqref{original_eqn} requires description. We first rewrite the system \eqref{original_eqn} as follows:
\begin{equation}
\begin{aligned}\label{original_eqn_again}
\frac{\p\X}{\p t}(s,t) &= -\big({\bf I}+\gamma\X_s\X_s^{\rm T}\big)\bm{F}_s\,, \; \bm{F}=\bm{Q}_s-\wh{\tau}\X_s\,, \; 
\bm{Q}=\X_{ss}-\kappa_0\be_{\rm n}\,,\\
\abs{\X_s}^2&=1\,, \\
\bm{Q}\big|_{s=0,1}&=0\,, \quad \bm{F}\big|_{s=0,1}=0 \,.
\end{aligned}
\end{equation}
Our spatial discretization is as in Figure \ref{fig:disc}, and we denote the arclength coordinate by $s_i=i\triangle s, i=0,\dots,N$ where $\triangle s=1/N$.
Let us also define the midpoints $s_{i+1/2}=(i+1/2)\triangle s, \; i=0,\dots,N-1$. Let $\X_i$ be the discretization of $\X$ and $\wh{\tau}_{i+1/2}$
be the discretization of $\wh{\tau}$. As the notation suggests, $\X_i$ are seen as residing at $s_i$ whereas $\wh{\tau}_{i+1/2}$ reside 
at the midpoints $s_{i+1/2}$. Our $3N+2$ unknowns are $\X_i$ and $\wh{\tau}_{i+1/2}$. First, set:
\begin{equation}\label{Qdisc}
\begin{split}
\bm{Q}_i&=\frac{1}{(\triangle s)^2}\paren{\bm{X}_{i+1}-2\bm{X}_i+\bm{X}_{i-1}}-\kappa_{0,i}\bm{e}_{{\rm n},i},\quad i=1,\dots,N-1,\\
\bm{e}_{{\rm n},i}&=\mc{R}_{\pi/2}\bm{e}_{{\rm t},i}, \; \bm{e}_{{\rm t},i}=\frac{\X_{i+1}-\X_{i-1}}{\abs{\X_{i+1}-\X_{i-1}}},\; \mc{R}_{\pi/2}=\begin{pmatrix} 0 & -1 \\ 1 & 0\end{pmatrix},\quad i=1,\dots,N-1, \\
\bm{Q}_0&=\bm{Q}_N=0.
\end{split}
\end{equation}
Using the above expressions, define:
\begin{equation*}
\bm{F}_{i+1/2}=\frac{1}{\triangle s}\paren{\bm{Q}_{i+1}-\bm{Q}_i}-\wh{\tau}_{i+1/2}\frac{\bm{X}_{i+1}-\bm{X}_i}{\triangle s}, \; i=0,\dots,N-1.
\end{equation*}
We thus obtain the following semi-discretized system:
\begin{equation}
\begin{split}
\frac{d\X_i}{dt}&=-\frac{1}{\triangle s}\paren{{\bf I}+\gamma \bm{e}_{{\rm t},i}\bm{e}_{{\rm t},i}^{\rm T}}(\bm{F}_{i+1/2}-\bm{F}_{i-1/2}) \text{ for } i=1,\dots,N-1,\\
\bm{F}_{1/2}&=\bm{F}_{N-1/2}=0,\\
\frac{\abs{\X_{i+1}-\X_i}^2}{(\triangle s)^2}&=1, \; i=0,\dots, N-1.
\end{split}
\end{equation}
Notice that this is a differential algebraic system for $\X_i$ and $\wh{\tau}_{i+1/2}$, and there are exactly $3N+2$ equations matching 
the number of unknown functions. We discretize the above equation in time as follows. We let $\triangle t$ be the time step. Let $\X_i^n$
and $\wh{\tau}_{i+1/2}^n$ denote the values of $\X_i$ and $\wh{\tau}_{i+1/2}$ evaluated at time $n\triangle t, n=0,1,2,\dots\,$. 
Our time discretization is essentially a backward Euler method 
in which some parts of the equation are treated explicitly. Given $\X_i^{n-1}$, we solve the following system of equations to find $\X_i^n$ and $\wh{\tau}_{i+1/2}^n$:
\begin{equation}\label{method_b}
\begin{split}
\frac{\X_i^n-\X_i^{n-1}}{\triangle s}&=-\frac{1}{\triangle s}\paren{{\bf I}+\gamma \bm{e}_{{\rm t},i}^{n-1}\bm{e}_{{\rm t},i}^{n-1, \rm T}}(\bm{F}_{i+1/2}^n-\bm{F}_{i-1/2}^n)\,, \; i=1,\dots N-1\,,\\
\bm{F}_{1/2}^n&=\bm{F}_{N-1/2}^n=0\,,\\
\bm{F}_{i+1/2}^n&=\frac{1}{\triangle s}\paren{\bm{Q}_{i+1}^n-\bm{Q}_i^n}-\wh{\tau}_{i+1/2}^n\frac{\bm{X}_{i+1}^n-\bm{X}_i^n}{\triangle s}\,, \; i=0,\dots,N-1\,,\\
\bm{Q}_i^n&=\frac{1}{(\triangle s)^2}\paren{\bm{X}^n_{i+1}-2\bm{X}_i^n+\bm{X}_{i-1}^n}-\kappa_{0,i}^n\bm{e}_{{\rm n},i}^{n-1}\,,\quad i=1,\dots,N-1\,,\\
\bm{Q}_0^n&=\bm{Q}_N^n=0\,,\\
\frac{\abs{\X_{i+1}^n-\X_i^n}^2}{(\triangle s)^2}&=1, \; i=0,\dots, N-1\,.
\end{split}
\end{equation}
In the above, $\bm{e}_{\rm t,i}^{n-1}$ and $\bm{e}_{{\rm n},i}^{n-1}$ are computed with $\X_i^{n-1}$ using the second line of equation \eqref{Qdisc}.
Note also that $\kappa_0$ and hence $\kappa_{0,i}^n$ are known quantities.
The above constitute $3N+2$ equations in the unknowns $\X_i^n$ and $\wh{\tau}_{i+1/2}^n$. These equations are solved using Newton's method. \\

Hereafter we will refer to the method \eqref{our_disc0}--\eqref{our_disc2} as Method (a), and the method \eqref{method_b} as Method (b), and denote the corresponding fiber configurations by $\X_a(s,t)$ and $\X_b(s,t)$, respectively. 

We begin by comparing the methods in the case of zero preferred curvature $\kappa_0\equiv0$. The initial filament configuration is a semicircle, and the fiber rapidly relaxes into a straight line. In Figure \ref{fig:comparison1} we evolve the fiber until time $t=0.004$ and plot the norm of the difference between Method (a) and Method (b) as the number of fiber segments $N$ is increased. We plot the $L^\infty$ and $L^2$ norms of the difference $\norm{\X_a-\X_b}$, which both display $O(N)$ convergence.
In Figure \ref{fig:comparison2} we plot the actual fiber configurations $\X_a(s,t)$ (blue) and $\X_b(s,t)$ (red) at different times $t$ for $N=100$. 

\begin{figure}[!ht]
\centering
	\begin{subfigure}[b]{0.45\textwidth}
		\includegraphics[scale=0.18]{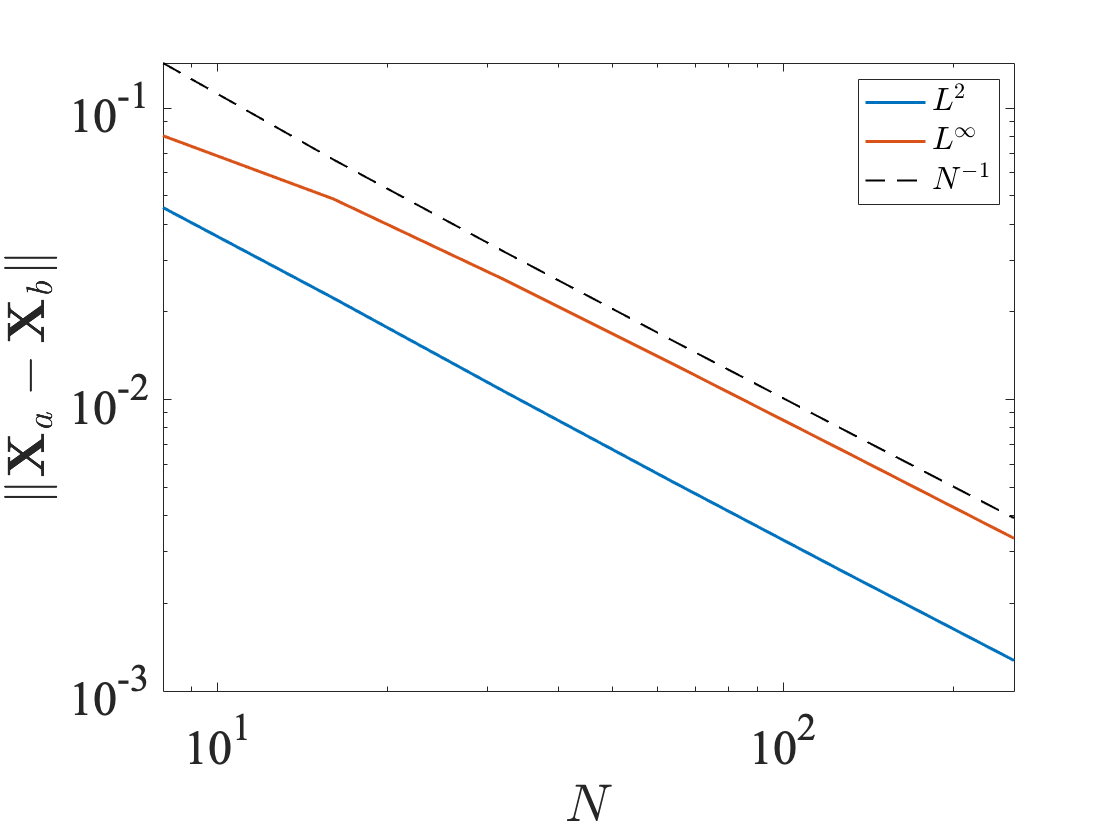}
		\caption{}
		\label{fig:comparison1}
	\end{subfigure}
	\begin{subfigure}[b]{0.45\textwidth}
		\includegraphics[scale=0.18]{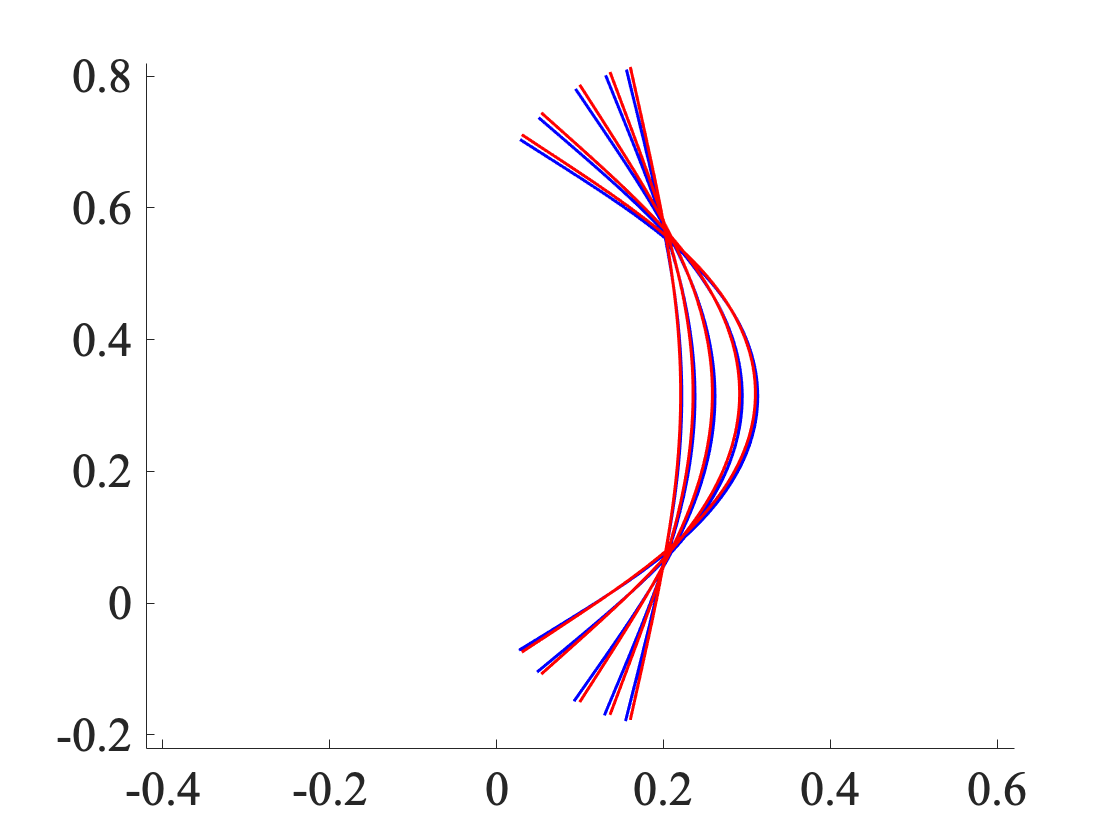}
		\caption{}
		\label{fig:comparison2}
	\end{subfigure}
\label{fig:comparison}
\caption{ (A) Log-log plot of the $L^\infty$ (red) and $L^2$ (blue) norm of the difference in fiber positions $\norm{\X_a-\X_b}$ at time $t=0.004$ as the discretization $N$ is increased. The methods display $O(N)$ convergence to each other.
(B) Comparison of fiber positions $\X_a(s,t)$ (blue) and $\X_b(s,t)$ (red) at $t=0.0005$, 0.001, 0.002, 0.003, and 0.004. Both fibers are discretized using $N=100$ segments.}
\end{figure}

We next compare Methods (a) and (b) using the traveling wave forcing $\kappa_0(s,t)=\sin(2\pi (s-t))$. The initial filament configuration is a straight line along the $x$-axis. In Figure \ref{fig:comparison3} we evolve the fiber until time $t=0.17$ and plot $\norm{\X_a-\X_b}_{L^\infty(I)}$ and $\norm{\X_a-\X_b}_{L^2(I)}$ as the number of fiber segments $N$ is increased. As in the $\kappa_0\equiv0$ setting, we observe $O(N)$ convergence between the methods.
In Figure \ref{fig:comparison4}, we again plot the actual positions $\X_a(s,t)$ (blue) and $\X_b(s,t)$ (red) for different snapshots in time. 

\begin{figure}[!ht]
\centering
	\begin{subfigure}[b]{0.45\textwidth}
		\includegraphics[scale=0.18]{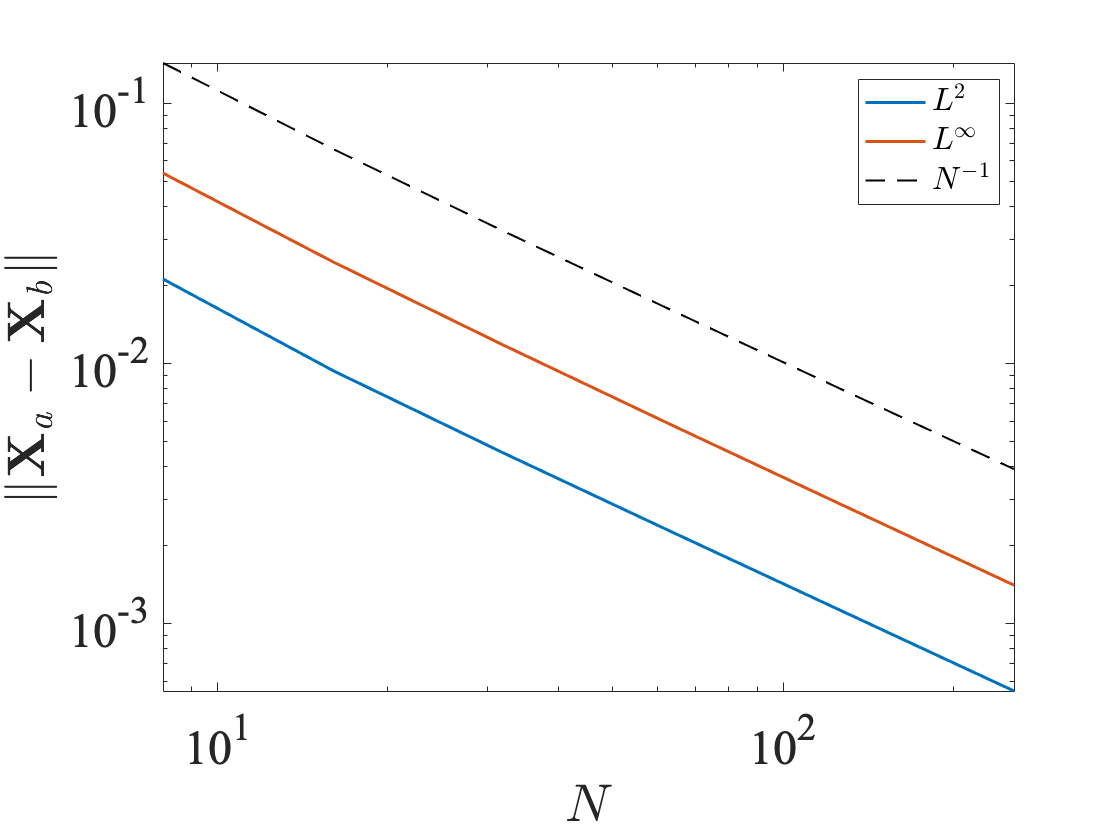}
		\caption{}
		\label{fig:comparison3}
	\end{subfigure}
	\begin{subfigure}[b]{0.45\textwidth}
		\includegraphics[scale=0.18]{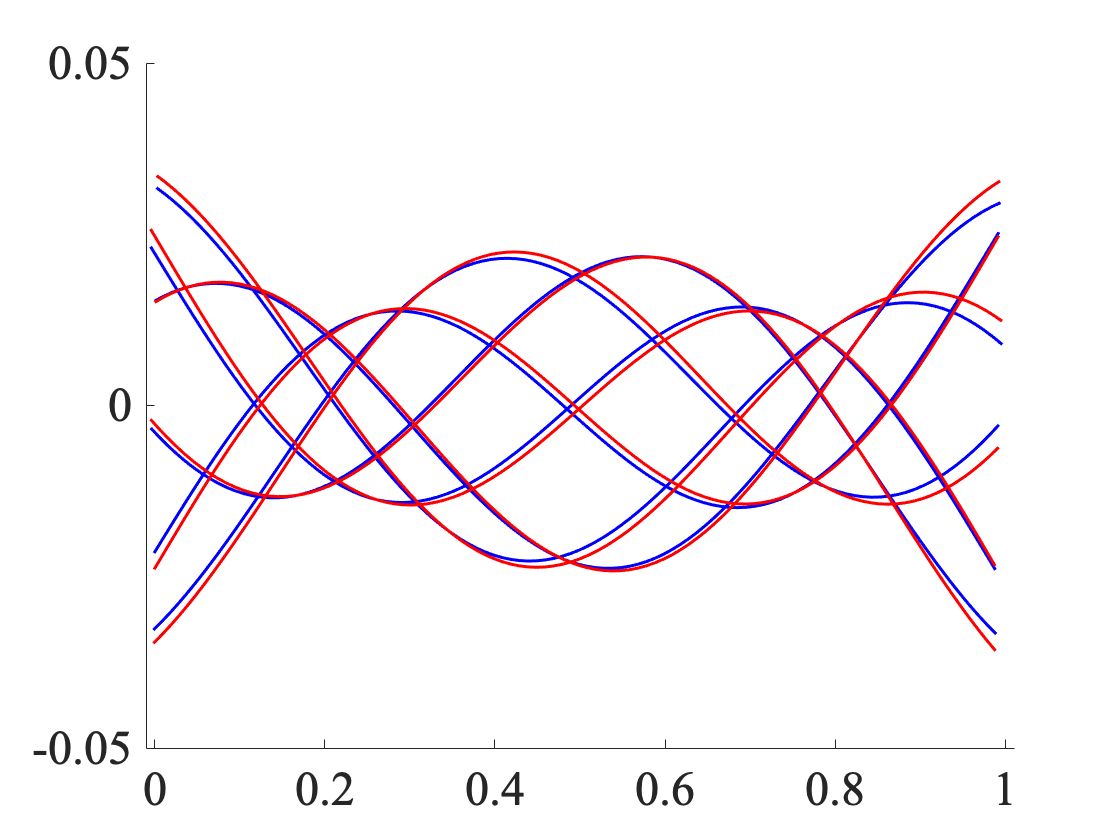}
		\caption{}
		\label{fig:comparison4}
	\end{subfigure}
\label{fig:comparison0}
\caption{ A. Log-log plot of the $L^\infty$ (red) and $L^2$ (blue) norm of the difference in fiber positions $\norm{\X_a-\X_b}$ at time $t=0.17$ as the discretization $N$ is increased. The methods display $O(N)$ convergence to each other.
B. Comparison of fiber positions $\X_a(s,t)$ (blue) and $\X_b(s,t)$ (red) at $t=0.17$, 0.30, 0.49, 0.64, 0.86, and 1.00. Both fibers are discretized using $N=100$ segments.}
\end{figure}

\vspace{0.5cm}
{\bf Acknowledgments.} \quad
Y.M. acknowledges support from the Math+X grant from the Simons Foundation and NSF DMS-2042144.
L.O. acknowledges support from NSF Postdoctoral Fellowship DMS-2001959 and thanks Dallas Albritton for helpful discussion regarding the regularity theory of this problem.


\bibliographystyle{abbrv} 
\bibliography{RFTbib}


\end{document}